\documentclass[11pt]{article}

\usepackage[margin=1.0in]{geometry}
\usepackage[utf8]{inputenc}            %
\usepackage[T1]{fontenc}               %
\usepackage[hidelinks]{hyperref}    	   %
\hypersetup{
	colorlinks,
	linkcolor={magenta!50!black},
	citecolor={magenta!50!black},
	urlcolor={magenta!50!black}
}
\usepackage[mathscr]{euscript}
\usepackage[T1]{fontenc}

\usepackage{url}                        %
\usepackage{booktabs}                   %
\usepackage{amsfonts}                   %
\usepackage{nicefrac}                   %
\usepackage{bbm}
\usepackage{microtype}                  %
\usepackage{algorithm}
\usepackage{algpseudocode}
\usepackage{xcolor}                     %
\usepackage{lmodern}
\usepackage{enumitem}
\usepackage[numbers, sort]{natbib}
\usepackage{amsmath, amssymb, amsfonts}
\usepackage{amsthm, thmtools, thm-restate}
\usepackage{mathtools, nccmath, extarrows}
\usepackage{enumitem}
\usepackage{subcaption}
\usepackage{bm}
\usepackage{scalerel}
\usepackage[capitalise]{cleveref}
\usepackage{array}

\newcounter{imagerow}[figure]
\newcounter{imagecolumn}[imagerow]

\newcolumntype{I}{>{\refstepcounter{imagecolumn}}{c}<{}}

\newlength{\imagewidth}

\usepackage{stackengine}
\stackMath

\DeclareFontEncoding{LS1}{}{}
\DeclareFontSubstitution{LS1}{stix}{m}{n}

\makeatletter
  \newcommand*\textmathversion{\csname textmv@\math@version\endcsname}
  \newcommand*\textmv@normal{m}
  \newcommand*\textmv@bold{b}
\makeatother

\newtheorem{theorem}{Theorem}[section]
\newtheorem{proposition}[theorem]{Proposition}
\newtheorem{lemma}[theorem]{Lemma}

\newtheorem{claim}[theorem]{Claim}

\theoremstyle{definition}
\newtheorem{definition}[theorem]{Definition}
\newtheorem{assumption}{Assumption}

\newtheorem{fact}[theorem]{Fact}

\usepackage{bm}

\newcommand{\Eb}{\mathbf{E}}

\newcommand{\bA}{\bm{A}}

\newcommand{\bE}{\bm{E}}
\newcommand{\bF}{\bm{F}}
\newcommand{\bG}{\bm{G}}
\newcommand{\bH}{\bm{H}}
\newcommand{\bI}{\bm{I}}

\newcommand{\bM}{\bm{M}}

\newcommand{\bU}{\bm{U}}
\newcommand{\bV}{\bm{V}}
\newcommand{\bW}{\bm{W}}
\newcommand{\bX}{\bm{X}}
\newcommand{\bY}{\bm{Y}}

\newcommand{\bu}{\bm{u}}
\newcommand{\bv}{\bm{v}}

\newcommand{\cP}{\mathcal{P}}

\newcommand{\cT}{{\mathcal{T}}}
\newcommand{\cU}{\mathcal{U}}
\newcommand{\cV}{\mathcal{V}}
\newcommand{\cW}{\mathcal{W}}

\newcommand{\bLambda}{\bm{\Lambda}}

\newcommand{\bSigma}{\bm{\Sigma}}

\newcommand{\R}{\mathbb{R}}

\newcommand{\RR}{{\mathbb R}}
\newcommand{\NN}{{\mathbb N}}
\newcommand{\SSS}{{\mathbb S}}

\newcommand{\rank}{\operatorname{rank}}
\newcommand{\Tr}{\operatorname{tr}}

\newcommand{\EE}{\operatorname{\mathbb E}}

\newcommand{\PP}{\operatorname{\mathbb P}}

\newcommand{\supp}{\operatorname{supp}}

\DeclarePairedDelimiter{\dotp}{\langle}{\rangle}

\def\shortdisplay{\setlength{\abovedisplayskip}{5pt}%
	\setlength{\belowdisplayskip}{5pt}%
	\setlength{\abovedisplayshortskip}{2pt}%
	\setlength{\belowdisplayshortskip}{2pt}}

\let\oldselectfont\selectfont
\def\selectfont{\oldselectfont\shortdisplay}

\newcommand{\FDR}{\operatorname{FDR}}

\newcommand{\FDRt}{\widetilde{\FDR}}

\newcommand{\rh}{\widehat r}

\newcommand{\Gh}{\widehat{G}}
\newcommand{\Gdh}{\widehat{G}'}
\newcommand{\as}{\xrightarrow{a.s.}}

\begin{document}

\title{
Controlling the False Discovery Rate in Subspace Selection
}
\author{
Mateo D\'\i az\thanks{Department of Applied Mathematics and Statistics, Johns Hopkins University, Baltimore, MD 21218, USA.} \and
    Venkat Chandrasekaran\thanks{Departments of Computing and Mathematical Sciences and of Electrical Engineering, California Institute of Technology, Pasadena, CA 91125, USA.} 
    }
\date{}

\maketitle

\begin{abstract}

Controlling the false discovery rate (FDR) is a popular approach to multiple testing, variable selection, and related problems of simultaneous inference.  In many contemporary applications, models are not specified by discrete variables, which necessitates a broadening of the scope of the FDR control paradigm.  Motivated by the ubiquity of low-rank models for high-dimensional matrices, we present methods for subspace selection in principal components analysis that provide control on a geometric analog of FDR that is adapted to subspace selection.  Our methods crucially rely on recently-developed tools from random matrix theory, in particular on a characterization of the limiting behavior of eigenvectors and the gaps between successive eigenvalues of large random matrices.  Our procedure is parameter-free, and we show that it provides FDR control in subspace selection for common noise models considered in the literature. 
 We demonstrate the utility of our algorithm with numerical experiments on synthetic data and on problems arising in single-cell RNA sequencing and hyperspectral imaging.
\end{abstract}


\section{Introduction}\label{sec:intro}

Controlling the false discovery rate (FDR) is a prominent approach to simultaneous testing that was pioneered by Benjamini and Hochberg \cite{benjamini1995controlling}.  They defined the false discovery rate associated with a multiple testing procedure $\hat{R}$ as
\begin{equation}
    \label{eq:classic-fdr}
    \FDR(\widehat R) = \EE \left[\frac{|\widehat R \cap R^c|}{\max\{|\widehat R|, 1\}}\right],
\end{equation}
where $R$ denotes the set of true positives and $\widehat R$ denotes the set of rejected nulls.  Furthermore, they also introduced an innovative methodology for FDR control \cite{benjamini1995controlling}, which is less conservative and yields more discoveries than classic methodologies that control familywise error.  As a result, the paradigm of FDR control has played a prominent role in the practice of simultaneous testing and related problems such as variable selection, and these have had a profound impact in numerous scientific fields in recent decades.  Due to the evolving nature of modern applications, models for high-dimensional data in many contemporary problem domains are not always specified by discrete variables.  Selecting models that exhibit many discoveries and small false positive error in such contexts requires a broadening of the scope of the FDR control paradigm.

Low-dimensional subspaces offer perhaps the simplest and most popular model for low-complexity geometric structure in high-dimensional data.  As a result, subspace estimation is a fundamental task in a wide array of applications \cite{ vidal2005generalized, baumgartner2004subspace, charisopoulos2021low, wang2015subspace, cape2019two, cai2021subspace, charisopoulos2021composite, diaz2019nonsmooth}.  In this paper, we consider the problem of subspace estimation from a perspective motivated by multiple testing.  Specifically, given noisy data about a subspace of interest $\cU \subset \R^n$, we wish to obtain an estimate $\widehat \cU \subset \R^n$ such that the following quantity is controlled at a desired level:
\begin{equation} \label{eq:fdr}
    \FDR(\widehat \cU) = \EE \left[\frac{\Tr\left(P_{\widehat \cU} P_{\cU^{\perp}}\right)}{\max\{\dim (\widehat \cU), 1\}}\right].
\end{equation}
Here $\cU^{\perp}$ is the orthogonal complement of $\cU$, $P_{\widehat \cU}$ and $P_{\cU^\perp}$ denote projections onto the subspaces $\widehat \cU$ and $\cU^\perp$, and $\dim(\widehat \cU)$ is the dimension of the subspace $\widehat \cU$.  The numerator inside the expectation is the sum of the squares of the cosines of the principal angles between $\widehat \cU$ and $\cU^\perp$, and it evaluates the extent to which the estimate $\widehat \cU$ is not aligned with $\cU$.  The denominator normalizes by the total amount of discovery contained in the estimate $\widehat \cU$. 
 Thus, the quantity $\FDR(\widehat \cU)$ is a geometric analog of the original FDR notion \eqref{eq:classic-fdr}; in particular, when the subspaces $\cU$ and $\widehat \cU$ are axis-aligned \eqref{eq:fdr} reduces to \eqref{eq:classic-fdr}.

The quantity \eqref{eq:fdr} as well as its unnormalized analog $\mathrm{FD}(\widehat \cU) = \EE \left[\Tr\left(P_{\widehat \cU} P_{\cU^{\perp}}\right)\right]$ were proposed as false positive error control measures for subspace estimation in a recent paper \cite{taeb2020false}.  Thus, while subspace estimation has received considerable attention over several decades owing to the ubiquity of low-rank matrix models in applications, a viewpoint grounded in false positive error control is not common in this vast literature.  In particular, we are not aware of any prior work that controls FDR in subspace estimation; the methodological focus of the recent effort \cite{taeb2020false} was only on controlling the unnormalized variant FD (see Section~\ref{sec:related} for more details).

Concretely, we investigate the model problem of estimating the column space of a rank-$r$ matrix $\bA$ corrupted by additive noise $\bE$:
\begin{equation} \label{eq:spike}
\bX = \bA + \bE.
\end{equation}
We consider two cases of this model, one in which $\bA \in \SSS^n_+$ is positive semidefinite and $\bE \in \SSS^n$ is symmetric, and a second in which both $\bA, \bE \in \R^{n \times m}$ are general nonsquare matrices.  For simplicity, our exposition and development will focus on the symmetric case, with extensions to the nonsquare case studied in Section~\ref{sec:asymmetric}.  The model \eqref{eq:spike} was first introduced by Johnstone \cite{johnstone2001distribution} and has been widely studied in the literature \cite{bai2012sample, cai2021subspace, abbe2022l, agterberg2022entrywise, cape2019two}, although these results do not directly provide false positive error control (see Section~\ref{sec:related} for more details).  In contrast our objective in this paper is the following -- given $\bX$ and $\alpha \in (0,1)$, we wish to obtain an estimate $\widehat \cU \subset \R^n$ for the column space of $\bA$ such that $\widehat \cU$ has as large a dimension as possible while ensuring that $\FDR(\widehat \cU) \leq \alpha$.  We consider the estimator $\widehat \cU = \widehat \cU_{k}$ corresponding to the span of the $k$ leading eigenvectors of $\bX$, which is the most natural and widely used in practice.  Thus, our goal is to identify as large a threshold $k$ as possible such that $\FDR(\widehat \cU_{k}) \leq \alpha$.

As false positive error control for subspace estimation is an objective that has only been articulated recently, it is natural to ask whether there are any benefits to controlling $\FDR$ compared to the unnormalized $\mathrm{FD}$ in the model problem \eqref{eq:spike}.  The next numerical illustration demonstrates that the virtues of controlling $\FDR$ over $\mathrm{FD}$ are indeed analogous to what one observes with variable selection.   Specifically, Figure~\ref{fig:FDvsFDR} provides plots of $\mathrm{FD}(\widehat \cU_{k})$ and $\FDR(\widehat \cU_{k})$ against $k$ for two problem instances of size $n = 1000$ and true rank $r = 20, 40$, with the noise $\bE$ being drawn from the Gaussian Orthogonal Ensemble (GOE).  Comparing the plots of $\FDR$ and the unnormalized $\mathrm{FD}$, we observe that controlling $\FDR$ in subspace estimation offers the prospect of more powerful selection procedures that are also robust to the (unknown) dimension of the true subspace, much like with variable selection.

\begin{figure}[t]
    \centering
    \begin{subfigure}{0.45\textwidth}
        \centering False Discovery
        \vspace{0.2em}
        \includegraphics[width = \textwidth]{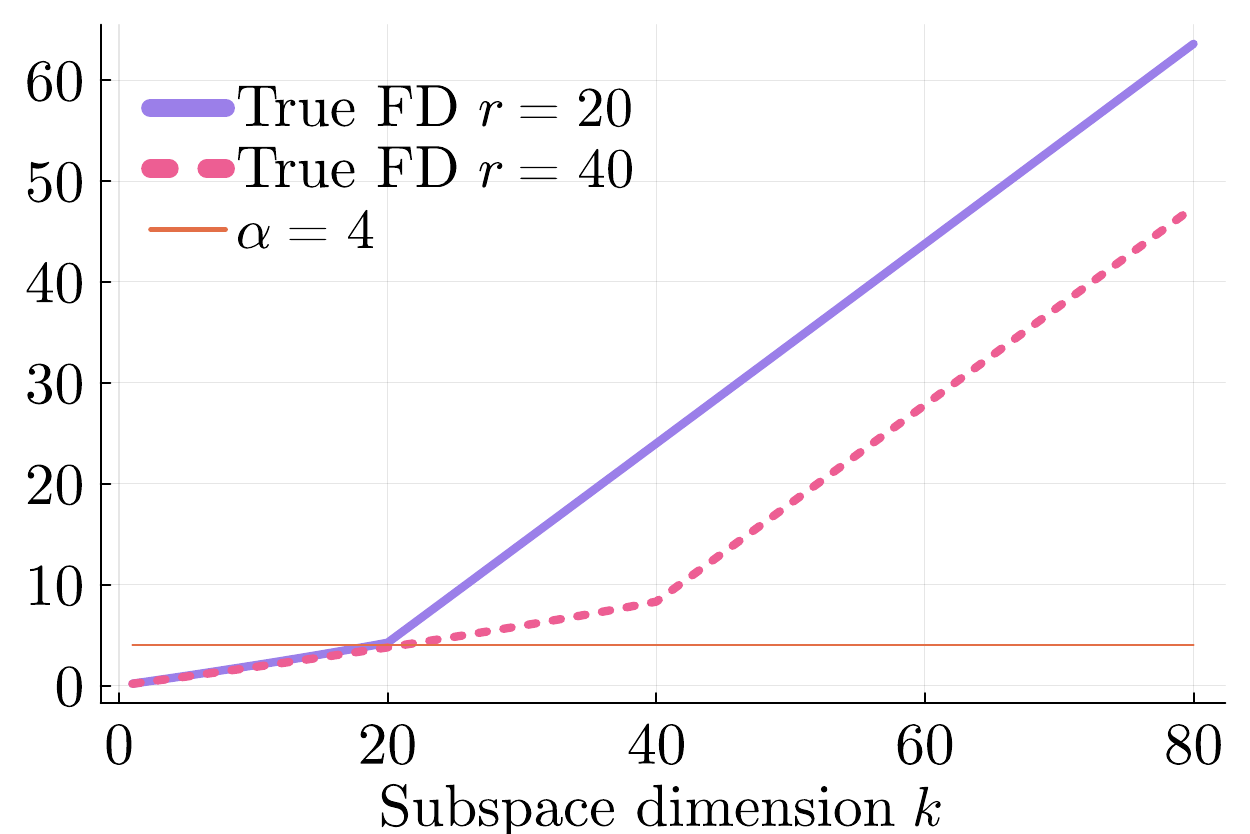}
    \end{subfigure}
    \begin{subfigure}{0.45\textwidth}
        \centering
        False Discovery Rate
        \vspace{0.2em}
        \includegraphics[width = \textwidth]{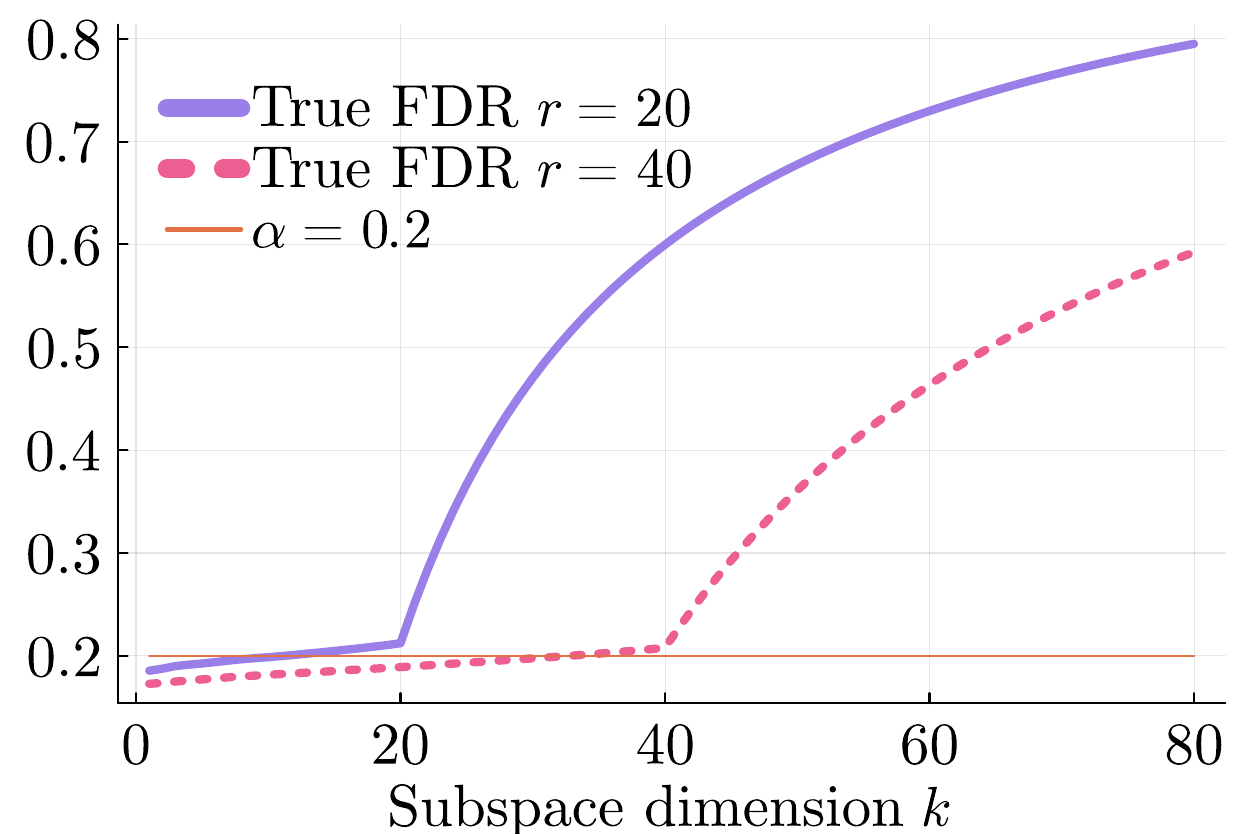}
    \end{subfigure}
    \caption{$\mathrm{FD}(\widehat \cU_{k})$ as a function of subspace dimension $k$ on the left and $\FDR(\widehat \cU_{k})$ as a function of subspace dimension $k$ on the right.  We consider two problem instances of size $n = 1000$ and true rank $r = \{20, 40\}$.  The noise $\bE$ is drawn from a GOE.  The horizontal red lines correspond to $\alpha = 4$ for FD and $\alpha =0.2$ for FDR. }
    \label{fig:FDvsFDR}
\end{figure}

\subsection{Our Contributions}
The core difficulty with our approach is that $\FDR(\widehat \cU_{k})$ cannot be computed directly from $\bX$, and therefore it is unclear how to select a suitable threshold $k$.  The point of departure for our development is to address this challenge by considering an asymptotic characterization of $\FDR(\widehat \cU_{k})$.  Specifically, consider a setting in which we have a sequence of growing-sized problem instances $\{\bX_n = \bA_n + \bE_n ~|~ \bA_n \in \SSS^n_+, ~ \bE_n \in \SSS^n\}$, and let $\widehat \cU^{(n)}_k$ denote the preceding $k$-dimensional principal subspace estimator on a problem instance of size $n$.  If the rank of $\bA_n$ is bounded above by a fixed $r \in \NN$ for all $n$, the noise $\bE_n$ has a compactly supported limiting spectral measure $\mu_{\bE}$ as $n \rightarrow \infty$,\footnote{The existence of a limiting spectral measure might sound too constraining, but several natural distributions satisfy this assumption; e.g., if $\bG_n$ is a standard Gaussian matrix, then the spectral distribution of the Gaussian Orthogonal Ensemble $\bE_n = \frac{1}{\sqrt{2n}} (\bG_n+\bG_n^\top)$ converges to the semicircle law.} and the spectrum of $\bA_n$ is `well-separated' from that of $\bE_n$, then we have the following asymptotic expression for $\FDR(\widehat \cU^{(n)}_k)$:
\begin{equation}\label{eq:key-limit}
    \lim_{n \rightarrow \infty} \FDR(\widehat \cU^{(n)}_k) = \lim_{n \rightarrow \infty} 1 + \frac{1}{k} \sum_{i=1}^{\min\{k, r\}} \frac{G_{\mu_{\bE}}\left(\lambda_i(\bX_n)\right)^2}{G'_{\mu_{\bE}}\left(\lambda_i(\bX_n)\right)}
\end{equation}
where the quantity $G_{\mu}$ is the \emph{Cauchy transform} of a compactly supported measure $\mu$ on the real line:
\begin{equation}
\label{eq:cauchy}
G_{\mu}(z) =\int \frac{1}{z-t} d\mu(t)\qquad \text{for all }z \notin \supp(\mu).
\end{equation}
Intuitively, the eigenvalues of $\bX_n$ associated with the noise spectrum $\lambda(\bE_n)$ form a bulk, and by `well-separated' we mean that the eigenvalues associated with the signal $\bA_n$ do not `mix' into the bulk. We formally quantify this separation in Section~\ref{sec:algo}, and we adapt this formula to more general settings in which the spectra of $\bA_n$ and of $\bE_n$ become `mixed.'  The expression \eqref{eq:key-limit} is derived based on the results from \cite{benaych2011eigenvalues, benaych2012singular} on the spectra of fixed-rank perturbations of large random matrices; see Section~\ref{sec:related} for a discussion of prior approaches to low-rank estimation that employ a similar approach.

Returning to the case of a fixed $n$ (and suppressing the explicit subscript), the upshot of the expression \eqref{eq:key-limit} is that the formula $1 + \frac{1}{k} \sum_{i=1}^{\min(k, r)} \frac{G_{\mu_{\bE}}\left(\lambda_i(\bX)\right)^2}{G'_{\mu_{\bE}}\left(\lambda_i(\bX)\right)}$ on the right-hand-side is more amenable to computation and estimation from $\bX$, with the potential to obtain accurate estimates for large problem sizes.  However, there still remain two difficulties corresponding to a lack of information about the limiting spectral measure $\mu_{\bE}$ and a lack of knowledge about the true rank $r$.  We describe how to address these challenges next, which yields an algorithm for controlling FDR in subspace estimation.

Assuming first that an estimate $\widehat r$ of the rank $r$ is available, the Cauchy transform $G_{\mu_{\bE}}$ and its derivative $G'_{\mu_{\bE}}$ may be estimated empirically via $\widehat G(y) = \frac{1}{n - \rh}\sum_{j = \rh + 1}^{n} \frac{1}{y - \lambda_{j}(\bX)}$ and $\widehat G'(y) = -  \frac{1}{n - \rh}\sum_{j = \rh +1}^{n} \frac{1}{(y-\lambda_{j}(\bX))^{2}}$, where the eigenvalues $\lambda(\bX)$ are ordered as $\lambda_1(\bX) \geq \cdots \geq \lambda_n(\bX)$.  With this estimate, we summarize our method for controlling FDR in subspace estimation in Algorithm~\ref{alg:FDR}; here, Steps~2 and 3 employ the above empirical estimates.

\begin{algorithm}[h]
	\caption{FDR control for eigensubspaces\hfill} %
	\label{alg:FDR}
	
	{\bf Input}: Observations $\bX$, level $\alpha$, and (observable) rank estimator $\rh$, e.g., via \cref{alg:rank-estimate}.\\
	{\bf Output}: Estimate $\widehat k$ of number of top components to select. \vspace{.2cm} \\ 
	{\bf Step 1} Obtain an eigenvalue factorization $\bX = \widehat \bU \bLambda \widehat \bU^{\top}$ with eigenvalues $\lambda_{1} \geq \dots \geq \lambda_{n}$.\\
	{\bf Step 2} Define estimates of the Cauchy transform and its derivative
	\begin{equation*}
	\widehat G(y) = \frac{1}{n - \rh}\sum_{j = \rh + 1}^{n} \frac{1}{y - \lambda_{j}} \quad \text{and} \quad \widehat G'(y) = -  \frac{1}{n - \rh}\sum_{j = \rh +1}^{n} \frac{1}{(y-\lambda_{j})^{2}}.
	\end{equation*}
	{\bf Step 3} Let $\widehat \cU_{k}$ be the span of the first $k$ columns of $\bU$ and estimate the FDR for different $k$:
	\begin{equation} \label{eq:fdr-estimate}
	\widehat{\FDR}(\widehat \cU_{k}) = 1 + \frac{1}{k}\sum_{i=1}^{\min\{k,\rh\}} \frac{\Gh(\lambda_{i})^{2}}{\Gdh(\lambda_{i})} %
	\end{equation}
	{\bf Step 4} Output $$\widehat{k} = \max\left\{k \in [n] \,\Big|\, \widehat{\FDR}(\widehat \cU_{k}) \leq \alpha \right\}.$$
\end{algorithm}

To ensure that Algorithm~\ref{alg:FDR} works as desired, the approximation of the Cauchy transforms needs to be accurate, which in turn requires a good estimate of the rank $r$.  Figure~\ref{fig:wrong_rank} illustrates the results of FDR estimates using \eqref{eq:fdr-estimate} with different rank estimates $\widehat r$.  We consider two problem instances of sizes $n = 500, 2000$, and in each case, the rank of $\bA$ is $r = 20$ and the noise $\bE$ is drawn from the GOE.  
In the right-hand-side of \eqref{eq:key-limit}, the rank $r$ is replaced by the estimates $\widehat r = r/2, r, 2r$.  As the figures demonstrate, an overestimate of the rank $r$ yields an underestimate of the true $\FDR$, which would ultimately result in subspace selection procedures that do not provide control at the desired level; however, these $\FDR$ underestimates appear to become less significant for larger $n$.  An underestimate of the rank $r$ yields an overestimate of the true $\FDR$, which would ultimately yield a procedure that provides control at the desired level but with nontrivial loss in power.  Therefore, high-quality estimates of the correct rank are crucial for obtaining a powerful procedure for subspace estimation that provides the desired $\FDR$ control.

\begin{figure}[t]
	\centering
	\begin{subfigure}{0.45\textwidth}
		\centering$\qquad n = 500$
		\vspace{0.2em}
		\includegraphics[width = \textwidth]{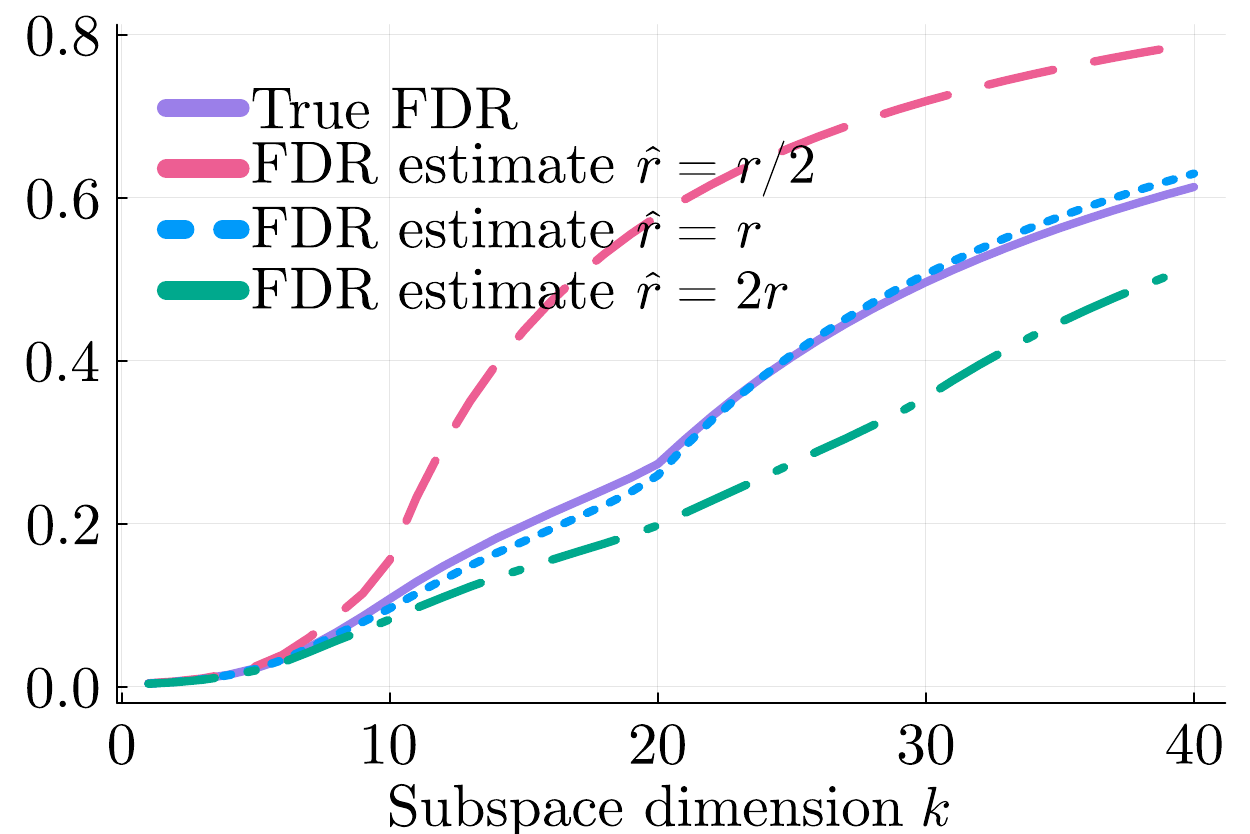}
	\end{subfigure}
	\begin{subfigure}{0.45\textwidth}
		\centering
		$\qquad n = 2000$
		\vspace{0.2em}
		\includegraphics[width = \textwidth]{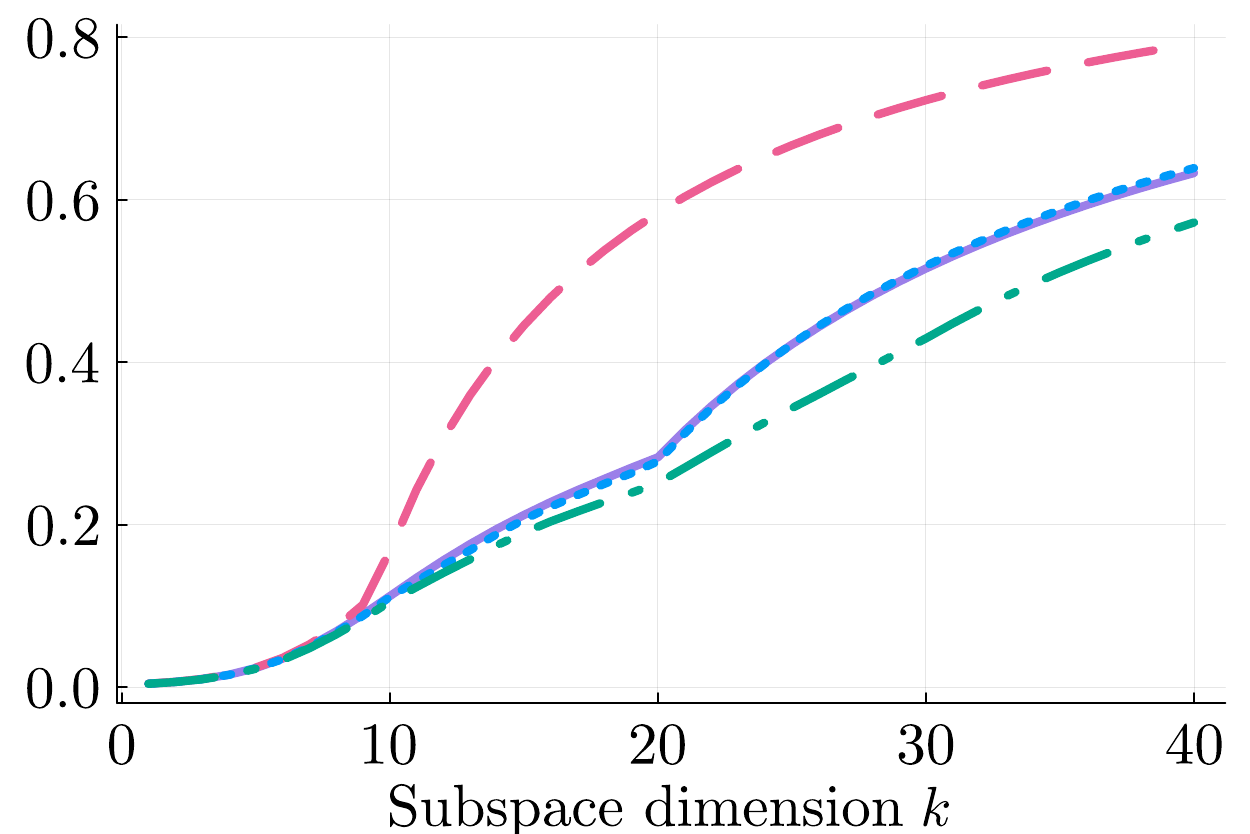}
	\end{subfigure}
	\caption{FDR estimates using \eqref{eq:fdr-estimate} against subspace dimension $k$ for different rank estimates $\widehat r$.  The noise is drawn from a GOE, and the true rank of $\bA$ is $r = 20$. The two figures display results for problem sizes $n = 500$ (left) and $n = 2000$ (right).}
	\label{fig:wrong_rank}
\end{figure}

To estimate the rank of $\bA$ from $\bX$, we consider the same setup as in Figure~\ref{fig:wrong_rank} and we plot the eigenvalues of $\bX$ in Figure~\ref{fig:eigenvalues}; the true rank $r = 20$ is hardly a distinctive feature and is therefore difficult to estimate directly from these scree plots.  To address this challenge, we again appeal to a basic observation from random matrix theory -- the gaps between consecutive eigenvalues of random matrices are controlled as $n$ grows large.  Unlike i.i.d. random variables, the eigenvalues of random matrices are known to repel each other and form a rather rigid grid with gaps of similar size.  This phenomenon was observed at the inception of random matrix theory. Indeed, Wigner's seminal work on resonance absorption identified eigenvalue spacings as key objects of study \cite{wigner1956results}:
\begin{quote}{\it
    Perhaps I am now too courageous when I try to guess the distribution of the distances between successive levels (of energies of heavy nuclei).  
     ... 
    The question is simply, what are the distances of the characteristic values of a symmetric matrix with random coefficients?}
\end{quote}
In that work, Wigner proposed a distribution for the spacings --- now known as Wigner's surmise --- which aimed to capture the repelling behavior. Subsequently, several works found connections with other problems in mathematics and the sciences \cite{mehta2004random, torquato2018hyperuniform, wolchover2018birds}; famously, Montgomery established the relation of eigenvalue spacings to the zeros of the Riemann zeta function \cite{montgomery1973pair}.  One of the main conclusions from the extensive line of work \cite{tao2010random, tao2011random, bourgade2012bulk, erdHos2015gap, tao2013asymptotic, nguyen2017random, bourgade2020random, benaych2016lectures, pillai2014universality} studying eigenvalue spacings is that they decrease at a controlled rate when $n$ goes to infinity. In particular, for eigenvalues in the interior of the bulk, the spacings are more or less uniform, i.e., $\lambda_i - \lambda_{i+1} = O(n^{-1})$. In comparison, spacings close to an edge of the bulk converge asymptotically to the Tracy-Widom distribution and tend to be more spread out, i.e., $\lambda_i - \lambda_{i+1} = O(n^{-2/3}).$ In both cases --- whether within the bulk or at the edge --- spacings decrease at a much faster rate than $n^{-1/2}$, and this is the property we exploit to identify the rank. 

\begin{figure}[t]
	\centering
	\begin{subfigure}{0.45\textwidth}
		\centering$\qquad n = 500$
		\vspace{0.2em}
		\includegraphics[width = \textwidth]{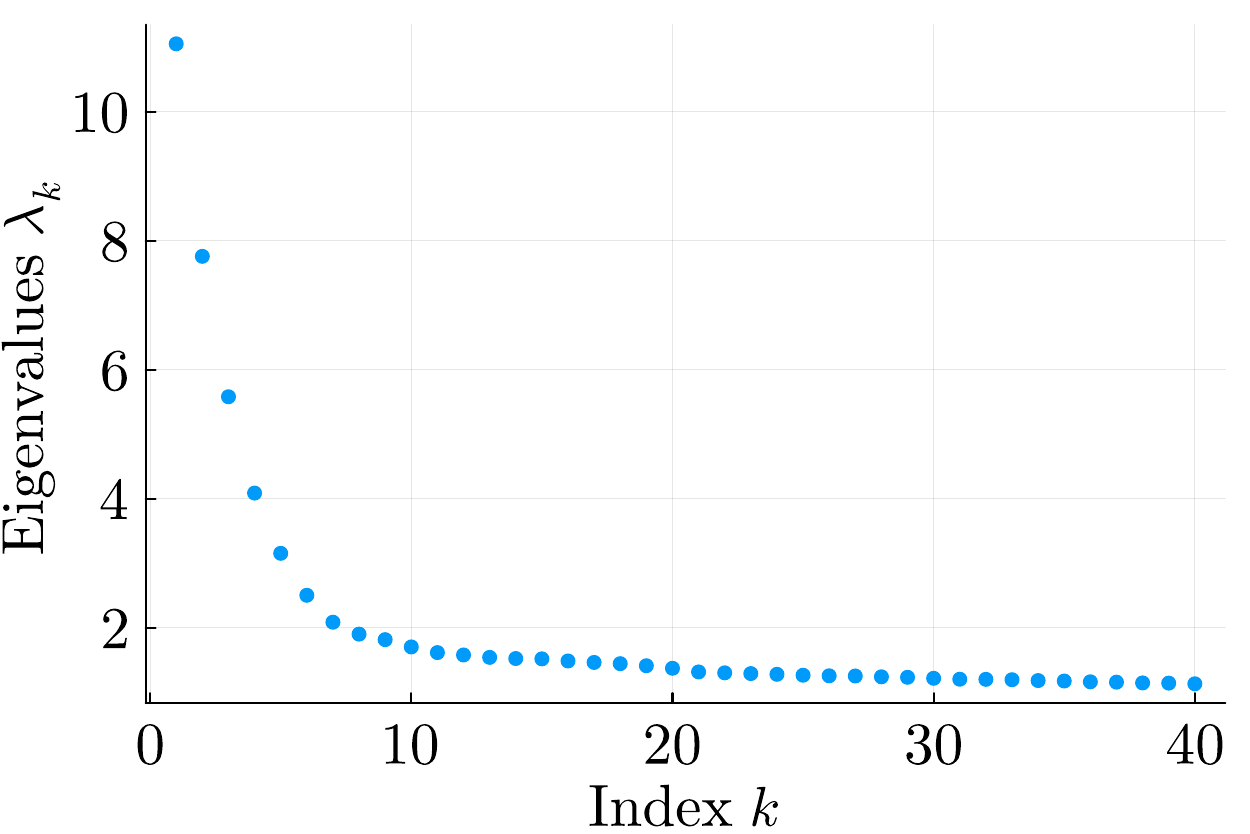}
	\end{subfigure}
	\begin{subfigure}{0.45\textwidth}
		\centering
		$\qquad n = 2000$
		\vspace{0.2em}
		\includegraphics[width = \textwidth]{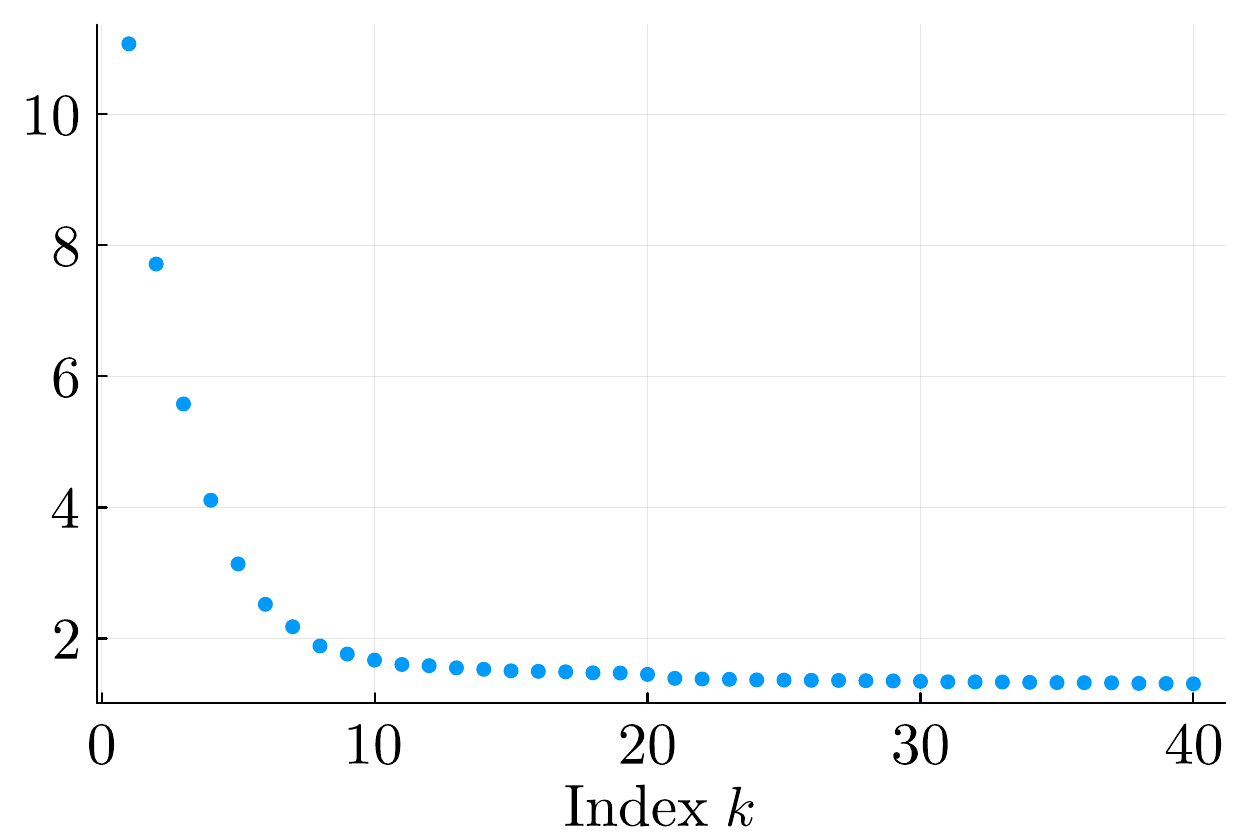}
	\end{subfigure}
	\caption{Eigenvalues against subspace dimension $k$. Same setting as in Figure~\ref{fig:wrong_rank}. }
	\label{fig:eigenvalues}
\end{figure}

To estimate the rank $r$ of $\bA$, Algorithm~\ref{alg:rank-estimate} computes the spacings between eigenvalues of $\bX$ and finds the largest index for which the spacing exceeds $p \cdot n^{-1/2}$ where $p > 0$ is a user-provided constant. Intuitively, the gaps associated with the noise $\bE$ should be much smaller than $p \cdot n^{-1/2}$ as discussed above, while the gaps associated with the signal $\bA$ have constant size; thus, the largest index for which the associated gap is above the threshold of $p \cdot n^{-1/2}$ should yield the components of $\bX$ that correspond to the signal $\bA$.  Figure~\ref{fig:spacings} provides an illustration of this approach in the problem instances from Figure~\ref{fig:eigenvalues}.
    \begin{algorithm}[h!]
  \caption{\textsc{RankEstimate}} %
  \label{alg:rank-estimate}
  {\bf Input}: Nonincreasing spectrum $\lambda(\bX)$, %
  {\it optional} $p > 0$ (default to $p = 1$).\\
{\bf Output}: A rank estimate $\widehat r$.\vspace{.2cm}\\
{\bf Step 1} For each $j  \leq n-1$ compute the eigenvalue spacing
$$ \Delta_{j} \leftarrow {\lambda_{j-1}(\bX) - \lambda_{j}(\bX)}.$$
{\bf Step 2} 
Compute
$$
\widehat r \leftarrow \max\left\{j \in [n/2]\,\Big| \, \Delta_{j+1} > p \cdot n^{-1/2}\right\}.
$$
\end{algorithm}

\begin{figure}[t]
    \centering
    \begin{subfigure}{0.45\textwidth}
        \centering$\qquad n = 500$
        \vspace{0.2em}
        \includegraphics[width = \textwidth]{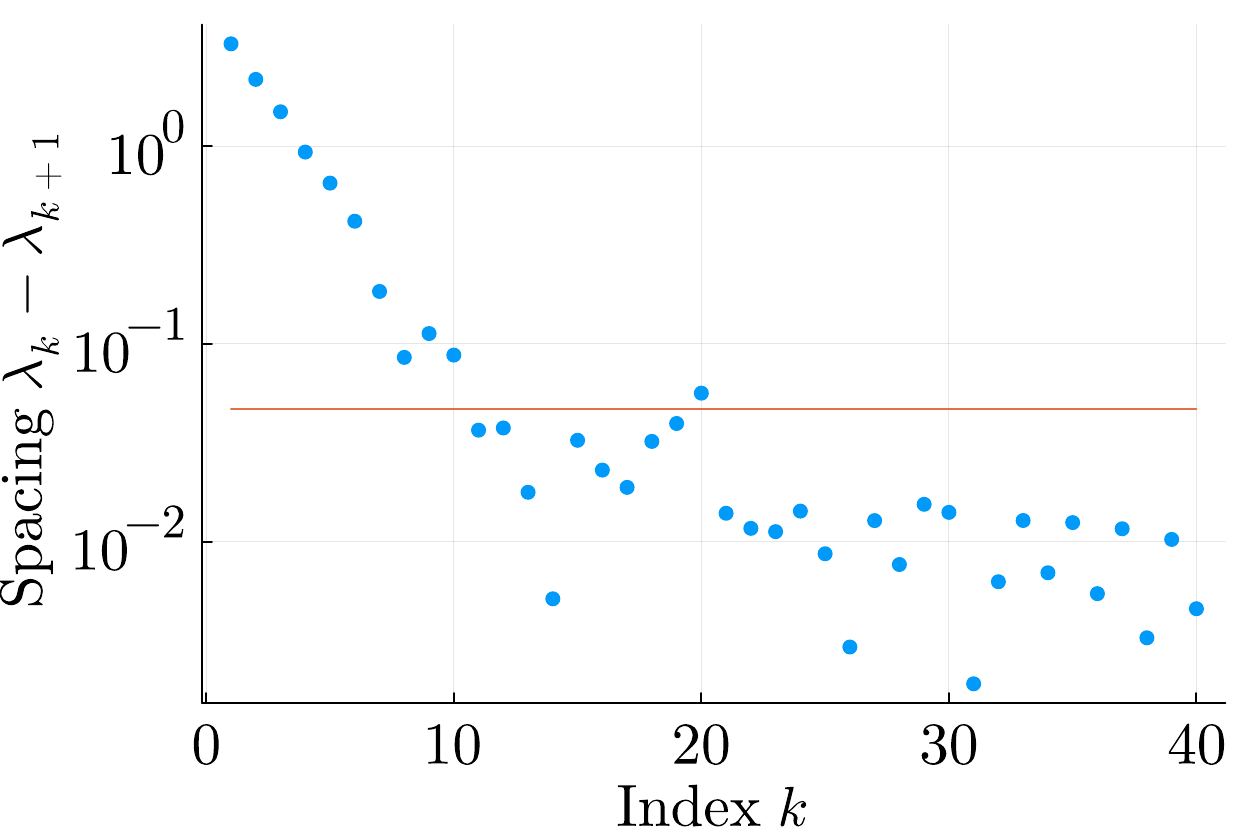}
    \end{subfigure}
    \begin{subfigure}{0.45\textwidth}
        \centering
        $\qquad n = 2000$
        \vspace{0.2em}
        \includegraphics[width = \textwidth]{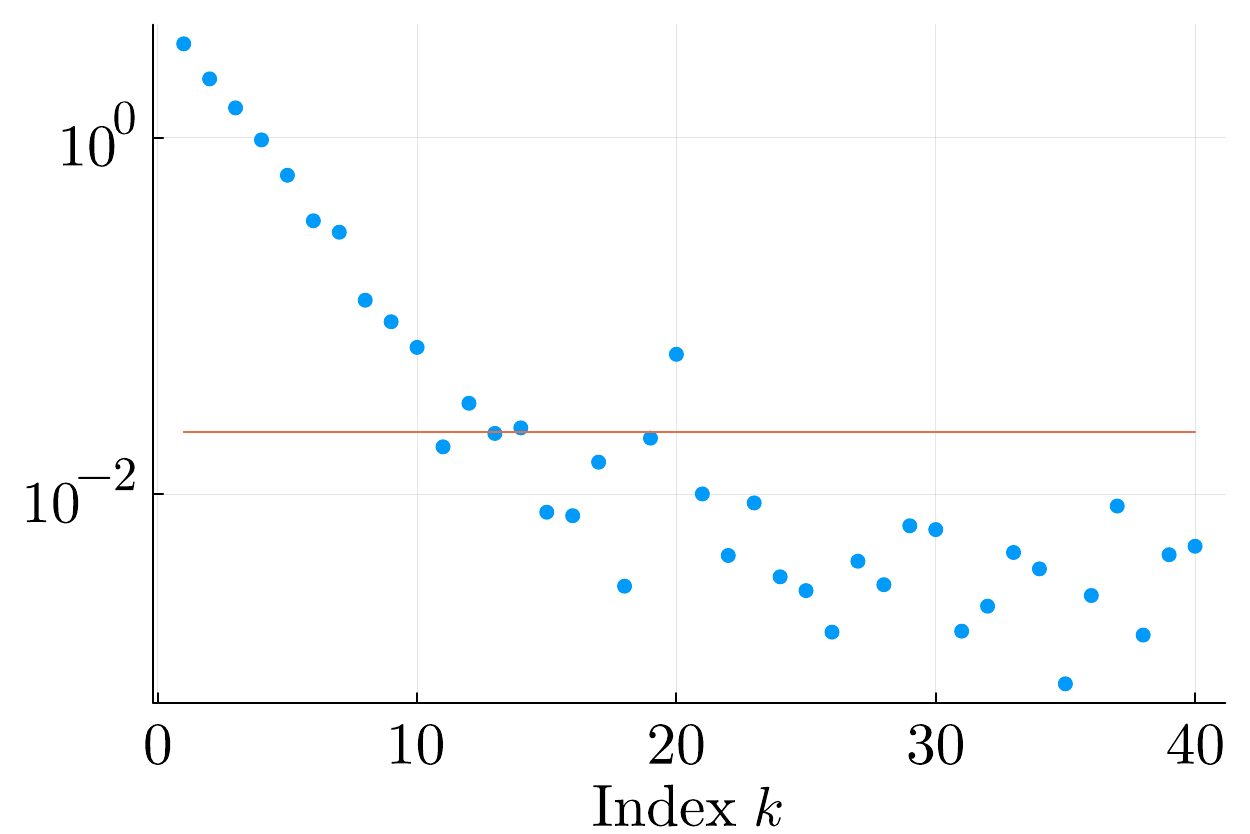}
    \end{subfigure}
    \caption{Eigenvalue spacings (log-scale) against subspace dimension $k$. Same setting as in Figure~\ref{fig:wrong_rank}. The red line shows the threshold that we use in our procedure (\cref{alg:rank-estimate}). }
    \label{fig:spacings}
\end{figure}

In summary, Algorithms~\ref{alg:FDR} and \ref{alg:rank-estimate} constitute our methodology for $\FDR$ control in subspace estimation.  Our procedure is parameter-free (the choice of $p$ in Algorithm~\ref{alg:rank-estimate} is optional and can default to $p=1$), and as we demonstrate in this paper, our method provably provides FDR control in subspace estimation under reasonable statistical assumptions.

\paragraph{Outline.} The remainder of this paper is organized as follows.  In Section~\ref{sec:algo}, we discuss the technical assumptions that form the basis of our theoretical guarantees.  Our main results on $\FDR$ control in subspace estimation are provided in Section~\ref{sec:guarantees}; we also provide examples of natural matrix ensembles that satisfy the conditions required for our theorems. 
Section~\ref{sec:asymmetric} provides extensions of our methods, assumptions, and results to the asymmetric case.  We describe the results of numerical experiments on synthetic and real data in Section~\ref{sec:experiments}. Section~\ref{sec:proofs} presents proofs of our theorems from Sections~\ref{sec:guarantees} and~\ref{sec:asymmetric}. Finally, Section~\ref{sec:future} closes the paper with future directions.

\subsection{Related Work}\label{sec:related}

\paragraph{False discovery rate control beyond classic settings.}
We are not aware of methods for controlling $\FDR$ in problems beyond multiple testing, variable selection, and related variants \cite{benjamini1995controlling, foster2008alpha, goeman2008multiple, benjamini2006adaptive, doi:10.1198/016214507000001373, 10.1214/009053606000000074, doi:10.1198/jasa.2010.tm09329, barber2015controlling,javanmard2018online, wang2022false}.  A line of work initiated by \cite{blanchard2014testing} studies FDR control in the setting of a continuum of binary hypotheses. The work of Taeb et al. \cite{taeb2020false} first proposed the $\FDR$ notion for subspace estimation, yet the methodological focus was on controlling the unnormalized variant $\mathrm{FD}$ of Figure~\ref{fig:FDvsFDR}.  Moreover, even for this unnormalized variant, the results in \cite{taeb2020false} do not provide control at a user-specified level.  In contrast, the present work provides control on $\FDR$ at a desired level for the particular case of a low-rank matrix corrupted by additive noise.
Finally, we remark that our methods share conceptual similarity with the work of Storey~\cite{storey2002direct} on multiple testing in that $\FDR$ control is provided by first estimating $\FDR$ and then controlling the estimate.
Indeed, the key steps of Algorithm~\ref{alg:FDR} center on estimating $\FDR$, and much of our analysis is focused on establishing the accuracy of these estimates.

\paragraph{Principal component analysis.}
The stylized model we study \eqref{eq:spike} was first introduced by Johnstone \cite{johnstone2001distribution} and has since been extensively investigated \cite{bai2012sample, cai2021subspace, abbe2022l, agterberg2022entrywise, cape2019two}.  There has been recent progress on the information-theoretic limits of PCA in detecting a low-rank signal \cite{perry2018optimality}.  A related line of work \cite{cai2021subspace, abbe2022l, agterberg2022entrywise, cape2019two} has focused on error bounds for the recovered eigenvectors and singular vectors.  There is also a line of work initiated by \cite{choi2017selecting} in which a testing approach is employed to threshold the singular values.  In broad overview, our work departs qualitatively from this prior literature by providing methods that control false discovery rate in subspace estimation, and we are not aware of prior efforts in this direction.

\paragraph{Matrix recovery via asymptotics.} Building on the random matrix theory results of Beynach-Georges and Nadakuditi~\cite{benaych2011eigenvalues, benaych2012singular}, several authors, including Donoho, Gavish, and Romanov~\cite{donoho2023screenot} and Nadakuditi~\cite{nadakuditi2014optshrink} developed methodology for denoising low-rank matrices corrupted by additive noise.  Broadly, the approach in these papers is based on identifying exact formulas that hold in the limit of large $n$ and employing data-driven estimates of these asymptotic quantities for the case of finite $n$.  The approach in the present paper is similar to this prior body of work but with the important distinction that we seek subspace estimates with $\FDR$ control guarantees rather than a low-rank matrix estimate with the smallest mean-squared error (MSE).  The precise asymptotic formulas corresponding to these two criteria are different.  Moreover, controlling the MSE for low-rank approximation is a distinct goal than controlling the FDR for subspace estimation; in some applications identifying such a subspace with control on false positive error is significant for subsequent steps in the data analysis pipeline such as dimension reduction.

\paragraph{Eigenvalue spacings.} Several papers have studied eigenvalue gaps in random matrices \cite{tao2010random, tao2011random, bourgade2012bulk, erdHos2015gap, tao2013asymptotic, nguyen2017random, bourgade2020random}, and the underlying phenomemon that the eigenvalue gaps decay at a predictable rate is arguably a universal one, i.e., it holds in many settings.  To our knowledge, this phenomenon has not been explicitly employed in low-rank or subspace estimation previously.  As described above, the problem of subspace estimation with false positive error control brings to the fore the challenge of estimating the rank of the underlying low-rank signal well.  In this context, we crucially leverage the eigenvalue rigidity of large random matrices~\cite{pillai2014universality, erdHos2012rigidity, benaych2016lectures}.

\subsection{Notation}
Given a real-valued function $f\colon \RR \rightarrow \RR$ and a constant $c \in \RR$, we denote the right-sided limit by $f(c^{+}) = \lim_{t \downarrow c}f(t)$ and the left-sided limit by $f(c^{-})=\lim_{t\uparrow c}f(t)$. We use $\delta_{ij}$ to denote the indicator function of $i = j$. We employ the shorthand $[n] = \{1, \dots n\}.$ %
The symbols $\SSS^n$ and $\SSS^n_+$ denote the sets of symmetric and (symmetric) positive semidefinite matrices, respectively. We index eigenvalues and singular values in nonincreasing order, e.g., $\lambda_1(\bM) \geq \cdots \geq \lambda_n(\bM)$ for any symmetric matrix $\bM \in \SSS^n$. Given a sequence $\{\bM_{n}\}$ of symmetric random matrices of increasing dimension, we use $\mu_{\bM_{n}}$ to denote the empirical eigenvalue distribution of $\bM_{n}$, i.e.,
  \begin{equation}
    \label{eq:empirical}
    \mu_{\bM_{n}} = \frac{1}{n}\sum_{j=1}^{n} \delta_{\lambda_{i}(\bM_{n})}.
  \end{equation}
Similarly, for a sequence of asymmetric matrix $\{\bM_n\}\subseteq \RR^{n\times m}$ we use $\mu_{\bM_n}$ to denote the empirical singular value distribution of $\bM_n.$
Given sequences $(a_n)_n$ and $(b_n)_n$, we write $a_n = o(b_n)$ if $\lim_{n\rightarrow \infty} a_n/b_n = 0.$  The $L^p$ norm of a random variable $\mathbf{Z}$ is denoted $\|\mathbf{Z}\|_p := \EE[|\mathbf{Z}|^p]^{1/p}$.

\section{Background and Assumptions}\label{sec:algo}

In this section we describe the main assumptions that form the basis of our development in the symmetric case.  Theoretical guarantees for the symmetric case are presented in Section~\ref{sec:guarantees}, with the proofs given in Section~\ref{sec:proofs}.  The nonsquare case is discussed in Section~\ref{sec:asymmetric}.

As described in Section~\ref{sec:intro}, the core of Algorithm~\ref{alg:FDR} is an estimate-then-control approach for selecting a subspace -- we estimate the FDR and use our estimate as a proxy to select the number of top components $k$. The key insight underlying this approach comes from random matrix theory: for a sequence of problem instances of increasing size $\{\bX_n = \bA_n + \bE_n ~|~ \bA_n \in \SSS^n_+, ~ \bE_n \in \SSS^n\}$ satisfying certain conditions, the $\FDR$ associated to the estimator that selects the principal eigenspaces of $\bX_n$ converges asymptotically to a quantity that is amenable to estimation from the spectrum of $\bX_n$.  We describe next the assumptions on problem sequences that facilitate our development. The first assumption primarily concerns the signal sequence $\{\bA_n\}$.
\begin{assumption}\label{ass:model1}
    The following conditions hold for the sequence $\{\bA_n\}$ for all $n$.
    \begin{itemize}[leftmargin=.4cm]
       \item[] (\textbf{Independence}) The distributions of $\bA_{n} \in \SSS^n_+$ and $\bE_{n}\in \SSS^n$ are independent.
      \item[] (\textbf{Spectrum of the low-rank signal}) There exists a constant $r \in \NN$ and deterministic real numbers $\theta_{1} \geq \dots \geq \theta_{r} > 0$ such that $\rank(\bA_{n})= r$ and $\lambda_i(\bA_n) = \theta_{i}$ for all $i\leq r$.
      \item[] (\textbf{Invariance}) The eigenvector distribution of $\bA_n$ or of $\bE_{n}$ is orthogonally invariant.
    \end{itemize}
\end{assumption}
\noindent The second assumption centers around properties of the noise sequence $\{\bE_n\}$.
\begin{assumption}\label{ass:model2}
    The following conditions hold for the sequence $\{\bE_n\}$.
    \begin{itemize}[leftmargin=.4cm]
      \item[] (\textbf{Spectrum}) The empirical eigenvalue distribution $\mu_{\bE_{n}}$, defined as in \eqref{eq:empirical}, converges almost surely weakly to a deterministic compactly supported measure $\mu_{\bE}$ as $n$ goes to infinity.
      \item[] (\textbf{Edge convergence})  Let the infimum and supremum of $\supp(\mu_{\bE})$ be denoted by $a$ and $b$.  The smallest and largest eigenvalues of $\bE_{n}$ converge almost surely to $a$ and to $b$, respectively.
      \item[] (\textbf{Decay at the edge}) The limiting distribution $\mu_{\bE}$ has a density $f_{\bE}$ such that as $t \rightarrow b$ with $t < b$, we have that $f_{\bE}(t) \sim c \cdot (b-t)^{\alpha}$ for some constant $c > 0$ and exponent $\alpha \in (0, 1]$. An analogous decay holds for the lower edge $a.$
    \end{itemize}
    \end{assumption}

\noindent We comment briefly on these assumptions. The requirement that $\bA_n$ is positive semidefinite is not strict; one could drop this condition by defining a method that considers components at both ends of the spectrum. However, we keep this assumption since it simplifies the presentation.  The invariance condition in Assumption~\ref{ass:model1} is necessary for our analysis, but it could potentially be slightly relaxed; see \cite[Section 2]{benaych2012singular}. 
 The three requirements in Assumption~\ref{ass:model2} might appear restrictive at first sight, but they hold for several natural random matrix ensembles such as the GOE from the introduction; we refer the reader to~\cref{sec:guarantees-matrix} for further examples.

Assumptions~\ref{ass:model1} and \ref{ass:model2} underpin the success of Algorithm~\ref{alg:FDR}; see Theorem~\ref{thm:main} in Section~\ref{sec:guarantees}. Under these assumptions, the asymptotic behavior of the eigenvalues and eigenvector of $\bX_n$ can be pinned down in terms of the distribution $\mu_{\bE}$ and the spectrum of $\bA_n$ based on recent advances in random matrix theory \cite{baik2005phase, benaych2011eigenvalues, benaych2012singular, au2023bbp}.  Broadly speaking, the spectrum of $\bX_n$ has two components: a bulk of eigenvalues coming from the noise and a small set of eigenvalues corresponding to those of $\bA_{n}$.  While the empirical distribution of the bulk converges towards $\mu_{\bE}$, the eigenvalues associated to those of $\bA_n$ exhibit a sharp phase transition. In particular, for any $i \in [r]$,
     \begin{equation}\label{eq:bbp}
         \lambda_i(\bX_n) \rightarrow \begin{cases}
             G^{-1}_{\mu_{\bE}}\left({1}/{\theta_i}\right) & \text{if $\theta_i > 1/G_{\mu_{\bE}}(b^+)$},\\
             b & \text{otherwise,}
         \end{cases}
     \end{equation}
where $G_{\mu_{\bE}}$ is the Cauchy transform \eqref{eq:cauchy} of $\mu_{\bE}$.  This type of phenomenon is known as the \emph{BBP phase transition}, named after Baik, Ben Arous, and Peche, whose pioneering work \cite{baik2005phase} unveiled this behavior for the first time. Intuitively, if an eigenvalue of $\bA_n$ is above the noise level, i.e., $\theta_i > 1/G_{\mu_{\bE}}(b^+)$, then the associated component of $\bX_n$ is not subsumed by the bulk. On the other hand, if an eigenvalue is below the noise level, then its associated eigenvalue is absorbed inside the bulk.  Further, correlations between the eigenvectors of $\bA_n$ and $\bX_n$ can be characterized in terms of the Cauchy transform of $\mu_E$ and the spectrum of $\bA_n$; we defer the additional details to \cref{sec:background}.  By understanding such asymptotic properties of $\bX_n$, one can characterize the limiting $\FDR$ for any fixed $k$ as follows: 
\begin{equation}\label{eq:key-limit-refined}
         \limsup_{n \rightarrow \infty} \FDR(\widehat \cU^{(n)}_k) \leq \lim_{n \rightarrow \infty} 1 + \frac{1}{k} \sum_{i=1}^{\min\{k, r^\star\}} \frac{G_{\mu_{\bE}}\left(\lambda_i(\bX_n)\right)^2}{G'_{\mu_{\bE}}\left(\lambda_i(\bX_n)\right)}
\end{equation} 
where $r^\star$ denotes the so-called \emph{observable rank}, which corresponds to the number of components that are not subsumed by the noise:
\begin{equation}
      \label{eq:identifiable-rank}
      r^{\star} := \# \{j \mid \theta_{j} > 1/G_{\mu_{\bE}}(b^{+})\}.
\end{equation}
Expression~\eqref{eq:key-limit-refined} is a generalization of \eqref{eq:key-limit} from Section~\ref{sec:intro} as it allows for problem instances in which the signal spectrum is not well-separated from that of the noise; see Proposition~\ref{lem:fdr-convergence} for a proof of \eqref{eq:key-limit-refined}.  
We shall see in~\cref{sec:guarantees} that when $r = r^\star$, the above inequality holds with equality and the lim-sup can be replaced by a limit.

Turning our attention next to assumptions under which Algorithm~\ref{alg:rank-estimate} provides high-quality rank estimates, we begin by elaborating on the rationale provided in Section~\ref{sec:intro}.  Recall that Algorithm~\ref{alg:rank-estimate} chooses the largest index for which the associated eigenvalue gap is greater than $p \cdot n^{-1/2}$.  As articulated in Section~\ref{sec:intro}, the reasoning behind this choice is based on the properties of the eigenvalue spacings of the noise $s_i = \lambda_i(\bE_n) - \lambda_{i+1}(\bE_n)$.  These spacings decrease at a rate of $s_i = O(n^{-1})$ inside the bulk and more slowly $O(n^{-2/3})$ near the edges \cite{tao2010random, tao2011random, bourgade2012bulk, erdHos2015gap, benaych2016lectures, pillai2014universality}.  In principle, we could have picked $n^{-\omega}$ with any $\omega < 2/3$ in Step 2 of~\cref{alg:rank-estimate}, but we have observed that the more aggressive $n^{-1/2}$ exhibits good practical performance.  The next assumption codifies this discussion, and it provides the basis for the success of Algorithm~\ref{alg:rank-estimate}.

\begin{assumption}\label{ass:spacings}
   The following conditions hold for the sequence $\bE_n$.
   \begin{enumerate}[leftmargin=.4cm]
       \item[] (\textbf{Connected support}) The support of the asymptotic spectral measure $\mu_{\bE}$ is connected.
       \item[] (\textbf{Spacings}) Let $s_{j}^{(n)} = \lambda_{j}(\bE_{n}) - \lambda_{j+1}(\bE_{n})$, for all large enough $n$ and all $i \leq n$
     $$
     s_{i}^{(n)} = o \left( n^{-1/2}\right) \qquad\text{almost surely.}
     $$
   \end{enumerate}
\end{assumption}

Without the connected support assumption, there could be large eigenvalue spacings within the noise, and thus, our method could fail to identify the right rank. One way to bypass the connectedness assumption is to assume that we have access to an upper bound on the rank $R$ of the signal sequence $\{\bA_n\}$.  With this information, one could only focus on the first $R-1$ eigenvalue spacings, and for sufficiently large $n$, Algorithm~\ref{alg:rank-estimate} would exclude large spacings produced by the noise. In Section~\ref{sec:guarantees-matrix}, we describe natural matrix ensembles that satisfy Assumptions~\ref{ass:model2} and \ref{ass:spacings}.

\section{Main Results}\label{sec:guarantees}

We present here the main theoretical results of our work for the symmetric case.  Our first set of results in Section~\ref{sec:guarantees-fdr-control} pertain to guarantees that Algorithms~\ref{alg:FDR} and~\cref{alg:rank-estimate} are correct under the assumptions described in Section~\ref{sec:algo} for all large $n$, i.e., they provide $\FDR$ control at the desired level as well as accurate rank estimates.  Next, we show in Section~\ref{sec:guarantees-matrix} that a number of natural random matrix ensembles satisfy Assumptions~\ref{ass:model2} and~\ref{ass:spacings} of Section~\ref{sec:algo}.

\subsection{Guarantees for FDR Control}\label{sec:guarantees-fdr-control}

Recall that we consider a sequence of growing-sized problem instances $\{\bX_n = \bA_n + \bE_n ~|~ \bA_n \in \SSS^n_+, ~ \bE_n \in \SSS^n\}$.  We denote the column space of $\bA_n$ by $\cU^{(n)}$, and we denote the $k$-dimensional principal subspace of $\bX_n$, i.e., the span of the $k$ eigenvectors of $\bX_n$ corresponding to the $k$ largest eigenvalues, by $\widehat{\cU}^{(n)}_k$.

One challenge with characterizing $\Tr\left( P_{\widehat{\cU}_{k}^{(n)}} P_{\cU^{\perp}}\right)$ in the definition of $\FDR(\widehat{\cU}^{(n)}_k)$ from \eqref{eq:fdr} -- beyond the case of well-separated signal and noise spectra as in \eqref{eq:key-limit} -- is that the eigenvectors of $\bX_n$ associated with eigenvalues that are subsumed by the noise in the BBP transition, i.e., $\theta_j \leq 1/G_{\mu_{\bE}}(b^+)$, are not clearly characterized in terms of their correlation with the corresponding eigenvectors of $\bA_n$.  Intuitively, one would expect that these pairs of eigenvectors are asymptotically uncorrelated; whether this is the case is a significant open question in random matrix theory that has only recently been answered in more restrictive settings \cite[Theorem 1.3]{au2023bbp}, and might require additional assumptions \cite[Remark 2.6]{benaych2011eigenvalues}.  We circumvent this difficulty by considering only the span $\bar{\cU}^{(n)}$ of the first $r^\star$ components of $\bA_n$ corresponding to the observable rank \eqref{eq:identifiable-rank}, and defining the following quantities:
\begin{equation} 
    \label{eq:upper-fdr-main}
\overline{\FDR}(\widehat{\cU}^{(n)}_k) 
= \EE \left[ \frac{\Tr \left( P_{\widehat{\cU}^{(n)}_k} P_{\bar{\cU}^{(n)^{\perp}}} \right)}{k}\right] \quad \text{and} \quad \FDR^{\infty}(k) =\lim_{n\rightarrow \infty } \overline{\FDR}(\widehat{\cU}^{(n)}_k).
\end{equation}

We prove in Appendix~\ref{sec:proof-main} that $\FDR^\infty$ is well-defined.  As $\bar{\cU}^{(n)} \subset \cU^{(n)}$, we have that $\FDR(\widehat{\cU}^{(n)}_k) \leq \overline{\FDR}(\widehat{\cU}^{(n)}_k)$.  We leverage this observation in our next result by proving that Algorithm~\ref{alg:FDR} provides control on $\overline{\FDR}(\widehat{\cU}^{(n)}_k)$ at the desired level, which in turn implies control on $\FDR(\widehat{\cU}^{(n)}_k)$. The proof of this result is deferred to Section~\ref{sec:proof-main}.

\begin{theorem}[\textbf{FDR control}]\label{thm:main}
Consider a sequence of problem instances $\{\bX_n = \bA_n + \bE_n ~|~ \bA_n \in \SSS^n_+, ~ \bE_n \in \SSS^n\}$ for which Assumptions~\ref{ass:model1} and~\ref{ass:model2} hold. Further, suppose the rank estimate $\widehat r$ required for Algorithm~\ref{alg:FDR} satisfies $$r^\star \leq \widehat r \leq r$$ almost surely for all large $n$.  For a generic $\alpha \in (0,1)$ and for all large $n$, the output of \cref{alg:FDR} satisfies $\overline{\FDR}(\widehat{\cU}^{(n)}_k) \leq \alpha$, which in turn implies that $\FDR(\widehat{\cU}^{(n)}_k) \leq \alpha$. Moreover, if $\alpha \leq \FDR^{\infty}(r)$, the output $\widehat k$ of \cref{alg:FDR} satisfies 
  $$\widehat k = \max\{k \in [r] \mid \overline{\FDR}(\widehat{\cU}^{(n)}_k) \leq \alpha\}.$$
\end{theorem}
In words, this theorem states that the true FDR is controlled by Algorithm~\ref{alg:FDR} for large enough $n.$  Moreover, Algorithm~\ref{alg:FDR} maximizes power for sufficienty small $\alpha$ in the sense that it selects the largest-dimensional principal subspaces of $\bX_n$ that satisfy $\overline{\FDR}(\widehat{\cU}^{(n)}_k) \leq \alpha$.  

We comment on several aspects of the result.  First, the requirement that $\alpha$ is generic is to ensure that $\alpha$ does not equal $\FDR^\infty(k)$ for any $k \in \NN$, and this is arguably an artifact of our analysis. Nonetheless, our theorem holds for any value of $\alpha$ outside of this set of measure zero.  Second, in terms of power maximization, the bound $\alpha \leq \FDR^\infty(r)$ is appropriate as any value of $\alpha$ greater than this quantity could lead to the selection of more components than the rank of $\bA_n$; we expect these components to arise from the noise $\bE_n$ and to have little correlation with the signal.  Finally, it is natural to wonder whether the inequality $\FDR(\widehat{\cU}^{(n)}_k) \leq \overline{\FDR}(\widehat{\cU}^{(n)}_k)$ is loose, which would lead to our method being too conservative.  Numerical experiments suggest that this is not the case, and in fact, it appears that $\FDR(\widehat{\cU}^{(n)}_k)$ converges to $\FDR^{\infty}(k)$ as $n \rightarrow \infty$.  It would be interesting to establish such convergence in the future; such an analysis would rely on resolving the open question we described in the discussion preceding \eqref{eq:upper-fdr-main} on obtaining a characterization of the correlation between eigenvectors of $\bX_n$ and $\bA_n$ below the BBP threshold.

Next, we recall that Algorithm~\ref{alg:FDR} must be supplied with an accurate rank estimate.  The following result shows that for noise matrices with controlled eigenvalue spacings,~\cref{alg:rank-estimate} produces accurate estimates.
  
\begin{theorem}[\textbf{Rank estimate}]\label{thm:rank-estimate} Consider a sequence of problem instances $\{\bX_n = \bA_n + \bE_n ~|~ \bA_n \in \SSS^n_+, ~ \bE_n \in \SSS^n\}$ for which Assumptions \ref{ass:model1}, \ref{ass:model2}, and \ref{ass:spacings} hold.  Then, the output of Algorithm~\ref{alg:rank-estimate}, $\textsc{RankEstimate}(\bX_{n}),$  almost surely satisfies 
  $$
  r^\star \leq \textsc{RankEstimate}(\bX_{n}) \leq r \qquad \text{for all large enough }n.
  $$
  \end{theorem}
This result follows from an application of Weyl's interlacing inequalities, and its proof is presented in~\cref{proof:rank-estimate}. Numerical experiments suggest that Algorithm~\ref{alg:rank-estimate} outputs $r^\star$ even when $r \neq r^\star$; nonetheless, existing results from random matrix theory do not appear to be strong enough to establish this equality.  We leave this as an interesting open question for future work. %

\subsection{Matrix Ensembles satisfying Assumptions~\ref{ass:model2} and~\ref{ass:spacings}}\label{sec:guarantees-matrix}

In this section, we identify two ensembles — Wigner and Wishart matrices — for which Assumptions~\ref{ass:model2} and~\ref{ass:spacings} hold.  For Assumption~\ref{ass:spacings} in particular, we leverage results from random matrix theory concerning the rigidity of the empirical spectrum \cite{erdHos2012rigidity, pillai2014universality}.  Intuitively, rigidity amounts to the $i$'th eigenvalue of an $N \times N$ random matrix being anchored around the $i/N$-quantile of their asymptotic spectral distribution; this phenomenon ensures that eigenvalue spacings shrink at a controlled rate.

We start by describing Wigner matrices, a broad family of symmetric random matrices.
\begin{definition} \label{def:wigner}
    A \emph{Wigner ensemble} is a sequence of symmetric random matrices $\left\{\bH_n \in \SSS^n\right\}$ for which the following conditions hold for all $n$ (we suppress the subscript $n$ in stating these conditions).
    \begin{enumerate}
        \item (\textbf{Independence}) The upper-triangular entries $\{\bH_{ij} \mid 1 \leq i \leq j \leq n\}$ are independent. 
        \item (\textbf{Mean and variance}) For each $i,j$ we have $\EE \bH_{ij} = 0$ and $\EE |\bH_{ij}|^2 = \frac{1}{n} (1+O(\delta_{ij}))$; the factor $(1+O(\delta_{ij}))$ requires the off-diagonal entries to have variances equal to $\frac{1}{n}$ but it allows for the diagonal entries to have nonuniform variances that scale as $\frac{1}{n}$ as $n$ grows large.
        \item (\textbf{Moment bounds}) For each $p \in \NN$, there exists a constant $C_p > 0$ such that $\|\sqrt{n} \bH_{ij} \|_p \leq  C_p$ for all indices $i,j$.
    \end{enumerate}
\end{definition}

The first two conditions are standard.  The last condition entails some degree of concentration on the entries of $\bH$. In particular, with $C_p \asymp \sqrt{p}$ the third condition reduces to subgaussian decay of the entries, while with $C_p \asymp p$ the condition reduces to subexponential decay \cite{vershynin2018high}.  A popular example of Wigner matrices is the GOE, that is, $(\bG + \bG^\top)/\sqrt{2n}$ with $\bG$ being a random matrix with i.i.d. standard Gaussian entries.

Wigner ensembles have been widely studied over the last century, beginning with the seminal work of Eugene Wigner on energy levels of heavy nuclei \cite{wigner1993characteristic, wigner1958distribution}.  It is now well-known that the empirical distribution of the spectrum of Wigner matrices converges weakly almost surely 
to the semicircle law, i.e.,  to the distribution with density
$$
\varrho(t) = \frac{1}{2\pi} \sqrt{(4 - t^2)_+},
$$
which is supported on the interval $(-2, 2)$. Wigner matrices satisfy Assumption~\ref{ass:model2} whenever the distribution of the eigenvectors is rotationally invariant, e.g., when the entries are Gaussian. The next result shows that the spacings of Wigner matrices satisfy~\cref{ass:spacings}. The proof follows from an eigenvalue rigidity result first established in \cite{erdHos2012rigidity}, and we present it in~\cref{proof:wigner-spacings}.

  \begin{proposition}[\textbf{Wigner matrix}]\label{cor:wigner}
    Suppose the sequence of matrices $\{\bE_{n} \in \SSS^n\}$ is a Wigner ensemble. We have that Assumptions~\ref{ass:model2} and \ref{ass:spacings} hold.
  \end{proposition}

Next, we turn to another salient class of random matrices known as Wishart ensembles. These arise commonly in applications by virtue of being sample covariance matrices of data with independent coordinates. 
\begin{definition}\label{def:wishart}
A \emph{Wishart ensemble} is a sequence of symmetric matrices $\{\bW_n = \bY_{n \times m} \bY_{n \times m}^\top \in \SSS^n\}$ with $\{\bY_{n \times m} \in \RR^{n \times m}\}$ being a sequence of random matrices for which the following hold for all $n,m$ (we suppress the subscripts $n,m$ in stating these conditions).
    \begin{enumerate}
        \item (\textbf{Independence}) The 
        entries $\{\bY_{ij} \mid 1 \leq i  \leq n , 1 \leq j \leq m\}$ are independent. 
        \item (\textbf{Mean and variance}) For all indices $i,j$ we have $\EE \bY_{ij} = 0$ and $\EE |\bY_{ij}|^2 = \frac{1}{\sqrt{m}}.$
        \item (\textbf{Tail decay}) There exists a constant $c > 0$ such that for all $i, j$ we have that $\PP(|\bY_{ij}| \geq t) \leq c \exp\left(-(t\sqrt{m})^{c}\right)$.
    \end{enumerate}
Finally, we have the following requirement on the sizes $n,m$.
\begin{enumerate}
\setcounter{enumi}{3}
\item (\textbf{Limiting size}) The asymptotic ratio of the matrix dimensions $\vartheta := \lim_{n,m \rightarrow \infty } n/m$ satisfies $$ 0< \vartheta  < \infty. $$
\end{enumerate}
\end{definition}
All these conditions are rather standard in the literature. The third condition plays an analogous role to that of the moment bound in the third condition of \cref{def:wigner}, and it reduces to the tails being subgaussian and subexponential when $\vartheta = 2$ and $\vartheta = 1$, respectively \cite{vershynin2018high}. A common example of Wishart ensembles are sample covariance matrices of standard Gaussian vectors, i.e., $\bW = \frac{1}{m} \bG \bG^\top$ with the entries of $\bG$ being i.i.d. standard Gaussians.

Vladimir Marchenko and Leonid Pastur's influential work \cite{marchenko1967distribution} established the asymptotic convergence of the spectrum of certain Wishart ensembles to the Marchenko-Pastur distribution, which is characterized by the following density:
  \begin{equation}
      \label{eq:mar-pas}
      \nu(t) = \frac{1}{2\pi t} \sqrt{\left((c_+ - t)(t - c_-)\right)_+} \quad \text{with}\quad c_{\pm} = \left(1 \pm \sqrt{\vartheta}\right)^2.
  \end{equation} 
  Subsequent efforts showed the universality of this distribution for general Wishart ensembles \cite{pillai2014universality}. %
  Unlike Wigner matrices, Wishart ensembles might fail to satisfy the decay condition %
  Thus, to satisfy~\cref{ass:model2}, the matrices $\bY$ have to be asymptotically rectangular so that $\vartheta \neq 1$. Interestingly, an asymptotically rectangular shape is also required to obtain control of the eigenvalue spacings at the lower end of the spectrum. 

  \begin{proposition}[\textbf{Wishart matrices}] \label{prop:covariance}
  Suppose the sequence of matrices $\{\bE_{n} \in \SSS^n\}$ is a Wishart ensemble with $\vartheta \neq 1$. We have that the Assumptions~\ref{ass:model2} and~\ref{ass:spacings} hold.
  \end{proposition}

  The proof of this proposition is completely analogous to that of \cref{cor:wigner} and leverages an eigenvalue rigidity result established in \cite{pillai2014universality}. We give a proof sketch in~\cref{proof:wishart-spacings}.

\section{Extension to Asymmetric Matrices}\label{sec:asymmetric}

In this section, we extend our methods to asymmetric matrices. The results and algorithms that we present are conceptually analogous to those in the preceding sections concerning the symmetric case. However, the lack of symmetry makes several of the statements slightly more involved. We consider the problem of estimating the column and row spaces of a rank-$r$ matrix $\bA\in \RR^{n\times m}$ given the observation
\begin{equation*}
 \bX = \bA + \bE
\end{equation*}
where $\bE$ represents additive noise.  We denote the column and row spaces of $\bA$ by $\cU$ and $\cV$, respectively, and these two subspaces represent the signal we wish to recover.  As before, our approach is to our goal is to control the false discovery rate incurred by picking the top $k$ components of the singular value decomposition (SVD) of $\bX = \widehat \bU \bSigma \widehat \bV^\top$. Namely, we wish to control  
\begin{equation}\label{eq:as-FDR}
    \FDR_{\texttt{l}}\left(\widehat \cU^{(k)}\right) = \EE \left[\frac{\Tr\left(P_{\widehat \cU_k} P_{\cU^{\perp}}\right)}{\max\{k, 1\}}\right] \quad \text{and} \quad \FDR_{\texttt{r}}\left( \widehat {\cV}^{(k)}\right) = \EE \left[\frac{\Tr\left(P_{\widehat \cV_k} P_{\cV^{\perp}}\right)}{\max\{k, 1\}}\right]
\end{equation}
where $\widehat\cU^{(k)} = \text{span}\{\widehat \bu_1, \dots \widehat \bu_k\}$ and $\widehat\cV^{(k)} = \text{span}\{\widehat \bv_1, \dots \widehat \bv_k\}$ are the subspaces spanned by the top $k$ left and right singular vectors of $\bX$, respectively.  We note here that in some contexts a data analyst may only seek FDR control for one of these subspaces; for example, one might only wish to perform dimension reduction of the feature space in some applications. Accordingly, the algorithm we describe in the section can control just one or both FDRs concurrently. Throughout this section, we assume without loss of generality that $n \leq m$.

\subsection{FDR Control for Column and Row Spaces}

We begin by presenting extensions of Algorithms~\ref{alg:FDR} and~\ref{alg:rank-estimate} based on a similar template, but with appropriate adaptations to use singular values and to estimate the relevant quantities for the asymmetric FDR \eqref{eq:as-FDR}.

\begin{algorithm}[h]
  \caption{FDR control for singular subspaces\hfill} %
  \label{alg:FDR-as}

  {\bf Input}: Observation $\bX \in \RR^{n \times m}$, level $\alpha$, and rank estimator $\rh$ (e.g., via \cref{alg:rank-estimate-as}).\\
  {\bf Output}: Estimate $\widehat k$ of the number of top components to be selected. \vspace{.2cm} \\ 
  {\bf Step 1} Obtain the SVD of $\bX = \widehat \bU \bSigma \widehat \bV^{\top}$ with singular values $\sigma_{1} \geq \dots \geq \sigma_{n}$.\\
  {\bf Step 2} For any $q\in(0, 1]$ define the following transform estimates and their derivatives
  \begin{align*}
    &\widehat \varphi(y; q) = \frac{q}{n - \rh}\cdot\sum_{j = \rh + 1}^{n} \frac{y}{y^2 - \sigma_{j}^2} + \frac{1-q}{y}, \quad  \quad \widehat \varphi'(y; q) = -\frac{q}{n - \rh}\cdot\sum_{j = \rh + 1}^{n} \frac{y^2 + \sigma^2_j}{(y^2 - \sigma_{j}^2)^2} - \frac{1-q}{y^2},\\
    & \quad \widehat D(y) = \widehat \varphi\big(y; 1\big) \cdot \widehat \varphi\left(y ; \frac{n}{m}\right), \quad  \text{and}  \quad \widehat D'(y) = \widehat \varphi'\big(y; 1\big) \cdot \widehat \varphi\left(y ; \frac{n}{m}\right) + \widehat \varphi\big(y; 1\big) \cdot \widehat \varphi'\left(y ; \frac{n}{m}\right)
  \end{align*}

  {\bf Step 3} Let $\widehat \cU^{(k)}$ (resp. $\widehat \cV^{(k)}$)be the span of the first $k$ left (resp. right) singular vectors, set $q = 1 $ (resp. $q = \frac{n}{m}$), and estimate the left FDR (resp. right) for different $k$'s
    \begin{equation*}
    \widehat{\FDR}_{\texttt{l}}(\widehat \cU^{(k)}) = 1 - \frac{2}{k}\sum_{i=1}^{\min\{k ,\rh\}} \frac{\widehat D(\sigma_i) \cdot \varphi\left(\sigma_i; q\right)}{\widehat{D}'(\sigma_i)}. 
  \end{equation*}
    {\bf Step 4} To control the left FDR (resp. right), output $$\widehat{k} = \max\left\{k \in [n] \,\Big|\, \widehat{\FDR}_{\texttt{l}}(k) \leq \alpha \right\}.$$
     To simultaneously control left and right FDR, use $\max\left\{\widehat{\FDR}_{\texttt{l}}(k), \widehat{\FDR}_{\texttt{r}}(k)\right\} \leq \alpha$ instead.
    
\end{algorithm}

To facilitate the analysis of this algorithm, we consider as before a sequence of growing-size problem instances $\left\{\bX_n = \bA_n + \bE_n \in \RR^{n \times m}\right\}$. The second dimension $m = m_n$ grows with $n$; we omit the subscript to avoid cluttering notation. We make the following two assumptions.

\begin{assumption}\label{ass:amodel1}
    The sequence $\left\{\bX_{n} = \bA_{n} + \bE_{n} \in \RR^{n \times m} \right\}$  satisfies the following for all $n$.
    \begin{itemize}[leftmargin=.4cm]
       \item[] (\textbf{Independence}) The distributions of $\bA_{n}$ and $\bE_{n}$ are independent.
      \item[] (\textbf{Spectrum of the low-rank signal}) There exists a constant $r \in \NN$ and deterministic real numbers $\theta_{1} \geq \dots \geq \theta_{r} > 0$ such that $\rank(\bA_{n})= r$ and $\sigma_i(\bA_n) = \theta_{i}$ for all $i\leq r$.
      \item[] (\textbf{Invariance}) For either $\bA_n$ or $\bE_n$, the distributions of left and right singular vectors are independent of each other and orthogonally invariant.
      \item[] (\textbf{Asymptotic size}) The ratio between the dimensions $n, m$ converges $n/m \rightarrow \vartheta$ with $\vartheta \in (0, 1].$ 
    \end{itemize}
\end{assumption}
  \begin{assumption}\label{ass:amodel2}
    The sequence $\left\{\bE_{n} \in \SSS^n \right\}$  satisfies the following.
    \begin{itemize}[leftmargin=.4cm]
      \item[] (\textbf{Spectrum}) The empirical singular value distribution $\mu_{\bE_{n}}$, defined as 
      \begin{equation}
          \label{eq:empirical-singular-as}
          \mu_{\bE_n} = \frac{1}{n} \sum_{i=1}^n \delta_{\sigma_i(\bE_n)}
      \end{equation}
      converges almost surely weakly to a deterministic compactly supported measure $\mu_{\bE}$ as $n\rightarrow \infty$.
      \item[] (\textbf{Edge convergence})  Let the infimum and supremum of $\supp(\mu_{\bE})$ be denoted by $a$ and $b$.  The smallest and largest singular values of $\bE_{n}$ converge almost surely to $a$ and to $b$, respectively.
      \item[] (\textbf{Decay at the edge}) The limiting distribution $\mu_{\bE}$ has a density $f_{\bE}$ such that as $t \rightarrow b$ with $t < b$, we have that $f_{\bE}(t) \sim c \cdot (b-t)^{\alpha}$ for some constant $c > 0$ and exponent $\alpha \in (0, 1]$. 
    \end{itemize}
    \end{assumption}

These assumptions closely parallel Assumptions~\ref{ass:model1} and \ref{ass:model2}. The only additional assumption that we impose is that the limiting ratio $n/m \as \vartheta \in (0, 1].$ Thus, our matrices are asymptotically wide; this is not a significant constraint as we can transform $\bX$ into $\bX^\top$ when we have a tall matrix.

In what follows, we will focus our discussion on controlling FDR for column spaces, i.e., the quantity $\FDR_{\texttt{l}}$, and we highlight the necessary changes for controlling $\FDR_{\texttt{r}}$. Under the preceding assumptions, the singular values also experience a BBP transition, and some of the singular values of $\bA_n$ might get subsumed by the noise.  To quantify this, we define a notion of observable rank $r^\star,$ akin to the case of eigenvalues, which in the present context reduces to
\begin{equation*}
    r^\star := \# \{i \in [r] \mid \theta_i > D_{\mu_{\bE}}(b^+)^{1/2}\}.
\end{equation*}
The function $D_{\mu_{\bE}}$ is known as the D-transform of $\mu_{\bE}$ and it is to the asymmetric case what the Cauchy transform is for the symmetric case.   In particular, for any $q\in (0, 1]$ let 
\begin{equation*}
 {\varphi}_{\mu_{\bE}}(z; q) := q \cdot \int \frac{z}{z^2 - t^2} d\mu(t) + \frac{1 - q}{z}
\end{equation*}
Then the $D$-transform of $\mu_{\bE}$ 
is
given by
\begin{equation*}
    D_{\mu_{\bE}}(x) = \varphi_{\mu_{\bE}}(x; 1) \cdot \varphi_{\mu_{\bE}}(x; \vartheta)
\end{equation*}
where $\vartheta$ is the limiting aspect ratio from Assumption~\ref{ass:amodel1}.
As in the symmetric case, we let 
$\bar \cU^{(n)} = \text{span}\{u_1, \dots, u_{ {r}^\star}\}$ be the subspace generated by top $r^\star$ components of $\bA_n$ and we define the upper-bounding FDR
\begin{equation} 
    \label{eq:upper-fdr-as}
\overline{\FDR}_{\texttt{l}}(\widehat \cU^{(n)}_k) = \EE \left[ \frac{\Tr \left( P_{\widehat{\cU}^{(n)}_{k}} P_{\bar{\cU}^{(n)^{\perp}}} \right)}{k}\right] \quad \text{and} \quad \FDR^{\infty}_{\texttt{l}}(k) =\lim_{n\rightarrow \infty } \overline{\FDR}_{\texttt{l}}(\widehat \cU^{(n)}_k).
\end{equation}
Controlling $\overline{\FDR}_{\texttt{l}}$ provides control on $\FDR_{\texttt{l}}$ as $\FDR_{\texttt{l}} \leq \overline{\FDR}_{\texttt{l}}.$
Moreover, one can obtain an asymptotic characterization of the upper-bounding FDR for any fixed $k\in \NN$, 
\begin{equation} \label{eq:key-idea-as}
\FDR_{\texttt{l}}^\infty(k) = \lim_{n\rightarrow\infty }1-\frac{2}{k}\sum_{i=1}^{k \wedge r^*} \frac{D_{\mu_{\bE}}\left(\sigma_i(\bX_n)\right) \cdot \varphi_{\mu_{\bE}}\left(\sigma_i(\bX_n); 1\right)}{D_{\mu_{\bE}}'(\sigma_i(\bX_n))}.\end{equation}
The quantity $\FDR^\infty_{\textit{r}}(k)$ is controlled by an analogous limit where we substitute $\varphi_{\mu_{\bE}}\left(\sigma_i(\bX_n); 1\right)$ by $\varphi_{\mu_{\bE}}\left(\sigma_i(\bX_n); \vartheta\right)$.  The expression \eqref{eq:key-idea-as} presents the key insight underlying \cref{alg:FDR-as}: the method employs the spectrum of $\bX_n$ and a rank estimate $\widehat r$ to estimate the right-hand-side of \eqref{eq:key-idea-as} and uses this estimate as a proxy for the true FDR.

The following is our main result concerning columnspace FDR control; rowspace FDR control can be obtained with straightforward modifications.

\begin{theorem}[\textbf{FDR control for asymmetric matrices}]\label{thm:main-as}
Consider a sequence of growing-size problem instances $\left\{\bX_n = \bA_n + \bE_n \in \RR^{n \times m}\right\}$ that satisfy Assumptions~\ref{ass:amodel1} and~\ref{ass:amodel2}. Further, suppose the rank estimate $\widehat r$ required for \cref{alg:FDR-as} is such that $$r^\star \leq \widehat r \leq r$$ almost surely for all large $n$.  For a generic $\alpha \in (0,1)$ and for all large $n$, the output of \cref{alg:FDR-as} satisfies $\FDR_{\texttt{l}}\left(\widehat \cU^{(n)}_k\right) \leq \alpha$, which in turn implies that $\FDR_{\texttt{l}}\left(\widehat \cU^{(n)}_k\right)\leq \alpha$. Moreover, if $\alpha \leq \FDR^{\infty}_{\texttt{l}}(r)$, the output $\widehat k$ of \cref{alg:FDR-as} satisfies 
  $$\widehat k = \max\{k \in [r] \mid \overline{\FDR}_{\texttt{l}}\left(\widehat \cU^{(n)}_k\right) \leq \alpha\}.$$
\end{theorem}

The proof of this theorem is given in Section~\ref{sec:proof-main-as}.

\subsection{Rank Estimation}

The extension of our rank estimator for the asymmetric case is direct -- we substitute eigenvalues with singular values. 
\begin{algorithm}[h!]
  \caption{\textsc{RankEstimate} for asymmetric matrices} %
  \label{alg:rank-estimate-as}
  {\bf Input}: Nonincreasing spectrum $\sigma(\bX)$, %
  {\it optional} $p > 0$ (default to $p = 1$).\\
{\bf Output}: Estimate of the identifiable rank $\widehat r$.\vspace{.2cm}\\
{\bf Step 1} For each $j  \leq n-1$ compute the eigenvalue spacing
$$ \Delta_{j} \leftarrow {\sigma_{j-1}(\bX) - \sigma_{j}(\bX)}.$$
{\bf Step 2} 
Compute
$$
\widehat r \leftarrow \max\left\{j \in \left[ n/2 \right] \,\Big| \, \Delta_{j+1} > p \cdot n^{-1/2}\right\}.
$$
\end{algorithm}
The method relies on an equivalent version of Assumption~\ref{ass:spacings}.
   \begin{assumption}\label{ass:spacings-as}
   The following conditions hold for the sequence $\bE_n$. 
   \begin{enumerate}[leftmargin=.4cm]
       \item[] (\textbf{Connected support}) The support of the asymptotic distribution $\mu_{\bE}$ is connected.
       \item[] (\textbf{Singular value spacings}) Let $s_{j}^{(n)} = \sigma_{j}(\bE_{n}) - \sigma_{j+1}(\bE_{n})$, for all large $n$ and all $i \leq n,$
     $$
     s_{i}^{(n)} = o \left( n^{-1/2}\right) \qquad\text{almost surely.}
     $$
   \end{enumerate}
\end{assumption}

Next, we obtain the asymmetric counterpart of Theorem~\ref{thm:rank-estimate}.
\begin{theorem}[\textbf{Rank estimate}]\label{thm:rank-estimate-as} Consider a sequence of problem instances $\{\bX_n = \bA_n + \bE_n ~|~ \bA_n \in \SSS^n_+, ~ \bE_n \in \SSS^n\}$ for which Assumptions \ref{ass:amodel1}, \ref{ass:amodel2}, and \ref{ass:spacings-as} hold.  Then, the output of~\cref{alg:rank-estimate-as}, $\textsc{RankEstimate}(\bX_{n}),$  almost surely satisfies 
  $$
  r^\star \leq \textsc{RankEstimate}(\bX_{n}) \leq r \qquad \text{for all large enough }n.
  $$
  \end{theorem}

The proof of this theorem is given in Section~\ref{sec:proof-rank-estimate-as}.

\subsection{Asymmetric Matrix Ensemble satisfying Assumptions~\ref{ass:amodel2} and~\ref{ass:spacings-as}}

We conclude this section with an example of a matrix ensemble that satisfies our assumptions. 
\begin{definition}\label{def:data}
A \emph{Wishart-factor ensemble} is a sequence of 
$\{\bY \in \RR^{n \times m}\}$ with $n \leq m$ for which the following hold.
    \begin{enumerate}
        \item (\textbf{Independence}) The entries 
        $\{\bY_{ij} \mid 1 \leq i \leq n, 1 \leq j \leq m\}$ are independent. 
        \item (\textbf{Mean and variance}) For all $n$ and all $i,j \leq n$ we have $\EE \bY_{ij} = 0$ and $\EE |\bY_{ij}|^2 = \frac{1}{\sqrt{m}}.$
        \item (\textbf{Tail decay}) There exists a constant $c > 0$ such that for all $n$ and all $i, j\leq n$ we have that $\PP(|\bY_{ij}| \geq t) \leq c \exp\left(-(t\sqrt{m})^{c}\right)$.
        \item (\textbf{Limiting size}) The asymptotic ratio of the matrix dimensions $\vartheta := \lim_{n \rightarrow \infty } n/m$ satisfies $$ 0< \vartheta  \leq 1. $$
    \end{enumerate}
\end{definition}
These matrices simply correspond to the factors in the definition of a Wishart ensemble. The next proposition shows that Wishart-factor ensembles satisfy the desired assumptions.
\begin{proposition}[\textbf{Wishart factors}] \label{prop:data}
  Assume that the matrices $\{\bE_{n}\}$ form a Wishart-factor ensemble with $\vartheta \neq 1$ and that their singular vector distribution is orthogonally invariant. Then, Assumptions~\ref{ass:amodel2} and~\ref{ass:spacings-as} hold.
\end{proposition}

The proof of this result is given in Section~\ref{sec:proof-data}.

\section{Numerical Experiments} \label{sec:experiments}

In this section, we present the results experiments that demonstrate the utility of our methods.  The first set of experiments is focused on synthetically-generated data while the second set concerns datasets from RNA sequencing and hyperspectral imaging.  All the code used to generate the experimental results and figures in this paper can be found at
\begin{quote}\centering
\url{https://github.com/mateodd25/FDRSubspaces.jl}.\end{quote}
 In these experiments, we employ a data-driven approach for selecting the parameter $p$ in Algorithm~\ref{alg:rank-estimate} by setting it to 
$\frac{3}{5} \cdot \texttt{median}\left( \left\{ \Delta_j \right\}^{n-1}_{j=1}\right) \cdot n$.
This choice aims to provide a robust estimate of noise scaling (although any positive constant $p$ is in theory adequate under Assumption~\ref{ass:spacings}).

\subsection{Synthetic Data} \label{sec:generated}

We generate matrices of the form $\bX = \bA + \bE$ with a variety of ensembles for the noise $\bE$ and spectral models for the low-rank signal matrices $\bA$, which we now summarize.

\paragraph{Noise.} We consider seven noise ensembles; see Table~\ref{table:noise}.
These are generated as a product of a matrix drawn from a GOE and another matrix for which we vary the distribution as follows:
\begin{equation}\label{eq:correlated-matrix}\bF = \bG \bW.\end{equation}
Here $\bG \in \RR^{n \times m}$ has identically distributed independent entries with $\bG_{ij} \sim N(0, 1/n)$ and $\bW \in \RR^{{m \times m}}$ has a different distribution depending on the problem, thus leading to matrices with correlated entries.  The noise matrix $\bE$ is derived from $\bF$ in one of several ways; see Table~\ref{table:noise}.  The noise models in Table~\ref{table:noise} are known to satisfy Assumption~\ref{ass:model2} and~\ref{ass:amodel2} in all cases, see for example \cite{donoho2023screenot}.  However, to the best of our knowledge, it is not known whether the last four models in Table~\ref{table:noise} have a rigid spectrum in the sense of Assumptions~\ref{ass:spacings} and~\ref{ass:spacings-as}. Nonetheless, we observe that our methods perform reasonably well for all of them.

\begin{table}[t]
	\renewcommand{\arraystretch}{1.3}
	\begin{center}
		\begin{tabular}{lllll}
			\toprule
			\textbf{Distribution}
			                       & \textbf{Symmetric}
			                       & \textbf{Template for} $\bE$                                                                       %
			                       & \textbf{Distribution of} $\bW$
			\\
			\midrule
			\texttt{Wigner}        & Yes                            & $(\bF + \bF^{\top})/2$ & Deterministic $\bW =\bI$                   \\
			\texttt{Wishart}       & Yes                            & $\bF \bF^{\top}$       & Deterministic $\bW = \bI$                  \\
			\texttt{WishartFactor} & No                             & $\bF$                  & Deterministic $\bW = \bI$                  \\
			\texttt{Fisher}        & Yes                            & $(\bF + \bF^{\top})/2$ & $\bW \sim \texttt{Wishart}$                \\
			\texttt{FisherFactor}  & No                             & $\bF$                  & $\bW \sim \texttt{Wishart}$                \\
			\texttt{Uniform}       & Yes                            & $(\bF + \bF^{\top})/2$ & Diagonal $\bW_{ii} = \textbf{Unif}([0,1])$ \\
			\texttt{UniformFactor} & No                             & $\bF$                  & Diagonal $\bW_{ii} = \textbf{Unif}([0,1])$ \\
			\bottomrule
		\end{tabular}
	\end{center}
	\caption{Noise distributions: the matrices $\bF$ is derived from $\bW$ as described by \eqref{eq:correlated-matrix}.}\label{table:noise}
\end{table}

\paragraph{Low-rank signal.} For each noise model from Table~\ref{table:noise}, we consider three types of low-rank signal models: \texttt{well-separated}, \texttt{barely-separated}, and \texttt{entangled}.  Each of these matrices has rank $20$, and their spectrum is given as follows:
\begin{equation}\label{eq:exp-spectrum}
	\theta_{i} = {\texttt{BBP}} + \texttt{shift} + 10 \cdot 1.3^{1-i}\quad \text{for all }i \in \{1,\dots, 20\}.
\end{equation}
The $\theta_i$'s correspond to eigenvalues for symmetric problems and to singular values for asymmetric problems.  In the preceding expression, $\texttt{BBP}$ is an upper estimate of the BBP transition point that changes depending on the noise distribution, \texttt{shift} changes depending on the type of low-rank signal as follows: %
\begin{equation}
	\label{eq:shift}
	\texttt{shift} := \begin{cases}
		1.0                 & \text{ if generating a \texttt{well-separated} instance,}  \\
		0.0                 & \text{ if generating a \texttt{barely-separated} instance,} \\
		- 10 \cdot 1.3^{-9} & \text{ if generating a \texttt{entangled} instance.}
	\end{cases}
\end{equation}
Intuitively, well-separated instances correspond to low-rank signals having spectra that are easier to distinguish from the noise, barely-separated instances to low-rank signals having spectra that are barely distinguishable from the noise, and entangled instances to low-rank signals having some spectral components that are indistinguishable from the noise.  With the above selection of the $\texttt{shift}$ parameter, the observable rank $r^{\star}$ is equal to $20$ (which is the rank $r$ of the low-rank signal) for the first two instances and to $10$ for entangled instances.

 \begin{figure}[t!]
    \def\arraystretch{0}%
    \setlength{\imagewidth}{\dimexpr \textwidth - 4\tabcolsep}%
    \divide \imagewidth by 3
    \hspace*{\dimexpr -\baselineskip - 2\tabcolsep}%
    \begin{tabular}{@{}cIII@{}}
      \stepcounter{imagerow}\raisebox{0.35\imagewidth}{\rotatebox[origin=c]{90}%
        {\strut \texttt{well-separated}}} &
      \includegraphics[width=\imagewidth]{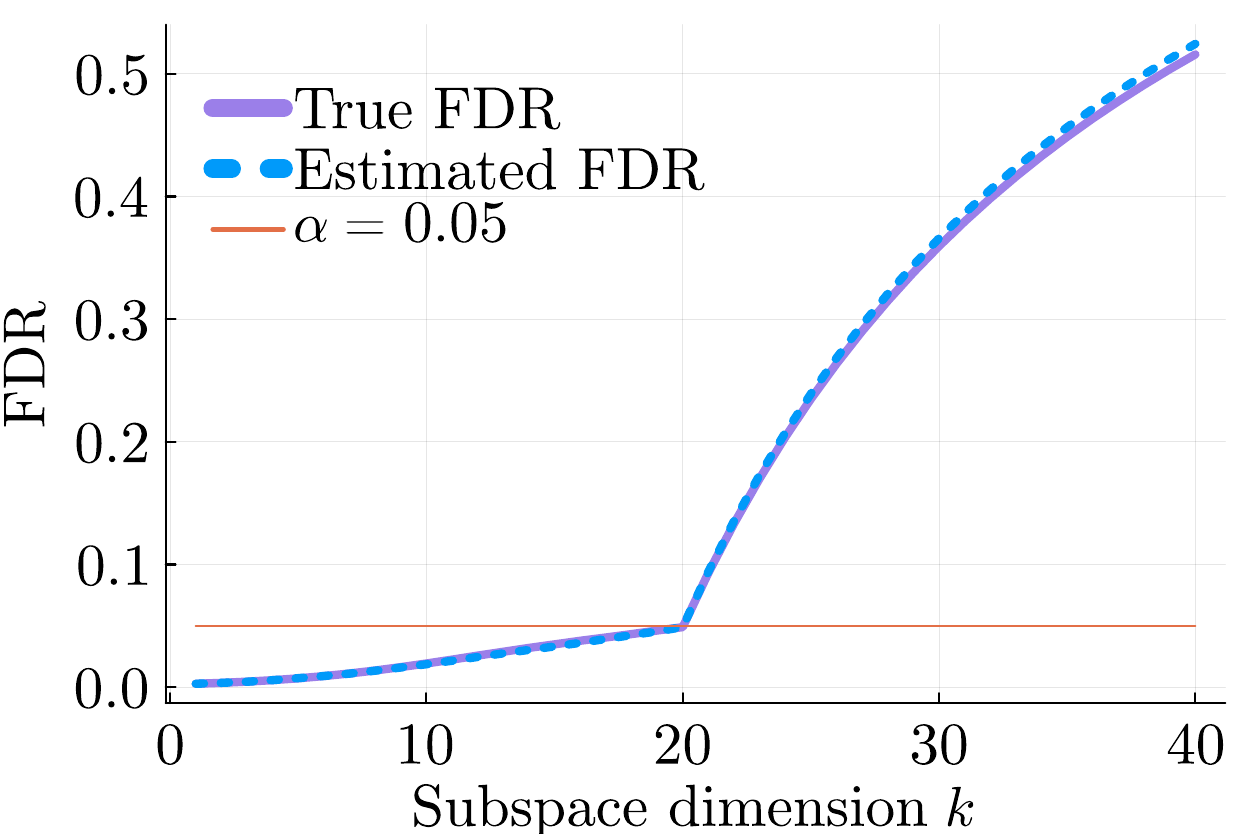} &
      \includegraphics[width=\imagewidth]{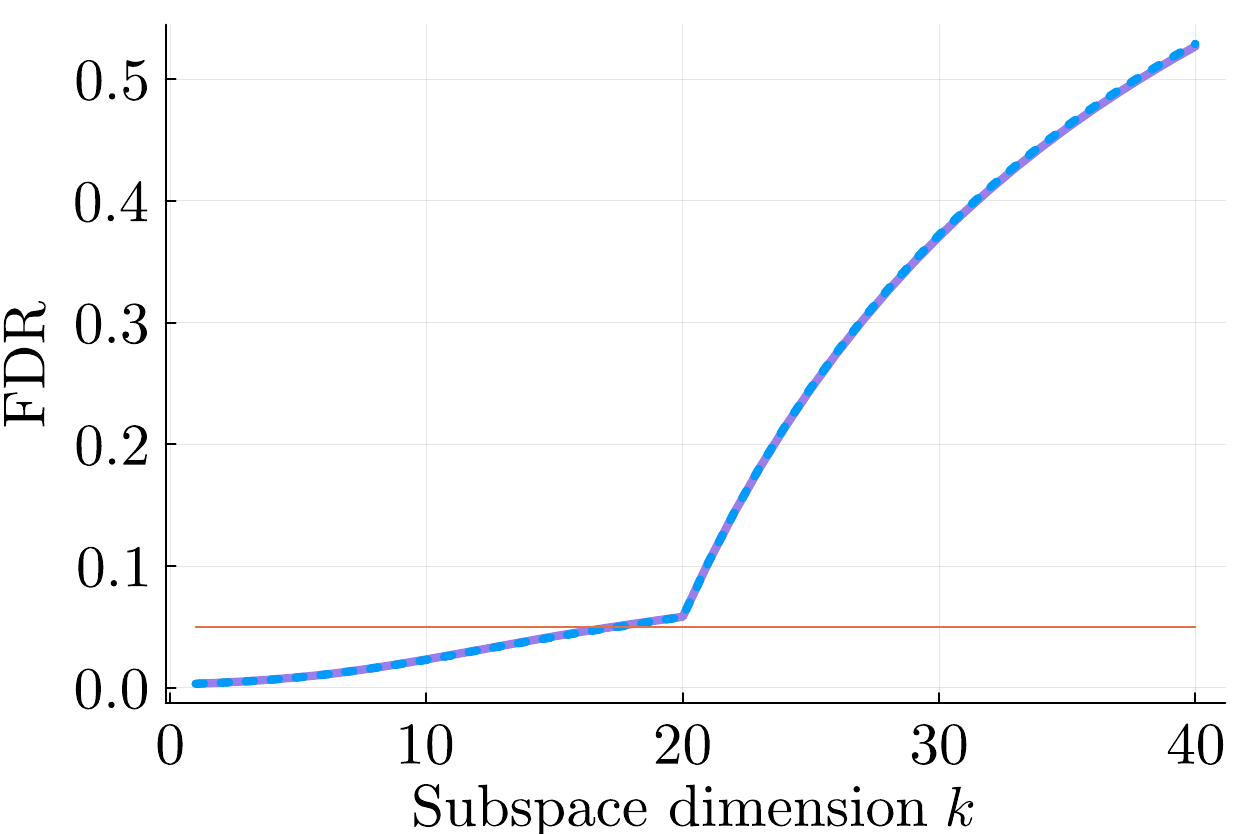}\label{example} &
      \includegraphics[width=\imagewidth]{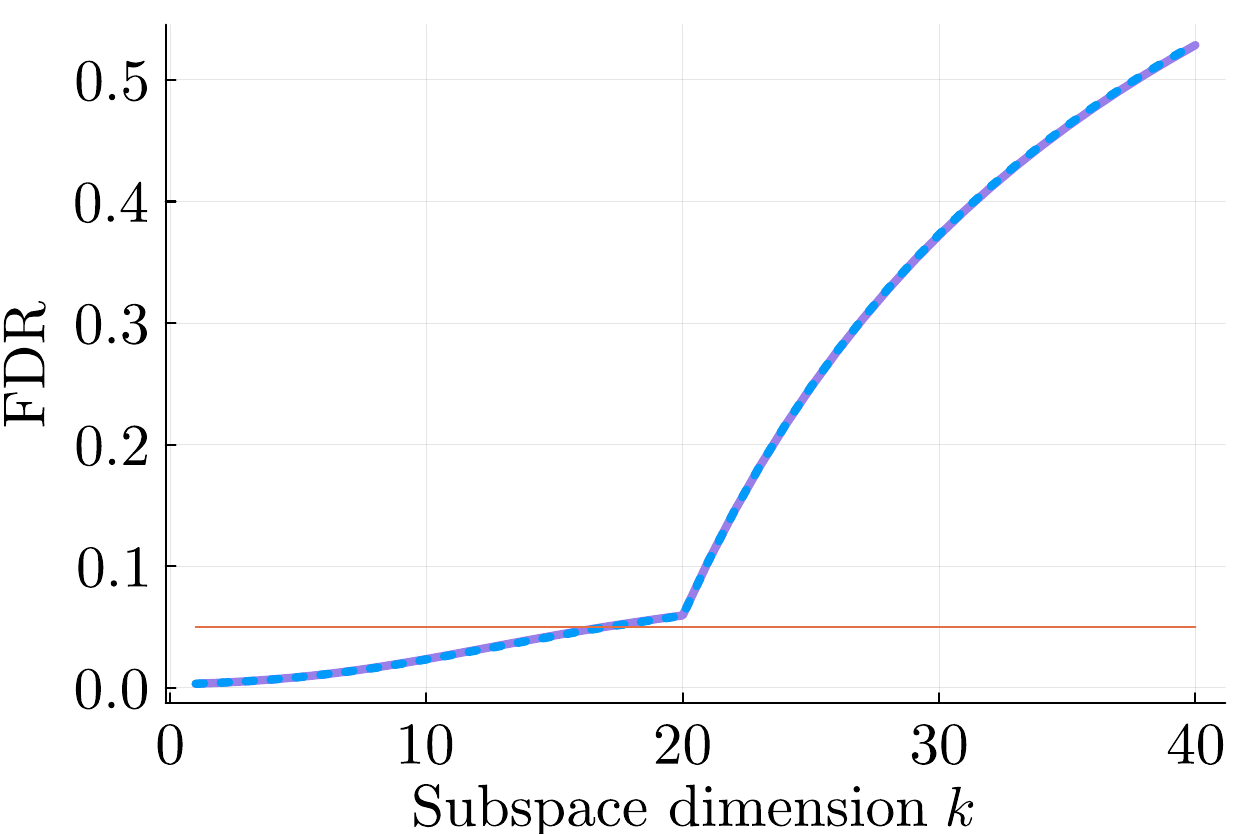} \\[2\tabcolsep]%
      \stepcounter{imagerow}\raisebox{0.35\imagewidth}{\rotatebox[origin=c]{90}%
        {\strut \texttt{barely-separated}}} &
      \includegraphics[width=\imagewidth]{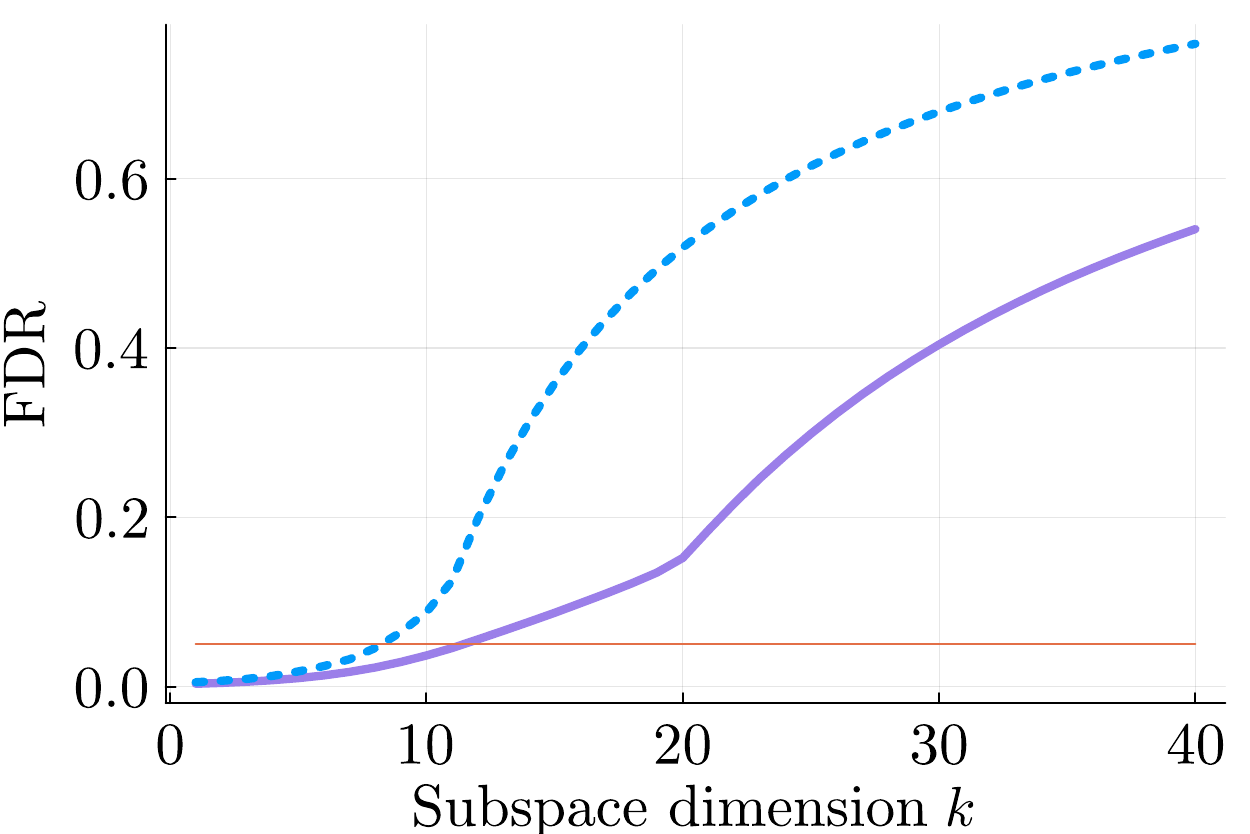} &
      \includegraphics[width=\imagewidth]{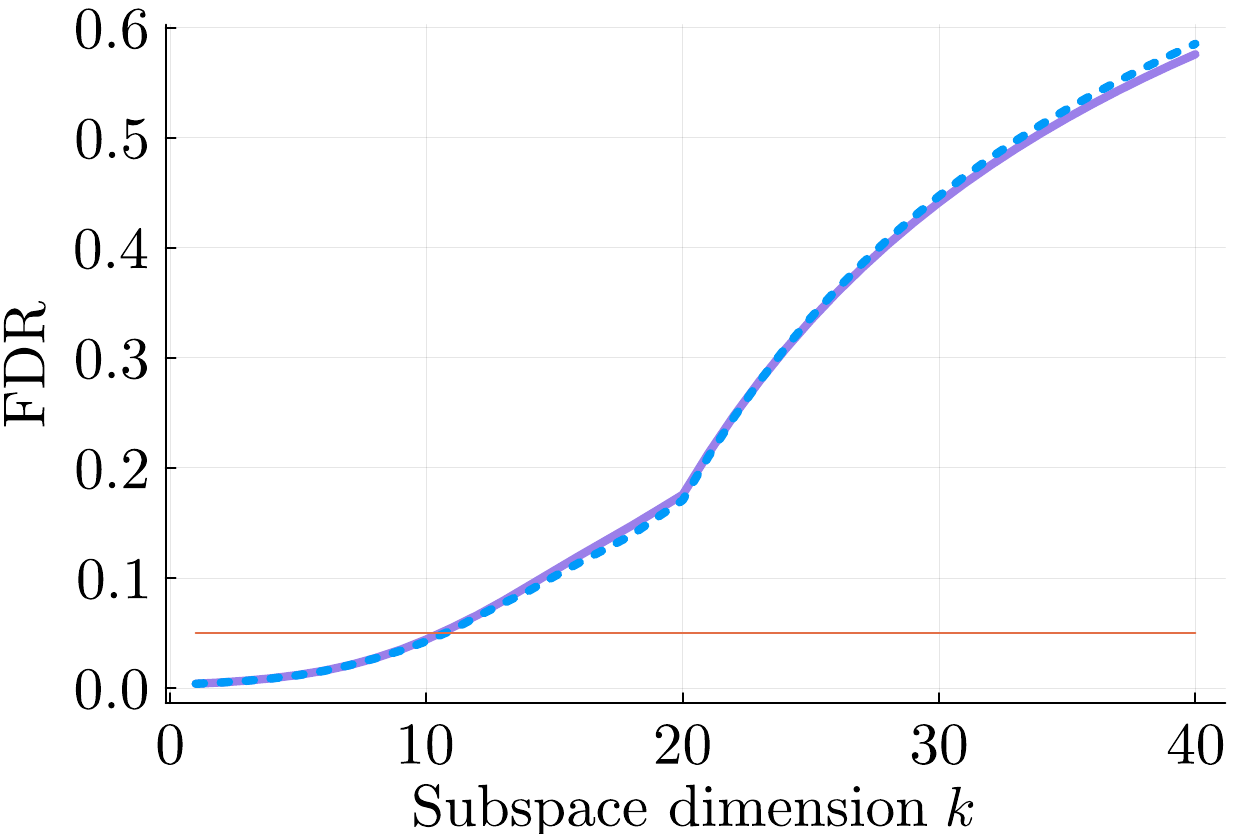} &
      \includegraphics[width=\imagewidth]{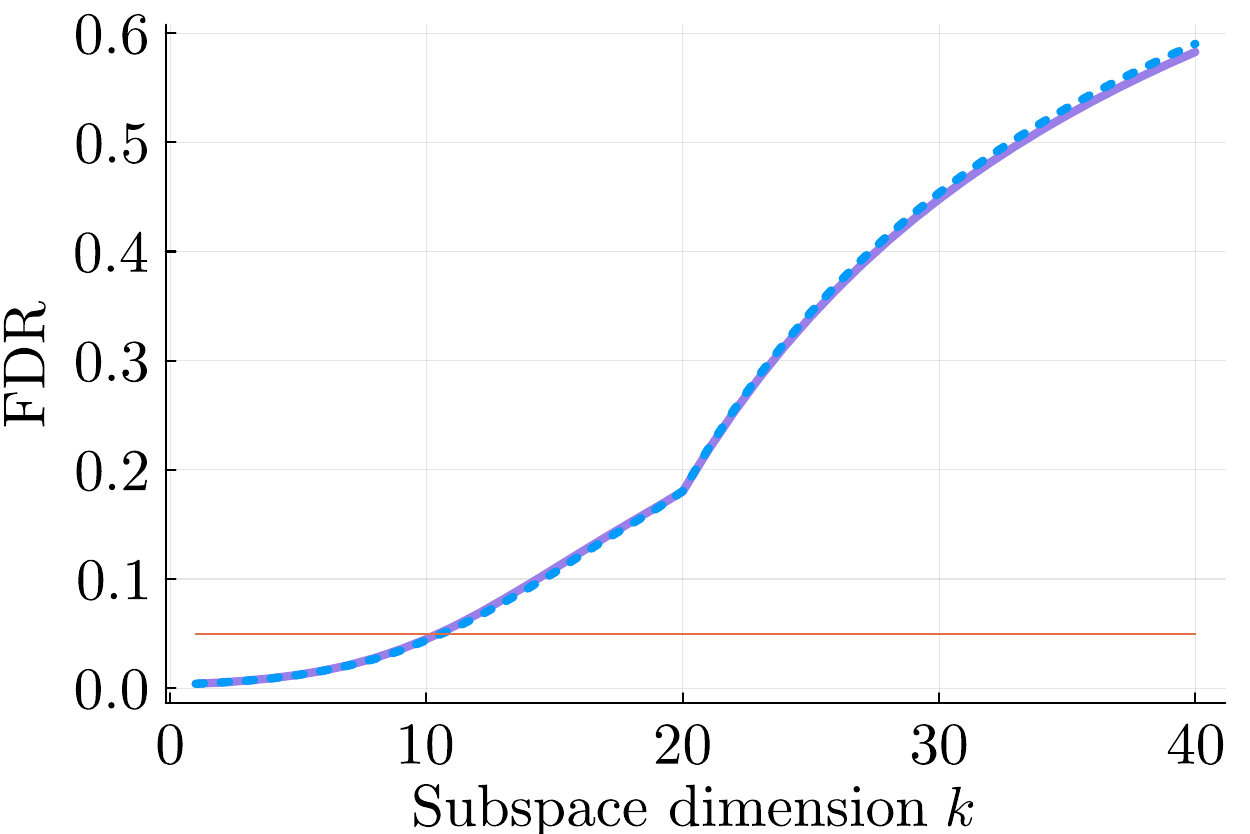} \\[2\tabcolsep]
      \stepcounter{imagerow}\raisebox{0.35\imagewidth}{\rotatebox[origin=c]{90}%
     {\strut \texttt{entangled}}} &
      \includegraphics[width=\imagewidth]{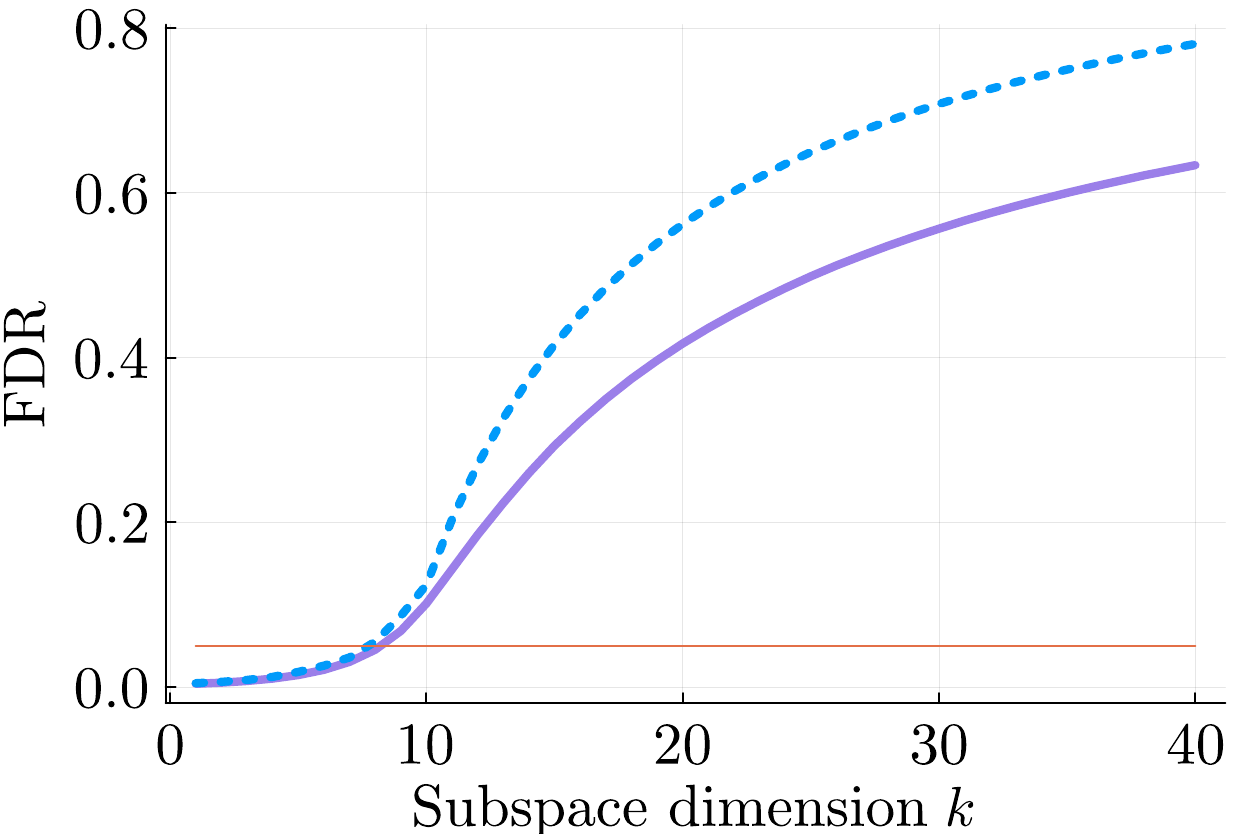} &
      \includegraphics[width=\imagewidth]{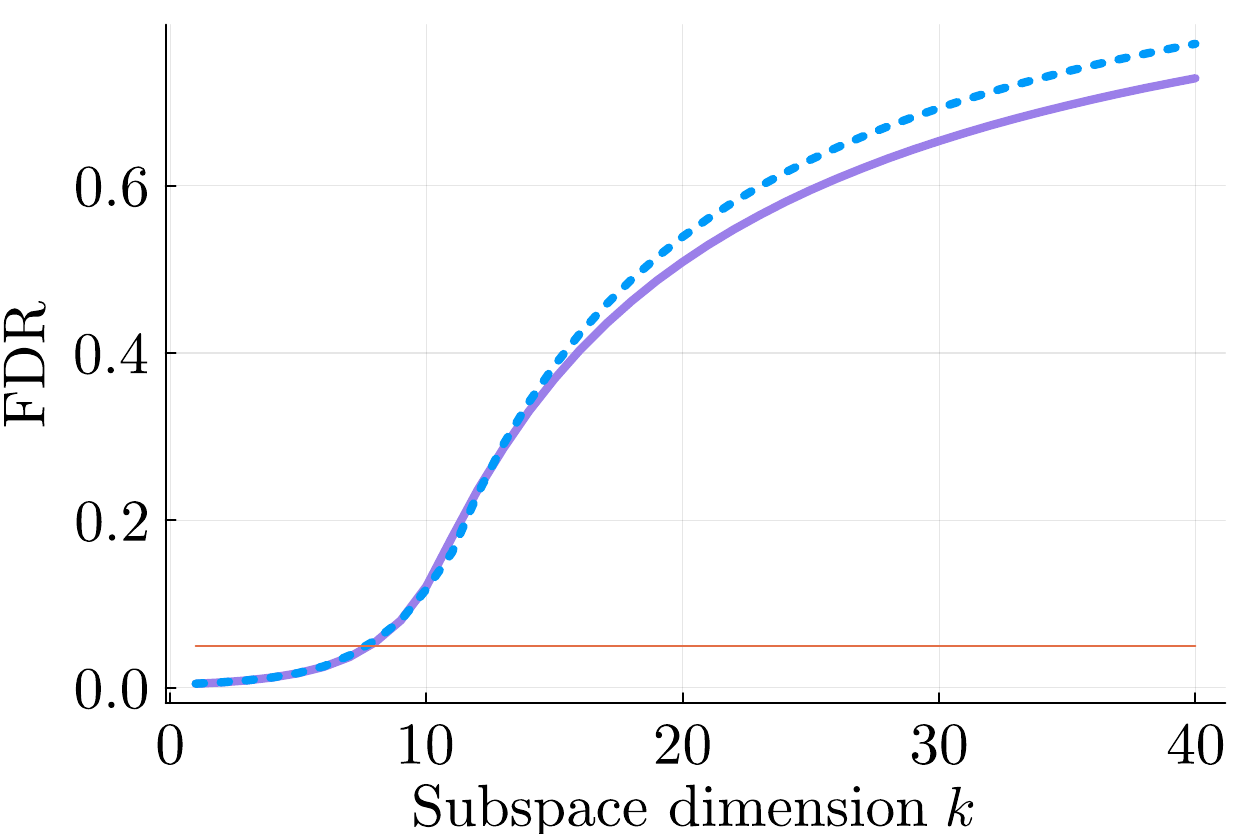} &
      \includegraphics[width=\imagewidth]{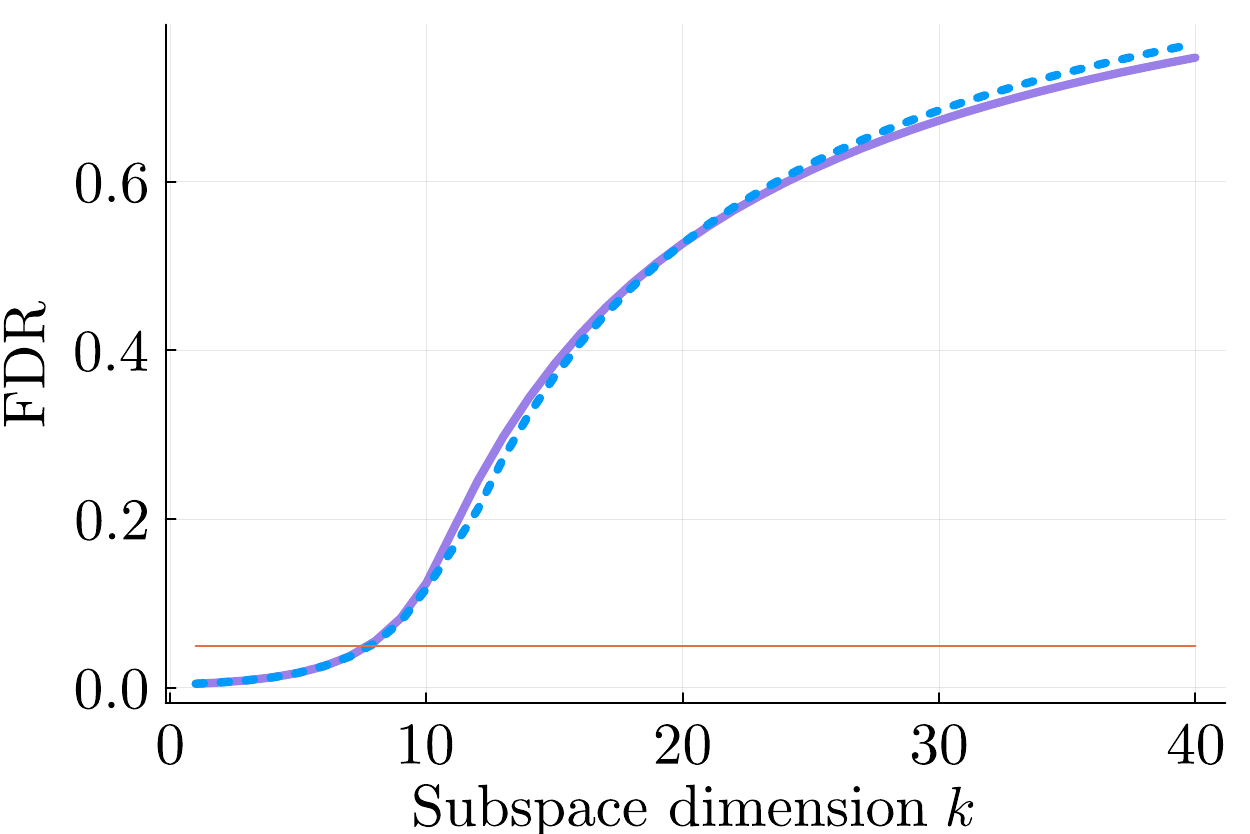} \\[2\tabcolsep]
        \setcounter{imagecolumn}{0} &%
        Dimension $n = 100$ &
        Dimension $n = 500$ &
        Dimenson $n = 1000$
    \end{tabular}
    \caption{Results on FDR estimation for the \texttt{Wigner} ensemble using Algorithms~\ref{alg:FDR} and~\ref{alg:rank-estimate}: the columns correspond to problem sizes $n= 100, 500, 1000$ and the rows correspond to the three models for the spectrum of the low-rank signal from~\eqref{eq:exp-spectrum}.}
	\label{fig:res-wigner}
  \end{figure}

We report results for two noise models in this section, namely the \texttt{Wigner} (symmetric) and the \texttt{WishartFactor} (asymmetric) ensembles.  The conclusions are analogous in the remaining cases and therefore we defer their description to Appendix~\ref{appendix:experiments}.  The results for the \texttt{Wigner} and the \texttt{WishartFactor} cases are given in Figures~\ref{fig:res-wigner} and~\ref{fig:res-wishart-factor}, which plot the true FDR and the estimated FDR against the subspace dimension $k$.  The `true FDR' is obtained using a Monte Carlo method with $100$ repetitions.  The `estimated FDR' is computed using Algorithms~\ref{alg:FDR} and \ref{alg:rank-estimate} for the symmetric case and Algorithms~\ref{alg:FDR-as} and \ref{alg:rank-estimate-as} for the asymmetric case.  

We make two observations regarding our results.  First, our FDR estimates tend to conservative by virtue of being larger than the true FDR; this is a consequence of our rank estimation method yielding smaller values than the actual rank.  This is a feature of both the \texttt{barely-separated} and the \texttt{entangled} cases, particularly for the smallest problem dimension (corresponding to $n = 100$).  Nonetheless, our approach reliably produces subspace estimates for which the false discovery rate is controlled at the desired level.  Second, we observe that the estimated FDR converges rapidly to the true FDR despite the asymptotic nature of our theoretical results, which suggests that our method remains practical even for modest problem dimensions $n$.  In particular, this convergence is notably quicker for smaller values of $k$, which are the cases of most practical interest in obtaining FDR control below a desired threshold.

 \begin{figure}[t!]
    \def\arraystretch{0}%
    \setlength{\imagewidth}{\dimexpr \textwidth - 4\tabcolsep}%
    \divide \imagewidth by 3
    \hspace*{\dimexpr -\baselineskip - 2\tabcolsep}%
    \begin{tabular}{@{}cIII@{}}
      \stepcounter{imagerow}\raisebox{0.35\imagewidth}{\rotatebox[origin=c]{90}%
        {\strut \texttt{well-separated}}} &
      \includegraphics[width=\imagewidth]{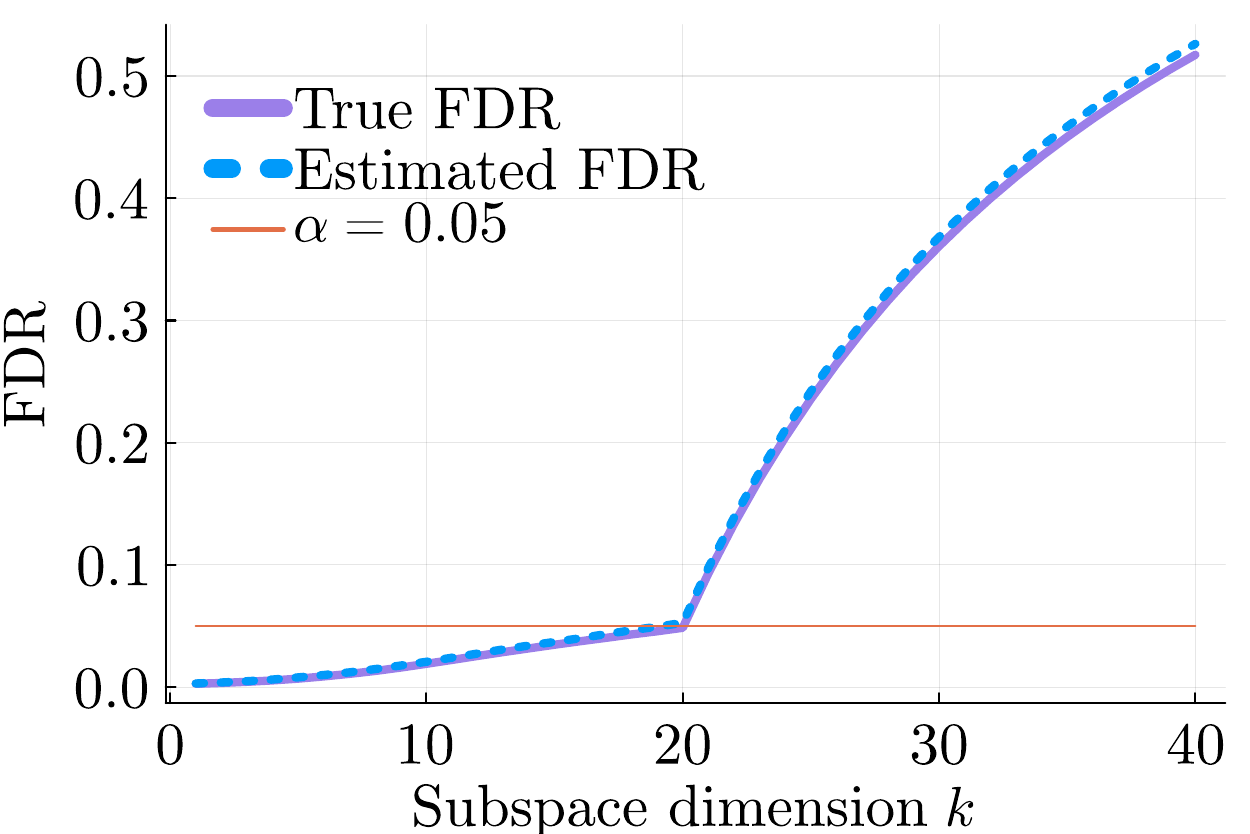} &
      \includegraphics[width=\imagewidth]{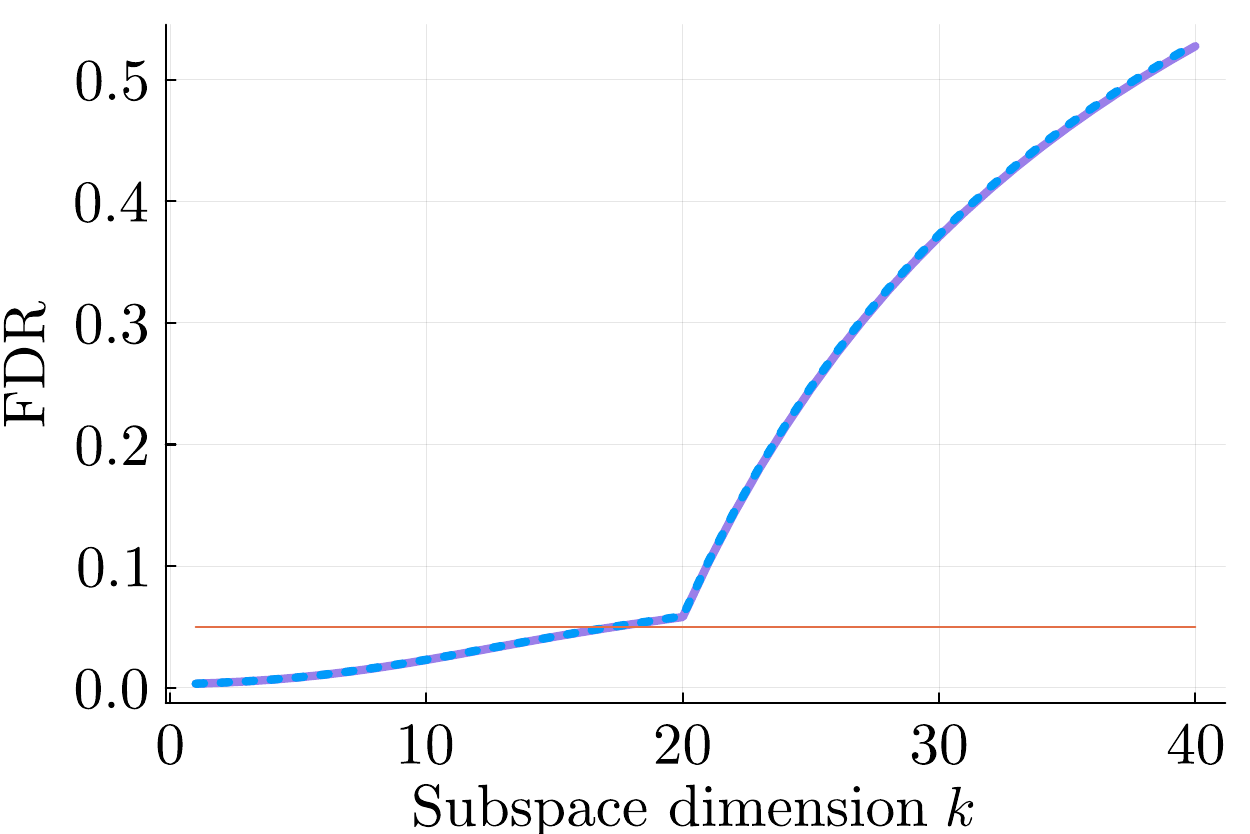}\label{example} &
      \includegraphics[width=\imagewidth]{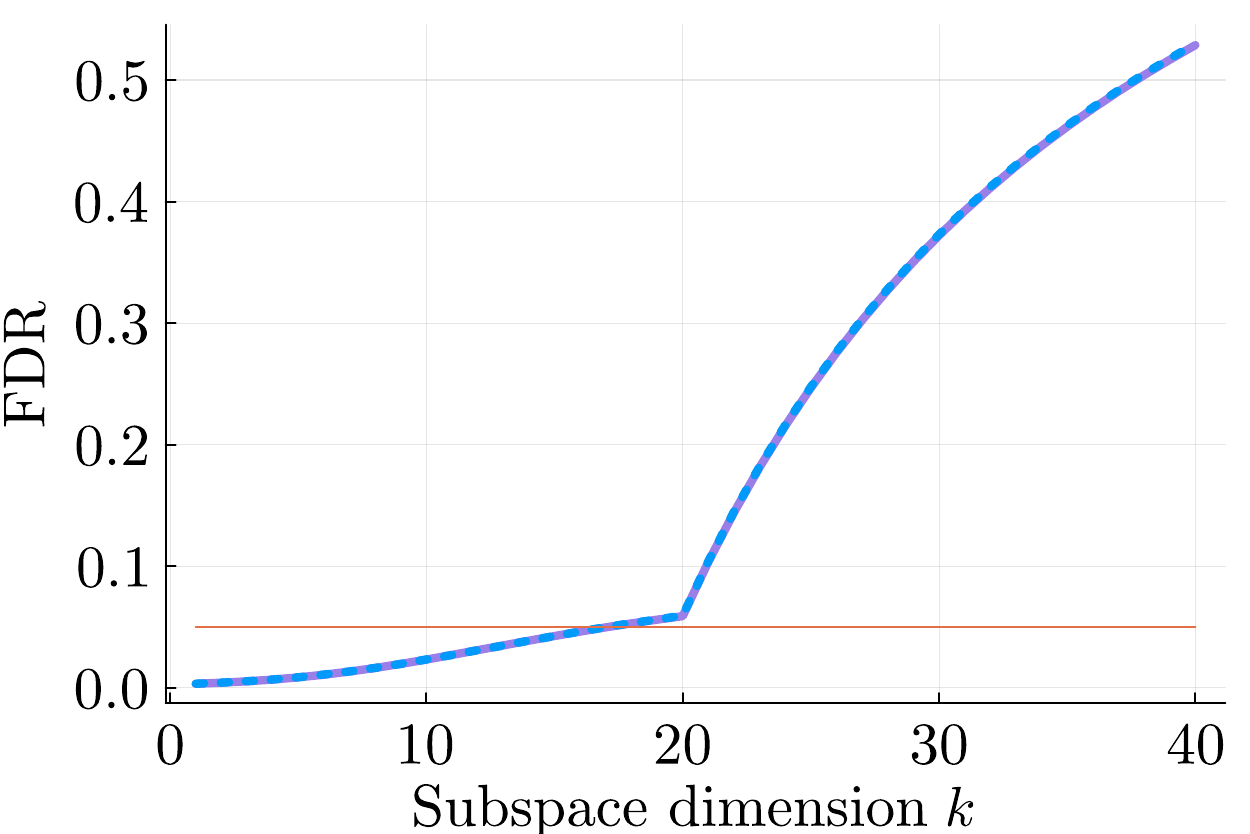} \\[2\tabcolsep]%
      \stepcounter{imagerow}\raisebox{0.35\imagewidth}{\rotatebox[origin=c]{90}%
        {\strut \texttt{barely-separated}}} &
      \includegraphics[width=\imagewidth]{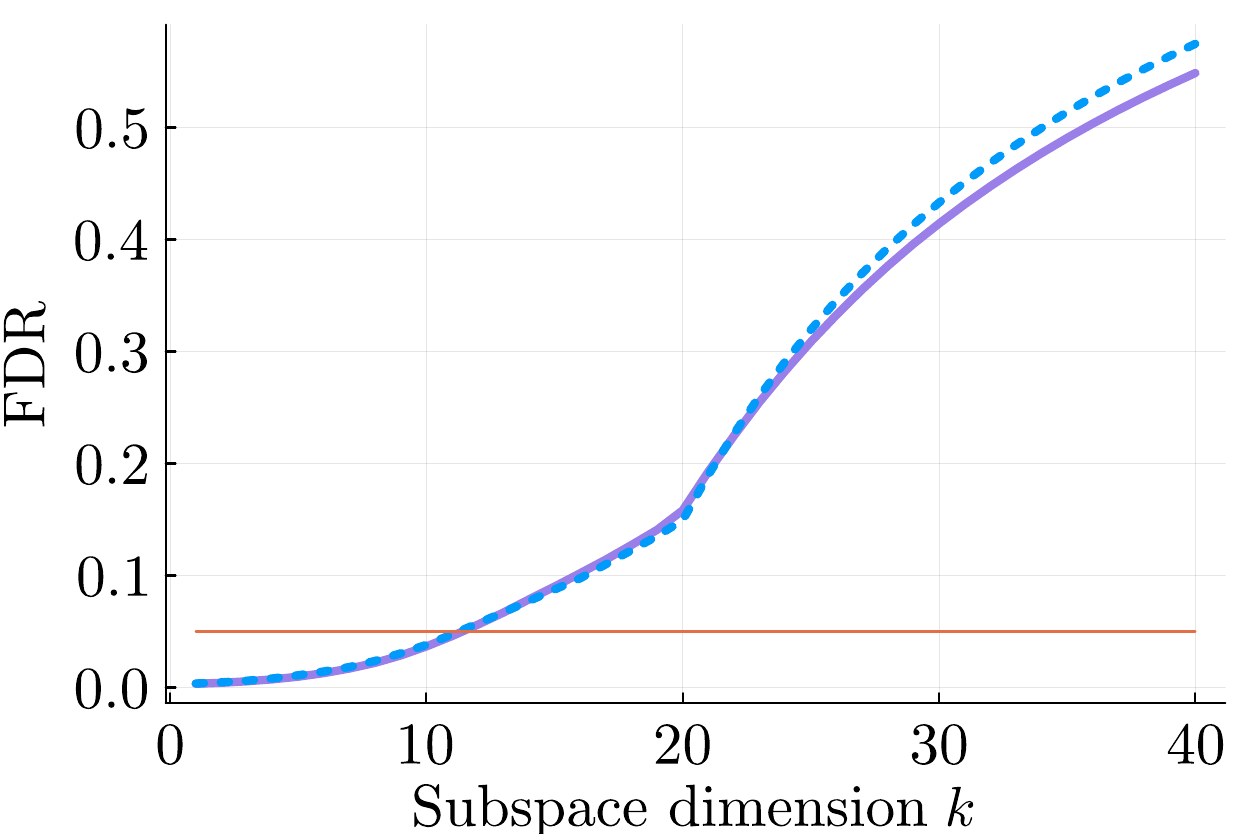} &
      \includegraphics[width=\imagewidth]{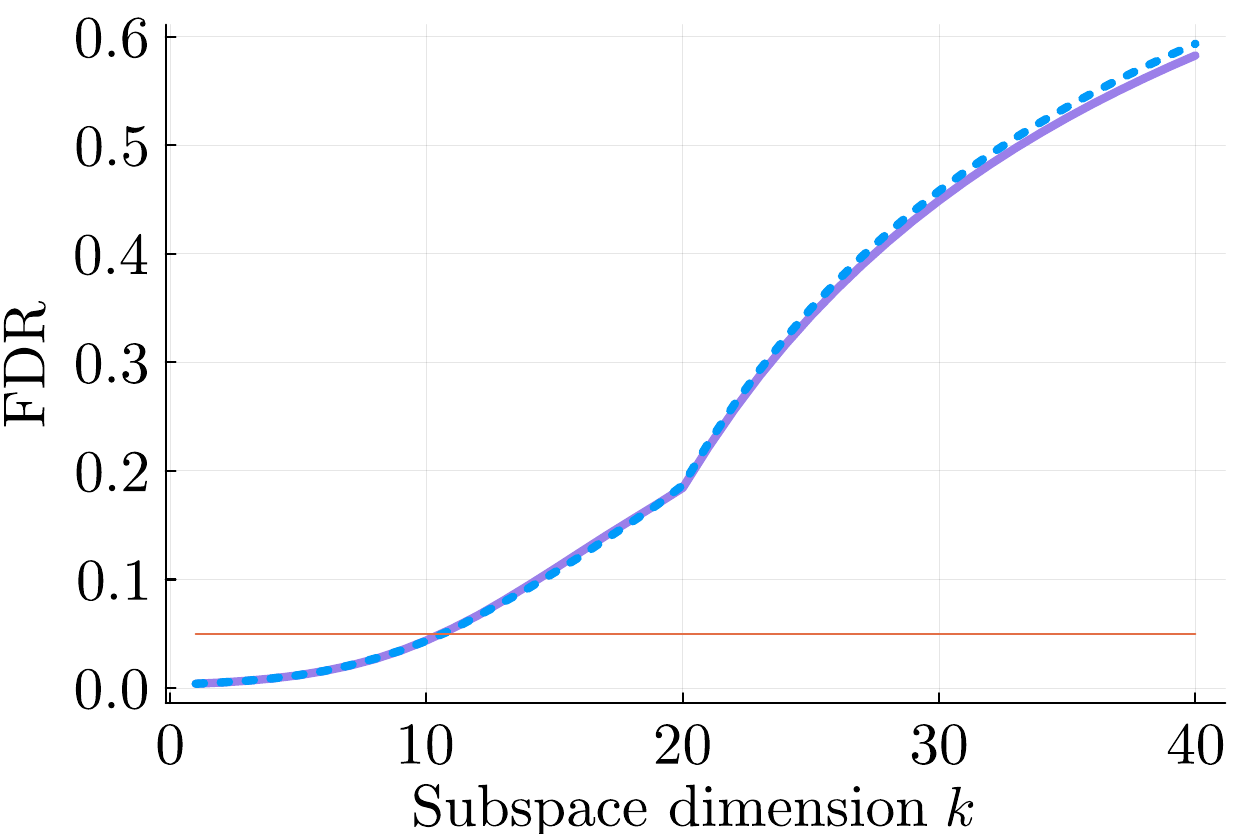} &
      \includegraphics[width=\imagewidth]{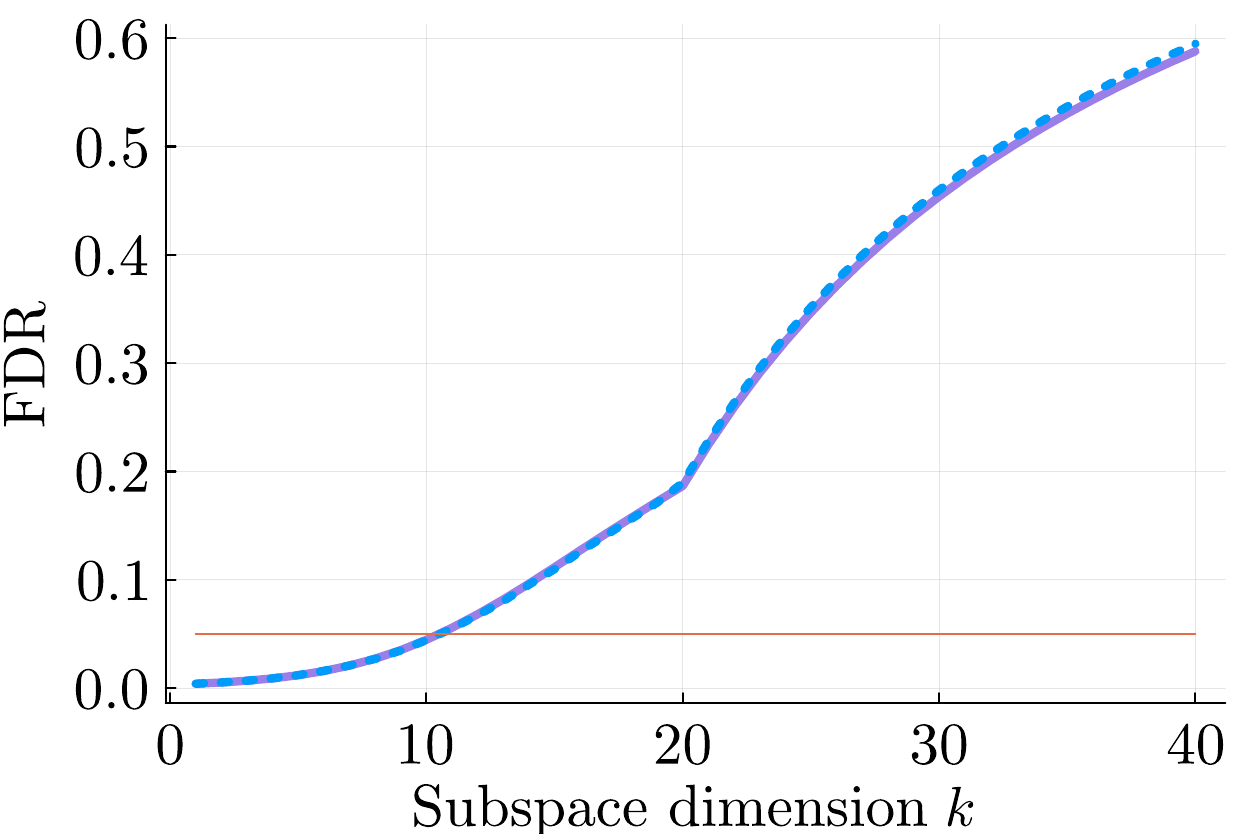} \\[2\tabcolsep]
      \stepcounter{imagerow}\raisebox{0.35\imagewidth}{\rotatebox[origin=c]{90}%
     {\strut \texttt{entangled}}} &
      \includegraphics[width=\imagewidth]{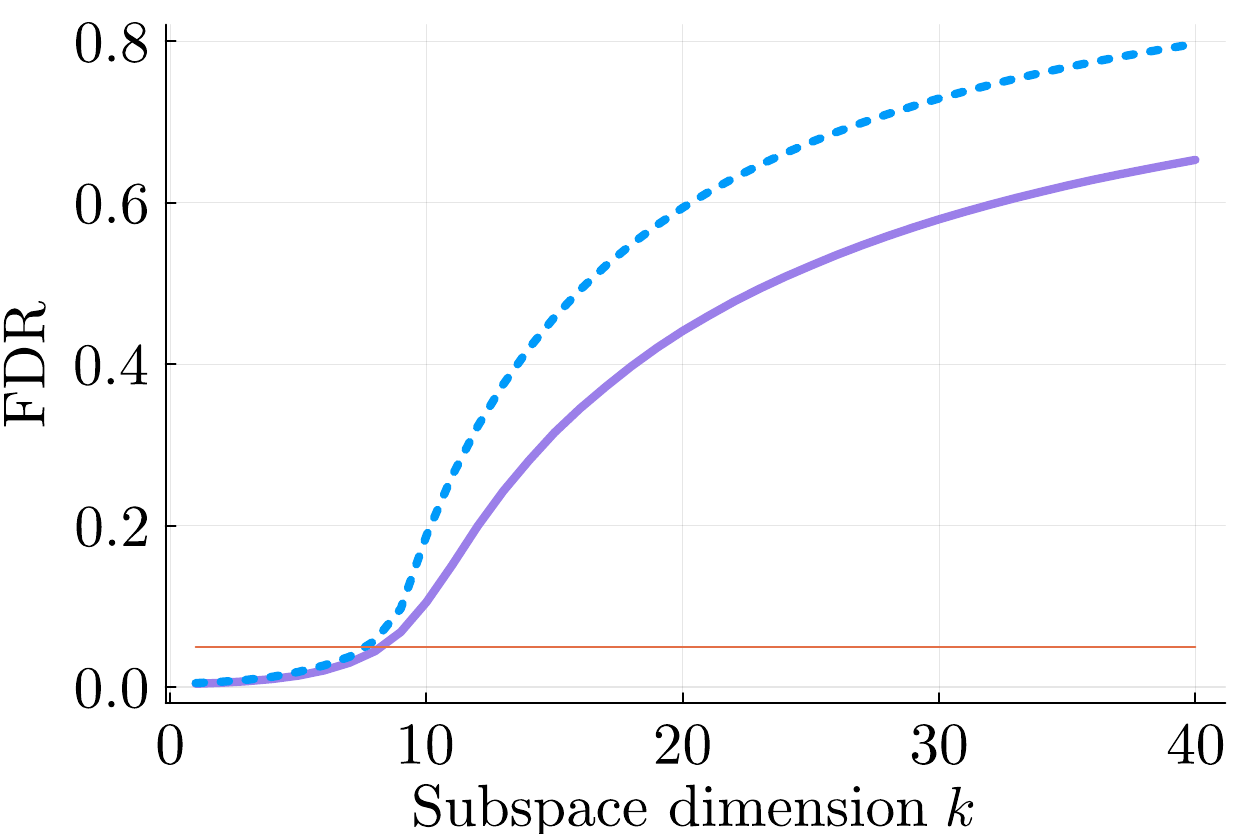} &
      \includegraphics[width=\imagewidth]{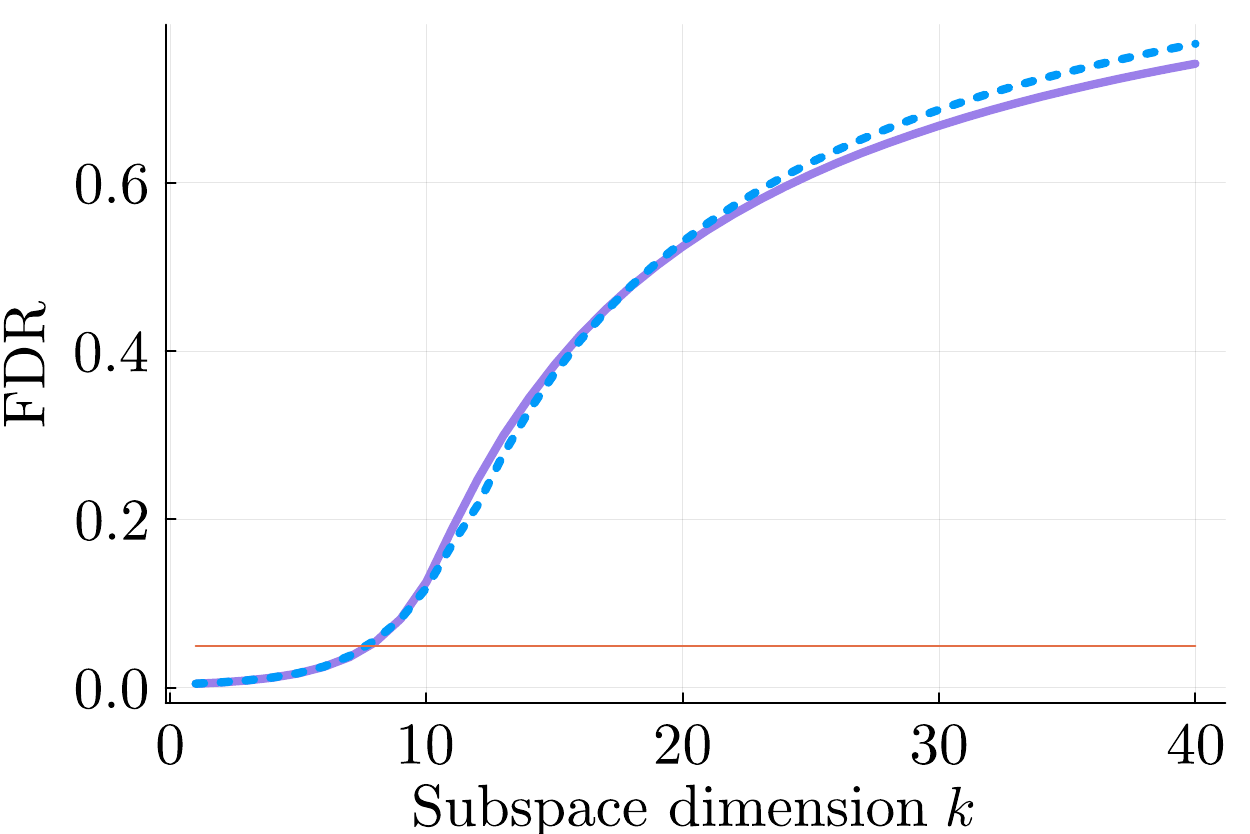} &
      \includegraphics[width=\imagewidth]{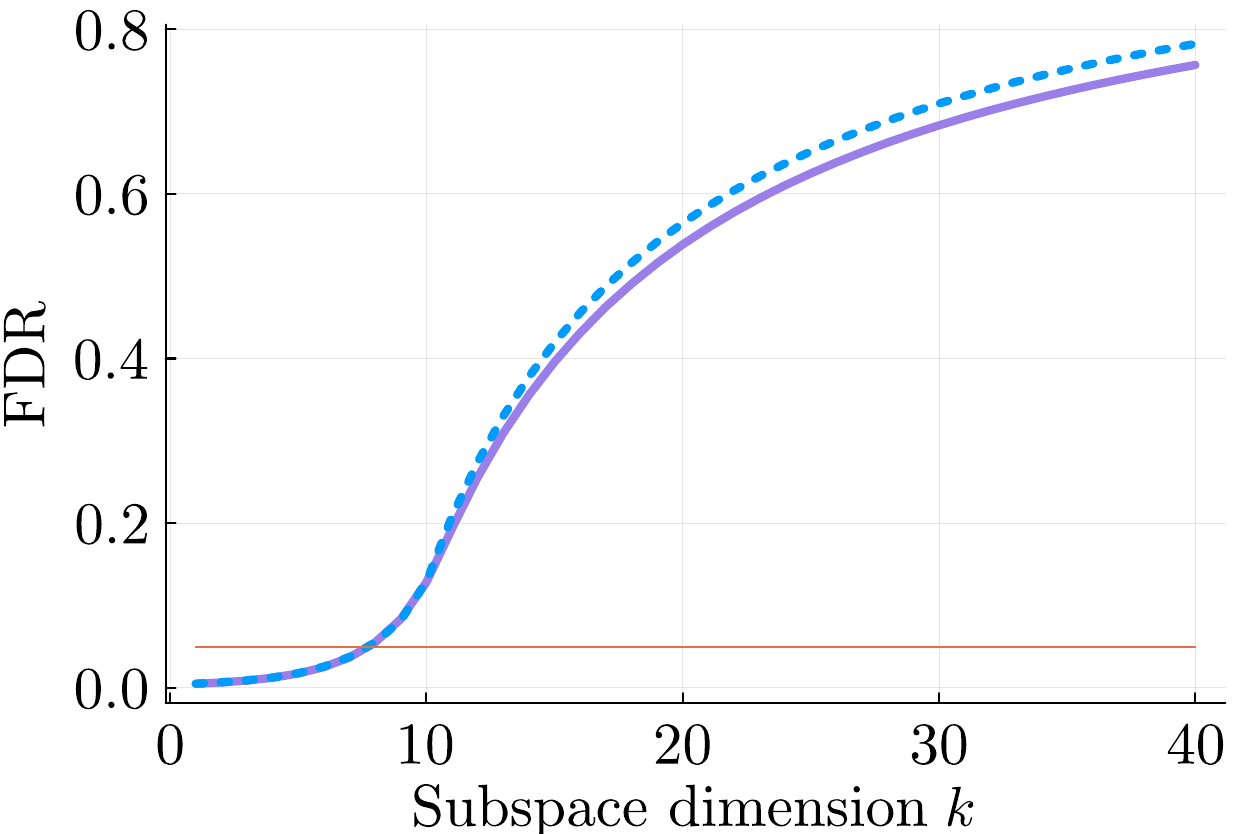} \\[2\tabcolsep]
        \setcounter{imagecolumn}{0} &%
        Dimension $n = 100$ &
        Dimension $n = 500$ &
        Dimenson $n = 1000$
    \end{tabular}
    \caption{Results on FDR estimation for the \texttt{WishartFactor} ensemble using Algorithms~\ref{alg:FDR} and~\ref{alg:rank-estimate}: the columns correspond to problem sizes $n= 100, 500, 1000$ and the rows correspond to the three models for the spectrum of the low-rank signal from~\eqref{eq:exp-spectrum}.  The proportion between dimensions is set to $ n/m = 1/2$.}
	\label{fig:res-wishart-factor}
\end{figure}

\subsection{Experiments with Real Data}
Here we present the results obtained by applying our methods to data from single-cell RNA sequencing and from hyperspectral imaging.  We compare rank estimates obtained using $p = \frac{3}{5} \cdot \texttt{median}\left( \left\{ \Delta_j \right\}^{n-1}_{j=1}\right) \cdot n$ --- the default value in our implementation of Algorithm~\ref{alg:rank-estimate} --- with those obtained using $0.5 p$ and $2p$, with lower values yielding larger rank estimates.

\paragraph{Single-cell RNA sequencing.} Traditional RNA sequencing technologies could only process mixed populations of cells and as a result they were limited in their ability to identify gene expression profiles of individual cells.  The recent development of single-cell sequencing techniques \cite{tang2009mrna} has had a profound impact on both biology and medicine, although it yields massive data matrices of gene expression counts for each cell.  Typical workflows employ PCA to identify features that provide a suitable low-dimensional embedding of the gene expressions.  We apply our methodology in this context.  We consider a publicly available dataset of Peripheral Blood Mononuclear Cells \cite{10xGenomics2023PBMC3K} that consists of 2,700 single cells, and we obtain a data matrix $\bX \in \RR^{1872 \times 2700}$ after using a standard pipeline to preprocess the data \cite{wolf2018scanpy}. 
The left panel of Figure~\ref{fig:rna} gives the empirical distribution of the singular values of $\bX$; the shape of this distribution closely matches the assumptions required for our theory to hold.  The right panel of Figure~\ref{fig:rna} depicts FDR estimates obtained using the three different rank estimates.  These FDR estimates change gracefully with different rank estimates, especially exhibiting robustness for smaller subspace dimensions $k$.  For $\alpha = 0.05$, we obtain a subspace estimate of dimension $\widehat k = 4$. 

\begin{figure}[t]
    \centering
    \begin{subfigure}{0.45\textwidth}
        \centering
        \vspace{0.2em}
        \includegraphics[width = \textwidth]{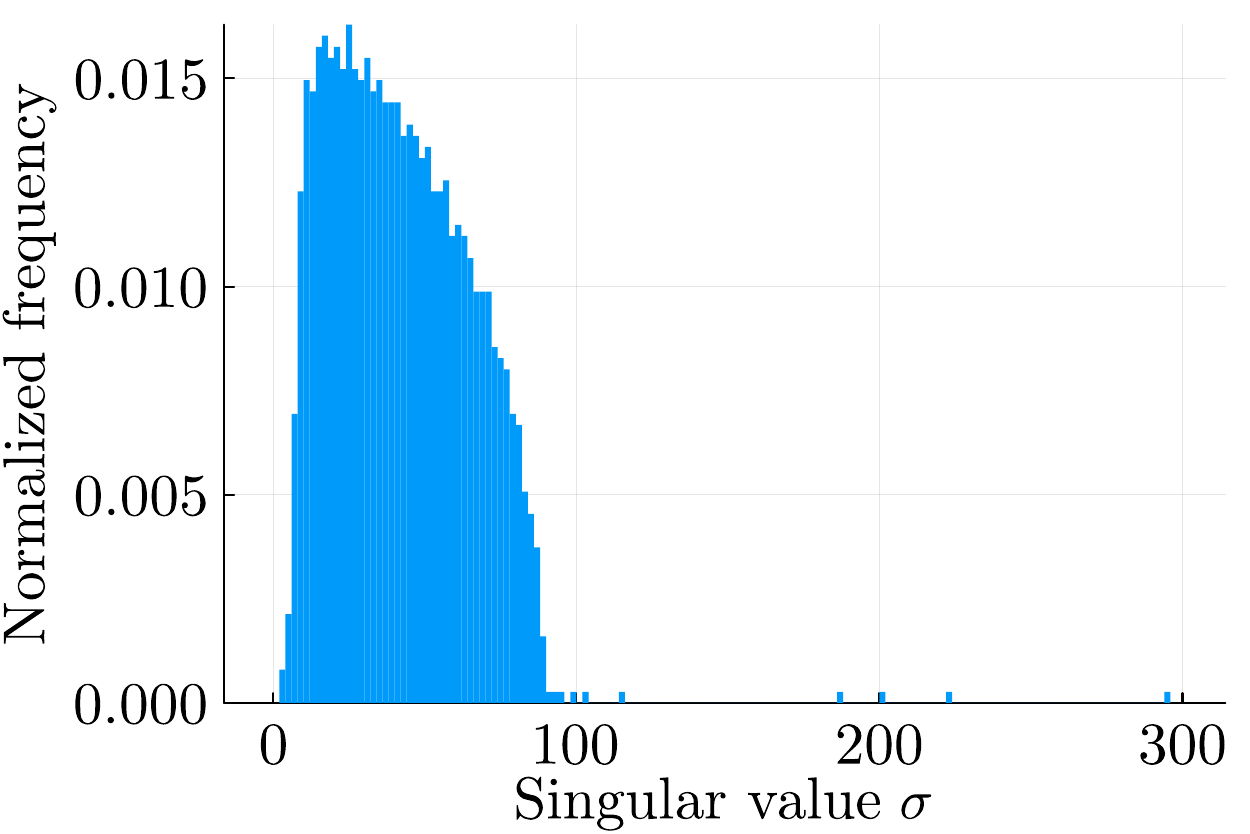}
    \end{subfigure}
    \begin{subfigure}{0.45\textwidth}
        \centering
        \vspace{0.2em}
        \includegraphics[width = \textwidth]{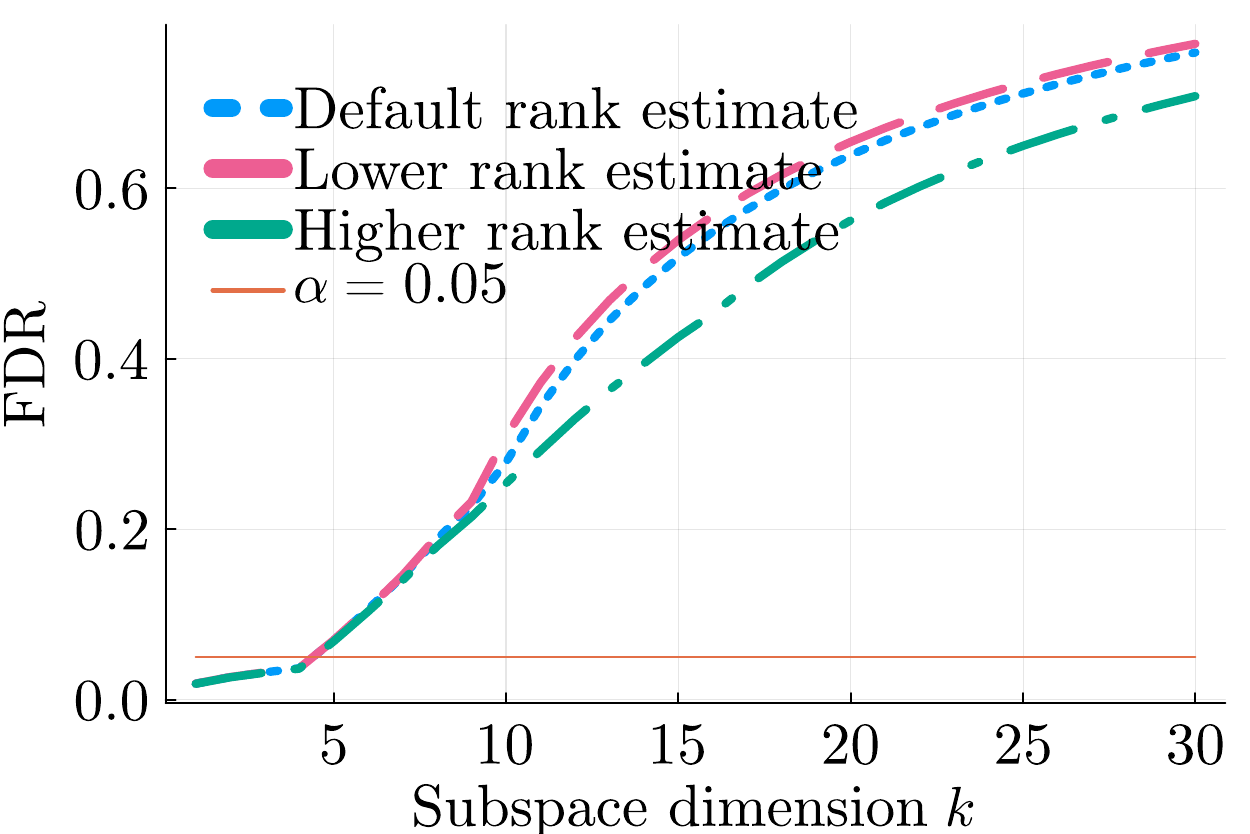}
    \end{subfigure}
    \caption{Results for single-cell RNA data: the left plot is a histogram of the singular values of the data matrix, and the right plot shows three FDR estimate using different rank estimators. 
    }
    \label{fig:rna}
\end{figure}

\paragraph{Hyperspectral imaging (HSI).} HSI is a technique that collects information across a wider spectrum of wavelengths than traditional cameras, and it is employed in application domains such as remote sensing, environmental monitoring, agriculture, and mineralogy for the identification and analysis of different materials and substances based on their unique spectral signatures.  HSI data is represented as a third-order tensor $\bW \in \RR^{n \times m \times d}$, 
where $d$ represents the number of spectral bands.  The value of $d$ is $3$ for RGB images and it is considerably larger for images obtained via HSI.  As spectral bands that are closer in wavelength exhibit high correlation, it is common to reduce dimensionality by applying PCA to the $d \times d$ outer-product matrix $\bX = \frac{1}{nm}\sum_{ij} \bW_{ij :} \bW_{ij:}^\top$.
For a comprehensive explanation of this process, we direct interested readers to \cite{rodarmel2002principal}.  For this experiment, we consider the Indian Pines dataset \cite{IndianPinesDataset} in which $\bW \in \RR^{145 \times 145 \times 220}$ and we apply our methodology to obtain estimates with FDR control of the column space of the outer-product matrix $\bX$.  The left panel of Figure~\ref{fig:hsi} displays the empirical distribution of the eigenvalues of $\bX$, while the right panel displays FDR estimates obtained using our three rank estimation methods.  The spectral profile of the HSI data is notably distinct, displaying a less pronounced spectral bulk compared to the RNA sequencing data. Although our FDR estimates are close to each other for smaller values of $k$, the absence of a well-defined bulk
in the HSI spectrum appears to negatively impact the robustness of the FDR estimates for larger $k$.

\begin{figure}[t]
    \centering
    \begin{subfigure}{0.45\textwidth}
        \centering
        \vspace{0.2em}
        \includegraphics[width = \textwidth]{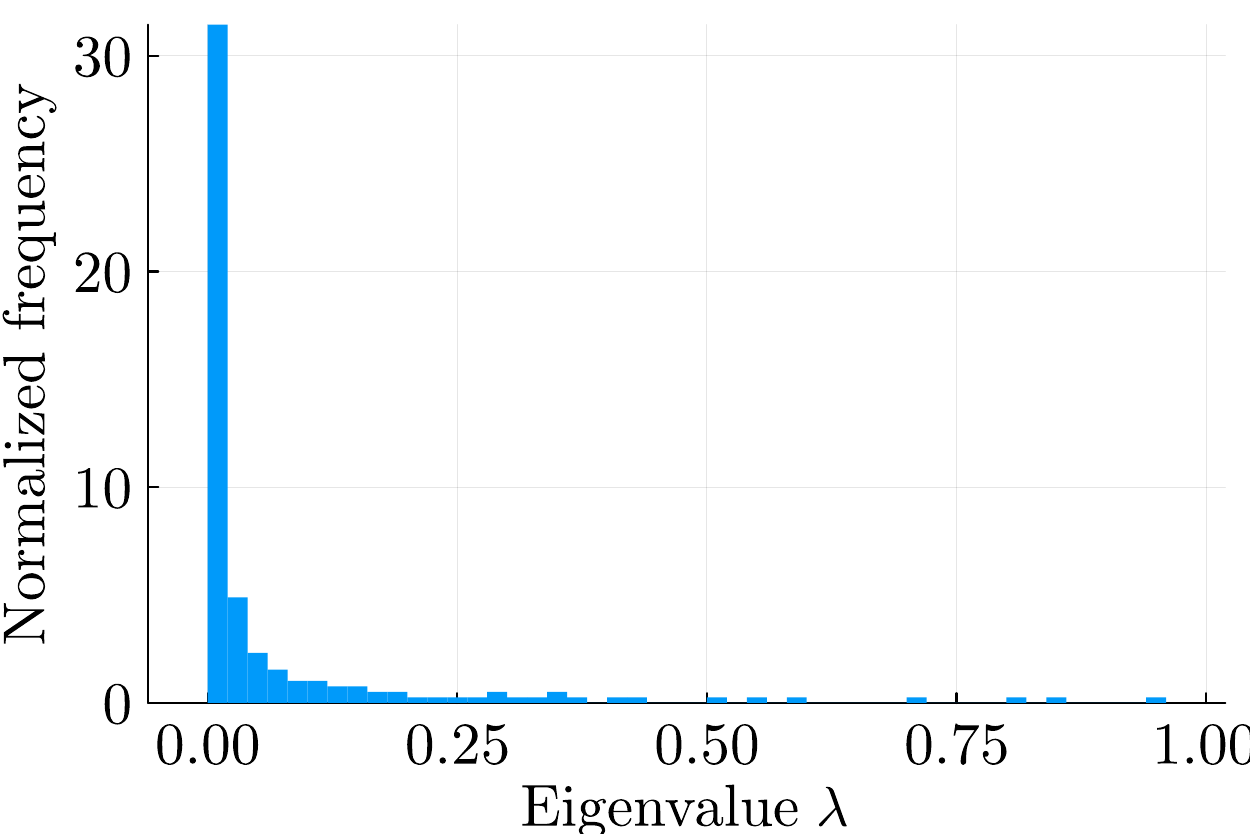}
    \end{subfigure}
    \begin{subfigure}{0.45\textwidth}
        \centering
        \vspace{0.2em}
        \includegraphics[width = \textwidth]{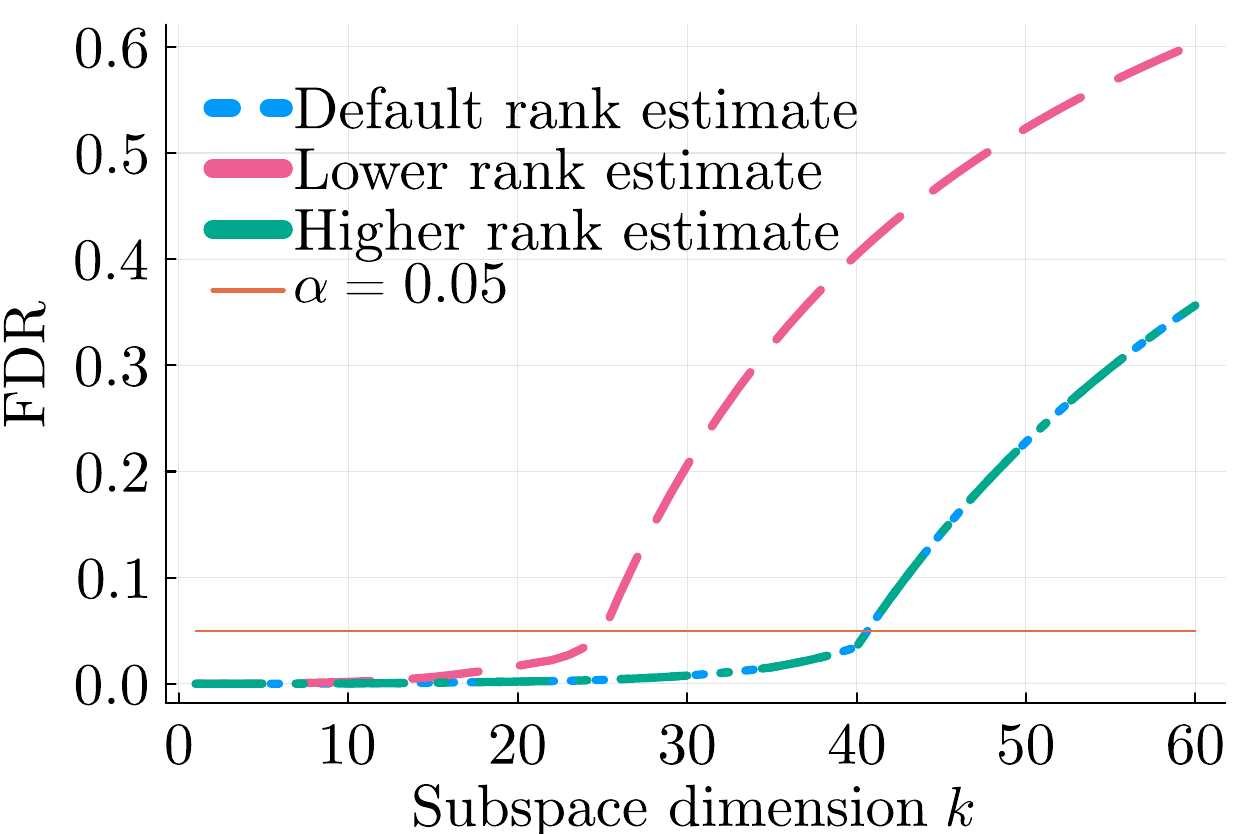}
    \end{subfigure}
    \caption{Results for Hyperspectral Imaging: the left plot is a histogram of the singular values of the data, and the right plot shows our FDR estimate using three rank estimators.}
    \label{fig:hsi}
\end{figure}

\section{Proofs}\label{sec:proofs}

In this section, we present all the proofs of the results stated earlier in the paper.  The first section concerns the symmetric case and here we give all the relevant details.  The second section concerns the asymmetric case, which largely parallels the symmetric case; hence, we only describe the modifications relative to the symmetric case. For notational convenience, we will drop the dependency on $\cU_{k}$ in $\FDR(\cU_{k})$ and simply write $\FDR(k)$.

\subsection{Symmetric Case}
\subsubsection{Proof of \cref{thm:main}} \label{sec:proof-main}
We start by 
sketching a roadmap for the proof. Recall that  $r^\star = \max\{j \in [r] \mid \theta_{j} > 1/G_{\mu_\Eb}(b^+)\} \leq r$ is the number of components that do not get subsumed by the noise. To set the stage, we define and recall a couple of FDR-related sequences that are going to be relevant in the proof.
\begin{enumerate}
     \item[] (\textbf{Limiting FDR})
\begin{equation}
  \label{eq:infty-FDR}
  \FDR^{\infty}(k) = 1 + \frac{1}{k} \sum_{i=1}^{\min\{k, r^\star\}} \frac{1}{\theta_i^2 G_{\mu_{\bE}}'(G_{\mu_{\bE}}^{-1}(1/\theta_{i}))} \quad \text{and} \quad k^\infty = \max\left\{j \in [r] \mid \FDR^\infty(j) \leq \alpha \right\}.
\end{equation}
        \item[] (\textbf{True FDR})
  \begin{equation}
    \label{eq:2}
\FDR^{(n)}(k) = \EE \left[ \frac{\Tr \left( P_{\widehat{\cU}^{(n)}_{k}} P_{{\cU^{(n)}}^{\perp}} \right)}{k}\right]. %
\end{equation}
\item[] (\textbf{Upper-bound FDR}) Let $\bar \cU = \text{span}\{u_1, \dots, u_{ {r}^\star}\}$ be the subspace generated by top $r^\star$ components of $\bA$ and define
  \begin{equation} 
    \label{eq:upper-fdr}
\overline{\FDR}^{(n)}(k) = \EE \left[ \frac{\Tr \left( P_{\widehat{\cU}^{(n)}_{k}} P_{({\bar{\cU}}^{(n)})^{\perp}} \right)}{k}\right] \quad \text{and} \quad \bar{k} = \max\{j \in [r]\mid \overline{\FDR}^{(n)}(j)\leq \alpha\}.
\end{equation}
Observe that since $\bar \cU \subseteq \cU$, we have that ${\FDR}^{(n)}(k) \leq \overline{\FDR}^{(n)}(k)$.
  \item[] (\textbf{Distribution-aware FDR Estimate})
        \begin{equation}
          \label{eq:3}
          \widetilde{\FDR}^{(n)}(k) = 1 + \frac{1}{k} \sum_{j=1}^{\min\{k,{r}^\star\}} \frac{G_{\mu_{\bE}}\left(\lambda_{i}^{(n)}\right)^{2}}{G_{\mu_{\bE}}'\left(\lambda_{j}^{(n)}\right)}.  %
        \end{equation}
        Note that unlike the FDR estimate in Algorithm~\ref{alg:FDR}, the estimate above uses the true Cauchy transform of the limiting distribution and the true observable rank $r^\star$.
\end{enumerate}

In the remainder of this section, we drop the subscript of $n$ and state $\cU, \lambda_i$ instead of $\cU^{(n)}, \lambda_i^{(n)}$ for notational brevity whenever the index of $n$ is implicit from the context.
Recall that $\widehat{\FDR}^{(n)}(k)$ denotes the FDR estimate given by Algorithm~\ref{alg:FDR}, and $\widehat k$ is the output of the same algorithm.
\cref{thm:main} follows from two simple steps:
\begin{enumerate}
    \item[] \textbf{Step 1: Convergence of FDR estimate and true FDR.} We show that for each fixed $k \in \NN$
    \begin{equation*}
        \widehat {\FDR}^{(n)}(k) \rightarrow \FDR^\infty(k) \quad \text{a.s.} \quad \text{and} \quad \overline{\FDR}^{(n)}(k) \rightarrow \FDR^\infty(k)        
    \end{equation*}
    as $n$ goes to infinity. 
    \item[] \textbf{Step 2: Equality of maximizers.} Using the convergence above, we can conclude that indeed $\bar k$ and $\widehat k$ match for large $n$ provided that the limiting $\FDR^\infty (k^\infty) \neq \alpha.$
\end{enumerate}

When $r \neq r^\star$ this gives control for the true FDR since we trivially have the upper bound $\FDR^{(n)}(\widehat k) \leq \overline{\FDR}^{(n)}(\widehat k) \leq \alpha,$ where the second inequality follows since $\widehat k = \bar k$ and by definition $\bar k$ controls the upper-bound FDR. Further, note that if $r = r^\star,$ then $\FDR^{(n)} = \overline{\FDR}^{(n)}$ and the preceding reasoning may be applied to conclude the theorem. Thus, for the remainder of this argument, we focus on bounding $\overline{\FDR}^{(n)}$ instead of $\FDR^{(n)}.$ Let us start with the formal proof.

\textbf{Step 1.} First, we establish an intermediate result that ensures that some of the FDR quantities defined above are nonincreasing and takes on values in $[0,1]$.

\begin{lemma}\label{lem:ratio-properties} Consider a probability measure over the real line $\mu$ with bounded support. Let $b = \sup\{ t : t \in \supp(\mu)\}.$ 
  Then, the map $\Phi\colon (b, \infty)\rightarrow \RR$ given 
  by $\Phi(x) = \frac{G_{\mu}(x)^2}{G_\mu'(x)}$ yields a $[-1, 0]$-valued continuous nonincreasing function.
  \end{lemma}
  \begin{proof}
    Note that $$\Phi(y) = -\frac{ \left[\EE_{\theta \sim \mu} (y-\theta)^{-1} \right]^{2}}{\EE_{\theta \sim \mu}(y-\theta)^{-2}},$$ this quantity is well-defined in the domain $(b, \infty)$. Observe that $\Phi$ is defined as a ratio of continuous functions with the denominator bounded away from zero; hence we conclude that $\Phi$ is continuous. Moreover, $\Phi$ takes nonpositive values and by Jensen's inequality $\Phi(y) \geq -1$. Since both $G_\mu$ and $G_\mu'$ are integrals of bounded continuous functions over compact sets, they are differentiable. Taking a derivative of $\Phi$ reveals
  $$
  \Phi'(y) = \frac{2G_\mu(y)G'_\mu(y)^{2} - G_\mu(y)^{2}G''_\mu(y)}{G'_\mu(y)^{2}}.%
  $$
  To prove the lemma it suffices to show that the numerator is negative.
  \begin{align*}
    2G_\mu(y)G_\mu'(y)^{2} &= 2\EE(y -\theta)^{-1} \left[\EE(y - \theta)^{-2}\right]^{2} \\
      &\leq 2 \left[\EE(y -\theta)^{-1}\right]^{2} \EE (y - \theta)^{-3}\\
    &= G_\mu(y)^{2}G_\mu''(y)
  \end{align*}
  where the inequality uses that for any  random variable $X$ we have $\left[\EE X^{2}\right]^{2} \leq \EE |X| \EE |X|^{3}$, which is a simple consequence of the Cauchy-Schwarz and Jensen inequalities. This concludes the proof of the claim.
  \end{proof}  

Equipped with this lemma, we can tackle the asymptotic convergence of the different FDRs defined at the start of the section.
\begin{proposition}\label{lem:fdr-convergence}
    Fix any $k \in \NN.$ Then, the following three assertions hold.
    \begin{enumerate}
      \item As $n \rightarrow \infty$ we have $\overline{\FDR}^{(n)}(k) \rightarrow \FDR^{\infty}(k)$.
      \item As $n \rightarrow \infty$ we have $\FDRt^{(n)}(k) \rightarrow \FDR^{\infty}(k)$ almost surely.
      \item As $n \rightarrow \infty$ we have ${\widehat{\FDR}}^{(n)}(k) \rightarrow \FDR^{\infty}(k)$ almost surely.
    \end{enumerate}
  \end{proposition}
  \begin{proof}
We start by proving the first statement. Recall $\widehat \bU$ and $\bU$ are matrices whose columns are the eigenvectors of $\bX$ and $\bA$, respectively.  For clarity and simplicity, we assume that the eigenvalues of $\bA$ are all distinct; nonetheless, the same proof strategy follows through with repeated eigenvalues by substituting $\widehat \bu_i \widehat \bu_i^\top$ with the projector onto the eigenvalue subspace associated with the $i$th eigenvalue. 

Expanding the definition of the upper-bound FDR yields 
    \begin{align*}
  \overline{\FDR}^{(n)}(k) & = \EE \left[ \frac{\Tr \left( \widehat{\bU}_{1:k} \widehat{\bU}_{1:k}^{\top }(\bI - \bU_{1:r^\star}\bU_{1:r^\star}^{\top}) \right)}{k}\right] \\
       & = \frac{1}{k}\sum_{j = 1}^{k}  \EE \left[\Tr\left(\widehat\bu_{j} \widehat\bu_{j}^{\top} \left(\bI - \sum_{i}^{r^\star} \bu_{i} \bu_{i}^{\top} \right) \right)\right]\\
       & = \frac{1}{k}\sum_{j = 1}^{k} \left\{1 - \sum_{i}^{r^\star} \EE \left[\left(\widehat\bu_{j}^{\top} \bu_{i} \right)^{2} \right]\right\}\\
       & = 1 - \frac{1}{k}\sum_{j=1}^k \sum_{i=1}^{r^\star} \EE\left[\left(\widehat\bu_j^\top \bu_i\right)^2\right]
\end{align*}
Thus, it is enough to focus on $\EE\left[\left(\widehat{\bu}_{j}^{\top} \bu_{i}\right)^{2}\right]$. Using \cref{thm:raj} in tandem with \cref{prop:zero-correlation} yields that for all $i, j \subseteq [n]$ the following almost sure convergence holds:
\begin{align} \label{eq:assym-correlation}
\left(\widehat{\bu}_{j}^{\top} \bu_{i}\right)^{2} \rightarrow
  \begin{cases}
    -\frac{1}{\theta_i^{2}G'_{\mu_{\bE}}\left(G_{\mu_{\bE}}^{-1}(1/\theta_{i})\right)}& \text{if }i = j \text{ and } i\leq r^\star,  \\
    0 & \text{otherwise.}
  \end{cases}
\end{align}
Recall that we have suppressed the dependence on $n$, so the above convergence result should be understood in the limit $n \rightarrow \infty$.
Hence, taking expectations and substituting, we derive that
$$  \overline{\FDR}^{(n)}(k)  \rightarrow 1 + \frac{1}{k} \sum^{\min\{k, r^\star\}}_{i = 1} \frac{1}{\theta_{i}^{2}G_{\mu_\Eb}'\left(G_{\mu_\Eb}^{-1}(1/\theta_i)\right)} = \FDR^{\infty}({k}),
$$
as we wanted. 

Turning next to the convergence of $\widetilde{\FDR}^{(n)}$, note that $x \rightarrow G_{\mu_\Eb}(x)^2/G_{\mu_\Eb}'(x)$ is a continuous function for $x \notin [a, b]$ from Lemma~\ref{lem:ratio-properties}. Thus, using the limit $\lambda_i \as G_{\mu_\Eb}^{-1}(1/\theta_i)$ from \cref{thm:raj}, we get that
$$
 \frac{G_{\mu_{\bE}}(\lambda_i)^2}{G_{\mu_{\bE}}'(\lambda_i)} \as \frac{1}{\theta_i^2 G_{\mu_\Eb}'\left(G_{\mu_\Eb}^{-1}(1/\theta_i)\right)},
$$
which leads to the result after a simple substitution.

Finally, let us prove that $\widehat{\FDR}^{(n)}$ also converges to the same limit. Assume that $n$ is large enough to ensure that $r^\star \leq \widehat r \leq r$, which happens almost surely by assumption. Expanding the definitions of $\widehat{\FDR}^{(n)}, \widetilde{\FDR}^{(n)}$ and using the triangle inequality, we have for any $k$ that
\begin{equation}\label{eq:diff-fdrt}
    \left|\widehat{\FDR}^{(n)}(k) - \widetilde{\FDR}^{(n)}(k) \right| 
    \leq \frac{1}{k} \left( \underbrace{\sum_{i=1}^{\min\{k , r^\star\}} \left|  \frac{G_{\mu_{\bE}}(\lambda_i)^2}{G'_{\mu_{\bE}} (\lambda_i)} - \frac{ \widehat{G}(\lambda_i)^2}{\widehat{G}'(\lambda_i)}  \right|}_{T_1} + \underbrace{\sum_{i = r^\star + 1}^{\widehat r} \left| \frac{ \widehat{G}(\lambda_i)^2}{\widehat{G}'(\lambda_i)} \right|}_{T_2}\right).
\end{equation}
We will show that each summand goes to zero. Let us start by bounding $T_1$. 
Let $\mu_n$ be the empirical distribution of $\lambda_{r^\star+1}, \dots, \lambda_n.$ 
By Assumption~\ref{ass:amodel2} and Weyl's interlacing inequalities, we obtain that $\mu_{n}$ 
converges weakly to $\mu_{\bE}$ %
. Thus, we have that $\widehat{G}(t) \rightarrow G_{\mu_{\bE}}(t)$ for any $t$ outside the support of the bulk $\supp(\mu_{\bE}) = (a,b)$. %
Furthermore, for any $C > b$ that is greater than the right edge of the bulk, there exists $L > 0$ such that the function $G_{\mu_{\bE}}$ is $L$-Lipschitz on $[C, \infty)$. The same statement holds almost surely for $\widehat G$ once $n$ is large enough thanks to Lemma~\ref{lem:ratio-properties}.
Moreover, by Theorem~\ref{thm:raj} we have that $\lambda_i^{(n)} \rightarrow \rho_i$ for some $\rho_i > b$ as $i \leq r^\star$.  Fix $C = (b + \min_{i \leq r^\star} \rho_i)/2$, then 
\begin{align*}
   \left| \widehat{G}\left(\lambda_i^{(n)}\right) - G_{\mu_{\bE}}\left(\lambda_i^{(n)}\right)\right| & \leq \left| \widehat{G}\left(\lambda_i^{(n)}\right) - \widehat G(\rho_i) \right| +  \left|\widehat G(\rho_i) - G_{\mu_{\bE}}(\rho_i) \right| + \left|G_{\mu_{\bE}}(\rho_i) -  G_{\mu_{\bE}}\left(\lambda_i^{(n)}\right)\right| \\ 
   & \leq 2L \left|\lambda_i^{(n)} - p_i\right| +  \left|\widehat G(\rho_i) - G_{\mu_{\bE}}(\rho_i) \right|.
\end{align*}
The last expression goes to zero almost surely as $n$ goes to infinity. Thus, we have that $\widehat{G}(\lambda_i^{(n)})$ and $G_{\mu_{\bE}}(\lambda_i^{(n)})$ share the same limit. An analogous argument shows the same conclusion for $\widehat{G}'(\lambda_i^{(n)})$ and $G'_{\mu_{\bE}}(\lambda_i^{(n)}),$ and, by construction, their limit is finite and bounded away from zero. Then, the quotient rule for limits yields that the bound on $T_1$ in \eqref{eq:diff-fdrt} goes to zero.

Next, we turn to $T_2$.  
Let $i$ be any index in $\{r^\star+1, \dots, r\}$. By Lemma~\ref{lem:ratio-properties}, the ratios $\widehat G(\lambda_i^{(n)})^2/ \widehat G'(\lambda_i^{(n)})$ are bounded between $-1$ and $0$. So it suffices to show that $\widehat G'(\lambda_i^{(n)})$ goes to infinity with $n$ to conclude that $T_2 \rightarrow 0$.  For any fixed $M > 0,$ define the truncated function $f_M(z, t) =  \min\left\{M, \frac{1}{(z -t)^2}\right\}$ and the constant $l_M = \lim_{t \downarrow b}\EE_{z \sim \mu_{\bE}} f_M(z, t).$  Moreover, for any $M$ we have $|\widehat G '(\lambda_i^{(n)})| \geq \EE_{z \sim \mu_{n}} f_M(z, \lambda_i^{(n)})$. Taking inferior limits on both sides yields 
\begin{equation}\label{eq:low-semi}
    \liminf_{n\rightarrow \infty} |\widehat G '(\lambda_i^{(n)})| \geq \liminf_{n\rightarrow \infty} \EE_{z \sim \mu_{n}} f_M(z, \lambda_i) = l_M 
\end{equation}
where the last relation uses that $\lim_{n\rightarrow \infty} f_M(z, \lambda_i^{(n)}) = l_M$ %
, which follows since $f_M(\cdot, \lambda_i^{(n)})$ converges uniformly to $f_M(\cdot, b)$ and $\mu_n$ converges weakly almost surely to $\mu_{\bE}$.
Using Assumption~\ref{ass:model2},~\cref{prop:nice-edge}, and Fatou's Lemma we obtain $\liminf_{M \rightarrow \infty} l_M \geq |G'(b^+)| = \infty$.  The $M$ we used in \eqref{eq:low-semi} was arbitrary, and so we derive that $\liminf_{n\rightarrow \infty} |\widehat G '(\lambda_i^{(n)})| = \infty$, thus establishing that $T_2 \rightarrow 0$ and 
concluding the proof of the lemma.

\end{proof}

\noindent \textbf{Step 2.} To finish the proof, we need the following auxiliary lemma.

\begin{lemma} For all large $n$, the functions $\widehat{\FDR}^{(n)}(\cdot) $ and $\widetilde{\FDR}^{(n)}(\cdot)$ are nondecreasing and their image is contained in $ [0,1]$. The same properties hold for $\FDR^\infty$.
\end{lemma}
\begin{proof}
  We prove the statement for $\widehat{\FDR}^{(n)}$. 
  Define $C = \frac{\lambda_{\hat{r}}+\lambda_{\hat{r} + 1}}{2}$ and the function $\Phi\colon [C, \infty) \rightarrow \RR$ given by $y \mapsto \frac{G_{\hat{\mu}_n}(y)^{2}}{G_{\hat{\mu}_n}'(y)}$ where $\widehat{\mu}_n = \frac{1}{n-\hat{r}} \sum_{i=\hat{r} +1}^n \delta_{\lambda_i}$. %
  The statement follows as direct consequence of Lemma~\ref{lem:ratio-properties} by noting that $\widehat{\FDR}^{(n)}(k)= 1 + \frac{1}{k}\sum^{\min\{k, \widehat r\}}_{j = 1} \Phi(\lambda_{j}).$ 
  The same arguments apply for $\widetilde{\FDR}^{(n)}$ after noting that for large enough $n$ the largest $r^\star$ eigenvalues are outside of the support of the asymptotic distribution $\mu_{\bE}$. The limiting FDR inherits both properties since it can be realized as the limit of $\widehat{\FDR}^{(n)}.$
\end{proof}

Finally, we can show that our estimate $\widehat k$ matches $\bar{k} = \max \{ j \in [r] \mid \overline{\FDR}^{(n)}(j) \leq \alpha \}$ for all large $n.$ To prove this, we show that both $\widehat k$ and $\bar{k}$ almost surely converge to $k^{\infty}$. Since $\FDR^{\infty}(k^{\infty}) \neq \alpha$, we have that \begin{equation}\label{eq:separation}
\FDR^{\infty}(k^{\infty}) < \alpha < \FDR^{\infty}(k^{\infty} + 1).\end{equation} Since $\widehat \FDR^{(n)}$ converges pointwise to the asymptotic ${\FDR}^{(\infty)}$ we have that for large $n$
$$
    \widehat{\FDR}^{(n)}(k^{(n)}) < \alpha < \widehat{\FDR}^{(n)}(k^{(n)} + 1).
$$ 
Further, by Lemma~\ref{lem:fdr-convergence} we have that for large enough $n$, $\widehat \FDR^{(n)}$ is nondecreasing in $k$, and so we derive that $\widehat k$ = $k^{\infty}$. Since $\overline{\FDR}^{(n)}$ converges to $\FDR^\infty$, for all $k \leq r$ and by \eqref{eq:separation} we also conclude $\bar{k} = k^{\infty}.$ 
Thus, we obtain that almost surely $\widehat k = \bar{k}$ for large enough $n$. This finishes the proof of Theorem~\ref{thm:main}. 
\hfill $\square$
\subsubsection{Proof of \cref{thm:rank-estimate}}\label{proof:rank-estimate}
By Weyl's interlacing inequalities, we have for all $i \geq n$ that
\begin{equation}
    \label{eq:inter}
    \lambda_i(\bE_n) \leq \lambda_i(\bX_n) \leq \lambda_{i-r}(\bE_n)
\end{equation} 
with the convention that $\lambda_k(\bX) = \infty$ if $k \leq 0$. Fix an index $j > r;$ then, using the interlacing inequalities combined with a telescoping sum, we obtain that 
\begin{align*}
    \Delta_j = \lambda_{j-1}(\bX_n) - \lambda_{j}(\bX_n) &\leq \lambda_{j-1-r}(\bE_n) - \lambda_{j}(\bE_n) \\
    &= \sum_{i = 0}^{r} \left(\lambda_{j-1-i}(\bE_n) - \lambda_{j-i}(\bE_n)\right)\\
    & \leq r \max_{i \in [n-1]}\left\{\lambda_{i}(\bE_n) - \lambda_{i+1}(\bE_n)\right\}.
\end{align*}
Recall the labeling $s_i := \lambda_{i}(\bE_n) - \lambda_{i+1}(\bE_n).$ Assumption~\ref{ass:spacings} ensures that for any $\varepsilon >0 $ there exists an $N > 0$ such that for all $n \geq N$, we have that $\max_{i \in [n-1]}\left\{s_i\right\} \leq \varepsilon \cdot
n^{-1/2}$. Thus, by taking $\varepsilon = p/2r$, we obtain 
\begin{equation}
    \Delta_j \leq \frac{p}{2} \cdot
    n^{-1/2} < 
    p \cdot  n^{-1/2} \quad \text{for all }j>r.
\end{equation}
Thus, we conclude that for all large enough $n$,  
$\textsc{RankEstimate}(\bX_{n}) \leq r$. Next, let us show the lower bound of $r^\star.$ By \cref{thm:raj} we have that $\lambda_{r^\star}(\bX) \rightarrow \rho_{r^\star} > b.$ Furthermore by Theorem~\ref{thm:raj} we have that $\lambda_{{r^\star} + 1}(\bX_n)$ almost surely converges to $b$. Thus, for large enough $n$, we obtain that $$\lambda_{r^\star}(\bX_n) - \lambda_{r^\star + 1}(\bX_n) \geq \frac{\rho_{r^\star} - b}{2} > 
p \cdot n^{-1/2}$$
where the last inequality holds for any $n$ large since $n^{-1/2}$ tends to zero. Hence $\textsc{RankEstimate}(\bX_{n}) \geq r^\star$
This concludes the proof of \cref{thm:rank-estimate}.
\hfill $\square$

\subsubsection{Proof of \cref{cor:wigner}} \label{proof:wigner-spacings}

It is well-known that the empirical distribution of the spectrum of Wigner matrices converges weakly almost surely to the semicircle law, which is given by the density
$$
\varrho(t) = \frac{1}{2\pi} \sqrt{(4 - t^2)_+}.
$$
The distribution is supported on the interval $[-2, 2].$ Hence, Wigner ensembles satisfy Assumption~\ref{ass:model2}. 
So to prove that~\cref{ass:spacings} holds, it suffices to show that $\sup_{i} |\lambda_i (\bE_n) - \lambda_{i+1}(\bE_{n})| = o(n^{-1/2}).$ To this end, we will use a celebrated result from random matrix theory known as eigenvalue rigidity. Intuitively, rigidity amounts to the concentration of eigenvalue around quantiles of $\mu_{\bE}.$ For any $i \in\{ 1, \dots, n\}$ define the quantiles $\gamma_i^{(n)}$ as
$$ \int_{\gamma_i^{(n)}}^{2} g(t) dt = \frac{i}{n}. $$
As we shall see, $\gamma_i^{(n)}$ gives the typical location of the $i$th largest eigenvalue $\lambda_i^{(n)}$. To avoid cluttering notation, we will often drop the super index in $\gamma_i^{(n)}$ when it is clear from the context.  

Our strategy to bound the spacing will be to control  
\begin{equation}\label{eq:triangle-spacings}
    | \lambda_i(\mathbf{H}_n) - \lambda_{i+1}(\mathbf{H}_n)| \leq |\lambda_i(\mathbf{H}_n) - \gamma_i| + |\lambda_{i+1}(\mathbf{H}_n) - \gamma_{i+1}| + |\gamma_i - \gamma_{i+1}|.
\end{equation}
Eigenvalue rigidity will give us a handle on the distance between quantiles and eigenvalues. Before we delve into rigidity, let us bound the deterministic term $|\gamma_i - \gamma_{i+1}|$. 
\begin{lemma}\label{lem:bound-quantiles}
    For large enough $n,$ we have the following bound for all $i \in \{1, \dots, n\}$:
    $$ \left| \gamma_i^{(n)} - \gamma_{i+1}^{(n)} \right| \leq \pi^2 \cdot n^{-2/3}. $$ 
\end{lemma}
\begin{proof}
     For simplicity, we focus on the larger quantiles $i \leq n/2$; the argument follows symmetrically for the other half. Define $\omega_n \in (0, 2)$ to be the largest scalar such that if $t \in \Omega_n = [-\omega_n, \omega_n]$ then $\varrho(t) \geq n^{-1/3}.$ 
     We consider a couple of cases. 
     
     \textbf{Case 1}. First, we consider the case where quantiles $\gamma_i, \gamma_{i+1} \in \Omega$. By definition, we have that 
$$ n^{-1/3} \cdot |\gamma_{i} - \gamma_{i+1}| \ \leq \int_{\gamma_{i+1}}^{\gamma_i}\varrho(t) dt = n^{-1},$$
and so rearranging terms gives the desired bound with an even better constant.

\textbf{Case 2}. Next, assume that $\gamma_{i}, \gamma_{i+1} \in (\omega_n, 2].$ Observe that it suffices to prove that $|2- \omega_n| \leq \pi^2 \cdot n^{-2/3}.$ We can write $\omega_n = 2 - h(n^{-1/3})$ where $h(\cdot) = f^{-1}(\cdot)$ is the functional inverse of $f(t) = \varrho(t-2)$ restricted to the interval $[0, 2]$ so that the inverse is well-defined. Thus, we focus on bounding $| 2 - \omega_n| = |h(n^{-1/3})|$. A simple computation reveals that $$h(t) = 2 - 2\sqrt{1 - \pi \cdot t^2}.$$ Near zero, the second derivative of this function is $Q$-Lipschitz  for some $Q>0$; to see this just note that its third derivative is bounded. Thus, the Taylor approximation error is bounded by
\begin{equation*}
    \left|h(t) - \frac{h''(0)}{2} t^2\right| \leq \frac{Q}{6} t^3,
\end{equation*}
here we used that $h(0) = h'(0) = 0$; see~\cite[Lemma 1.2.4]{nesterov2013introductory}.
Therefore, for $t$ small enough we obtain $\frac{\pi^2}{2} t^2 = \frac{h''(0)t^2}{4} \leq  h(t) \leq h''(0) t^2 \leq 2\pi^2 t^2$. For $n$ is large enough, we might substitute $t = n^{-1/3}$ and obtain
\begin{equation}\label{eq:quantile-approx}
\frac{1}{2} \pi^2 \cdot n^{-2/3} \leq | 2- \omega_n| \leq \pi^2 \cdot n^{-2/3},
\end{equation} 
as we wanted.

\textbf{Case 3}. Finally, assume that $\gamma_i \in (-\omega_n, 2)$ and $\gamma_{i+1} \in [-\omega_n, \omega_n].$ The first $\lfloor n/2 \rfloor$ spacings are monotonically decreasing
$$
|2- \gamma_1| > |\gamma_1 - \gamma_2| > \dots > |\gamma_{\lfloor n/2 \rfloor + 1} - \gamma_{\lfloor n/2 \rfloor}|,
$$
which follows since $\int_{\gamma_{i+1}}^{\gamma_i}\varrho(t) dt = n^{-1}$ and $\varrho$ is positive and monotonically decreasing on the positive real line. So, it suffices to prove that $\gamma_1 > \omega_n.$ Using the concavity of $\varrho$, we can lower bound the integral
\begin{align}
    \int^2_{\omega_n} \varrho(t) dt \geq \frac{1}{2} (2 - \omega_n) \cdot \varrho(\omega_n) = \frac{1}{2}(2 - \omega_n) \cdot n^{-1/3} \geq \frac{\pi^2}{4} n^{-1} > n^{-1}, 
\end{align}
where the first equality follows by the definition of $\omega_n$ and the second inequality follows by the lower bound in \eqref{eq:quantile-approx}. Thus, we derive that $\gamma_1 > \omega_n$. This concludes the proof of the lemma.
\end{proof}

Next, we establish bounds on the remaining terms in \eqref{eq:triangle-spacings}. To this end, we invoke the following result. This theorem was originally established in \cite{erdHos2012rigidity}; a more accessible proof can be found in \cite[Theorem 2.9]{benaych2016lectures}.
\begin{theorem}[\textbf{Wigner eigenvalue rigidity}]\label{thm:rigidity-wigner}
    Assume that $\mathbf{\bH}_n \in \RR^{n\times n}$ is a Wigner ensemble. Then, for all small $\varepsilon > 0$ and large $M$ we have that 
    \begin{equation}
        \sup_{i} \PP\left( |\lambda_i(\mathbf{H}_n) - \gamma_i| \geq n^{\varepsilon} \cdot n^{-2/3}\left(\min \left\{i, (n + 1 - i) \right\} \right)^{-1/3} \right) \leq n^{-M}
    \end{equation}
    for all large enough $n.$
\end{theorem}
To complete the proof of \cref{cor:wigner}, take $\varepsilon  < 1/6$ and $D \geq 3$. Define the events $$
\mathcal{E}_n = \left\{\text{For all }i \in [n], \,\, |\lambda_i(\mathbf{H}_n) - \gamma_i| \geq n^{\varepsilon} \cdot n^{-2/3}\left(\min \left\{i, (n + 1 - i) \right\} \right)^{-1/3}\right\}.
$$
Using the union bound in tandem with~\cref{thm:rigidity-wigner} 
we derive that $\PP(\mathcal{E}_n) \leq n^{1-M}$. By our choice of $M$, these probabilities are summable. So, the Borel-Cantelli Lemma ensures that, with probability one, only finitely many $\mathcal{E}_n$ occur. Thus, almost surely, there exists an $N$ such that for all $n \geq N$ we have  
\begin{equation}\label{eq:rigidity-conclusion}
|\lambda_i(\mathbf{H}_n) - \gamma_i| \leq n^{\varepsilon} \cdot n^{-2/3}\left(\min \left\{i, (n + 1 - i) \right\} \right)^{-1/3}\leq n^{\varepsilon -2/3} \quad \text{ for all } i \in [n].
\end{equation}
By our choice of $\varepsilon$, we obtain that $|\lambda_i(\mathbf{H}_n) - \gamma_i| = o(n^{-1/2}).$

Combining \eqref{eq:triangle-spacings}, \cref{lem:bound-quantiles}, and \eqref{eq:rigidity-conclusion} we conclude that for all large enough $n$, we have that 
$$
| \lambda_i(\mathbf{H}_n) - \lambda_{i+1}(\mathbf{H}_n)| \leq 2 n^{\varepsilon -2/3} + \pi^2 n^{-2/3} = o(n^{-1/2}) \quad \text{for all }i \in [n].
$$
This finishes the proof of the proposition. \hfill $\square$

\subsubsection{Proof of \cref{prop:covariance}}\label{proof:wishart-spacings}

The proof follows the exact same reasoning as that of~\cref{cor:wigner}. 
Again, for each $i \in [n]$, define the typical location $\gamma_i^{(n)}$ as 
$$
\int_{\gamma_i^{(n)}}^{c_+} \nu(t) dt = \frac{i}{n}, 
$$
where $\nu$ denotes the Marchenko-Pastur distribution \eqref{eq:mar-pas}. In order for the arguments to go through, we need an equivalent of Lemma~\ref{lem:bound-quantiles} and Theorem~\ref{thm:rigidity}. Let us state their counterparts for Wishart matrices. 
\begin{lemma}
    There exists a constant $C > 0$  depending only on $\vartheta$ such that for large enough $n,$
    $$| \gamma_i^{(n)} - \gamma_{i+1}^{(n)} | \leq C \cdot n^{-2/3} .$$
\end{lemma}
The proof parallels that of~\cref{lem:bound-quantiles}, and so we omit the details.
The eigenvalue rigidity of Wishart matrices was established in \cite{pillai2014universality}.

\begin{theorem}[\textbf{Wishart eigenvalue rigidity}]\label{thm:rigidity}
    Assume that $\mathbf{\bW}_n \in \RR^{n\times n}$ is a Wishart ensemble satisfying $\vartheta \neq 1.$ Then, for all small $\varepsilon > 0$ and large $M$ we have that 
    \begin{equation}
        \sup_{i \in [\min\{n_1, n_2\}]} \PP\left( |\lambda_i(\mathbf{H}_n) - \gamma_i| \geq n_1^{\varepsilon} \cdot n_1^{-2/3}\left(\min \left\{i, (\min\{n_1, n_2\} + 1 - i) \right\} \right)^{-1/3} \right) \leq n_1^{-M}
    \end{equation}
    for all large enough $n.$
\end{theorem}

The statement we transcribed here is a little weaker than that of \cite[Theorem 3.3]{pillai2014universality}, but it suffices for our purposes. Armed with these two results, the reasoning follows exactly the same logic as the proof of \cref{cor:wigner}. We omit the details.

\subsection{Asymmetric Case}

\subsubsection{Proof of~\cref{thm:main-as}} \label{sec:proof-main-as}

The proof is analogous to that of Theorem~\ref{thm:main}. Let us summarize the main changes one must implement for the same proof to follow. First, define the analogous FDR quantities~\eqref{eq:infty-FDR}, \eqref{eq:2}, \eqref{eq:upper-fdr}, and \eqref{eq:3} accordingly to account for the asymmetry. Then, use Theorem~\ref{thm:raj-as}, Proposition~\ref{prop:nice-edge-as}, and Proposition~\ref{prop:zero-correlation-as} in place of Theorem~\ref{thm:raj}, Proposition~\ref{prop:nice-edge}, and Proposition~\ref{prop:zero-correlation}. With these changes in place, almost all the same arguments follow, with the only exception being the equivalent of Lemma~\ref{lem:ratio-properties}, which we now state and prove. We omit all other details since they are equivalent to those in the proof of Theorem~\ref{thm:main}. %

\begin{lemma}\label{lem:ratio-properties-as} Consider a probability measure over the real line $\mu$ with bounded support such that $\supp(\mu) \subset [0, \infty)$ and let $b = \sup\{t : t \in \supp(\mu) \}$.
  Then, the map $\Phi\colon (b, \infty)\rightarrow \RR$ given
  by $\Phi(x) = 2\frac{D_{\mu}(x) \varphi_{\mu}(x; 1)}{D_{\mu}'(x)}$ yields a $[-1, 0]$-valued continuous nonincreasing function. 
  \end{lemma}

    \begin{proof}
  Recall that
  \begin{equation*}
 {\varphi}_{\mu_{\bE}}(x; s) := s \cdot \int \frac{x}{x^2 - t^2} d\mu(t) + \frac{1 - s}{x}
 \quad \text{and} \quad {\varphi}_{\mu}'(x;s) := - \left(s \cdot \int \frac{x^2 + t^2}{\left(x^2 - t^2\right)^2} d\mu(t) + \frac{1 - s}{x^2} \right),
\end{equation*}
for any $s \in (0,1]$ and, for a fixed $q\in (0,1]$ the $D$-transform and its derivative are given by
\begin{equation*}
    D_{\mu}(x) = \varphi_{\mu}(x; 1) \cdot \varphi_{\mu}(x; q).
    \quad \text{and} \quad D_{\mu}'(x) = \varphi_{\mu}'(x; 1) \cdot \varphi_{\mu}(x; q) + \varphi_{\mu}(x; 1) \cdot \varphi_{\mu}'(x; q).
  \end{equation*}
  Immediately we see that $\Phi(x) < 0$ because $\varphi_{\mu}(x, s) > 0$ and $\varphi_{\mu}'(x; s) < 0$ for any $s \in (0,1]$ and $x > b \geq 0$. Next, we show $\Phi(x) \geq -1.$ Let $\tilde \mu = q \mu + (1-q) \delta_0,$ then $\varphi_{\mu}(x; q) = \varphi_{\tilde \mu}(x,1).$ Expanding reveals that the following relations hold
  \begin{align}
      \label{eq:oh-nice}2 \EE_{t \sim \mu} \left[\frac{x^2}{(x^2 - t^2)^2}\right] &= \frac{\varphi_\mu(x, 1)}{x} - \varphi_\mu'(x, 1), \\
      \label{eq:oh-nice2} 2 \EE_{t \sim \tilde\mu} \left[\frac{t^2}{(x^2 - t^2)^2}\right] &= -\frac{\varphi_\mu(x, q)}{x} - \varphi_\mu'(x, q).
  \end{align}
  Then, 
  \begin{align*}
       2 D_{\mu}(x) \varphi_{\mu}(x; 1) &= 2 \varphi_{\mu}(x; q) \varphi_{\mu}(x; 1)^{2}\\ 
       &= 2 \varphi_{\mu}(x; q)\EE_{t\sim \mu} \left[\frac{x}{x^{2}-t^{2}}\right]^{2} \\ 
       &\leq 2 \varphi_{\mu}(x; q) \EE_{t\sim \mu}\left[\frac{x^{2}}{(x^{2}-t^{2})^{2}}\right] \\       
       &=  \frac{\varphi_{\mu}(x; q)\varphi_\mu(x; 1)}{x} - \varphi_{\mu}(x; q)\varphi'(x; 1)\\
       &= -\varphi_\mu(x; 1)\left[2 \EE_{t \sim \tilde\mu} \left[\frac{t^2}{(x^2 - t^2)^2}\right] + \varphi_\mu'(x, q) \right] - \varphi_{\mu}(x; q) \varphi'(x; 1) \\
       & \leq -\varphi_\mu(x; 1) \varphi_\mu'(x, q) - \varphi_{\mu}(x; q) \varphi'(x; 1) = - D'_{\mu}(x)
  \end{align*}
  where the third line uses Jensen's inequality, the fourth line uses \eqref{eq:oh-nice}, and the second to last line uses \eqref{eq:oh-nice2}. Multiplying by the negative number $1/ D'_{\mu}(x)$ yields $\Phi(x) \geq -1.$

  Finally, we prove that $\Phi(x)$ is nonincreasing. Since $\Phi(x) < 0$, the function $\Phi(\cdot)$ is nonincreasing if, and only if, $\Theta(\cdot) = 1/\Phi(\cdot)$ is nondecreasing. Let's show that $\Theta(\cdot)$ is nondecreasing. Expanding,
  $$
  \Theta(x) = \frac{1}{2}\Big(\underbrace{\frac{\varphi'_{\mu}(x;1)}{\varphi_{\mu}(x;1)^{2}}}_{T_{1}(x):=}
  + \underbrace{\frac{\varphi_{\mu}'(x;q)}{\varphi_{\mu}(x;q) \varphi(x;1)}}_{T_{2}(x;q):=} \Big).$$ 
  In turn, the first term $T_1$ is strictly increasing, while the second term $T_2$ is increasing for every $x$ when $q$ is close to one and decreasing for every $x$ when $q$ gets closer to zero. However, we shall see $T_1$ increases fast enough to guarantee that the whole function is nondecreasing for every $x$ for any fixed 
  $q$. The second term has the fastest decrease when $q$ goes to zero, as the next claim shows.  
  
  \begin{claim}\label{claim:omg}
      For any fixed $x$, we have that 
      $$T'_2(x;q) \geq T'_2(x;0)$$
      for all $q \in [0,1].$
  \end{claim}
  We defer the proof of this claim to the end of this subsection
  and focus on the extreme case $q = 0$. In this case, the function $\Theta$ reduces to 
  $$
  \Theta(x) = \frac{1}{2}\left({\frac{x\varphi'_{\mu}(x;1) -  \varphi_{\mu}(x;1)}{x\varphi_{\mu}(x;1)^{2}}}\right).$$ 
  Taking a derivative yields 
  \begin{align*}
        \Theta'(x)  
        &= \left({\frac{ x^2 \left(\varphi''_{\mu}(x;1)\varphi_{\mu}(x;1)^{2} - 2\varphi_{\mu}(x;1) \varphi_{\mu}'(x;1)^2 \right) +\varphi_{\mu}(x;1)^{3} + x \varphi_{\mu}(x;1)^{2} \varphi_{\mu}'(x;1) }{2x^2\varphi_{\mu}(x;1)^{4}}}\right)\\
        &= \left({\frac{ x^2 \left(\varphi''_{\mu}(x;1)\varphi_{\mu}(x;1)^{2} - 2\varphi_{\mu}(x;1) \varphi_{\mu}'(x;1)^2 \right) - 2x \varphi_{\mu}(x;1)^{2}\EE\left( \frac{t^2}{(x^2 - t^2)^2}\right) }{2x^2\varphi_{\mu}(x;1)^{4}}}\right)\\
        &= \left({\frac{ x^2\varphi_{\mu}(x;1) \left(  \left(\varphi''_{\mu}(x;1) - 2\EE\left( \frac{t^2/x}{(x^2 - t^2)^2}\right) \right) \varphi_{\mu}(x;1) - 2\varphi_{\mu}'(x;1)^2 \right)  }{2x^2\varphi_{\mu}(x;1)^{4}}}\right)
  \end{align*}
  where the second equality follows by~\eqref{eq:oh-nice2}. 
Thus, to show that $\Theta$ is nondecreasing, it suffices to show that the numerator is positive. We upper bound
\begin{align*}
  2\varphi_{\mu}'(x;1)^{2} =2\EE\left[\frac{x^{1/2}}{(x^{2}-t^{2})^{1/2}}\frac{\left(x^{3/2}+\frac{t^{2}}{x^{1/2}}\right)}{(x^{2}-t^{2})^{3/2}} \right]^{2} 
  &\leq2\EE\left[\frac{x}{x^{2}-t^{2}}\right] \EE \left[\frac{\left(x^{3/2}+\frac{t^{2}}{x^{1/2}}\right)^{2}}{(x^{2}-t^{2})^{3}} \right] \\
&=2\EE\left[\frac{x}{x^{2}-t^{2}}\right] \EE \left[\frac{x\left(x+\frac{t^{2}}{x}\right)^{2}}{(x^{2}-t^{2})^{3}} \right] \\
  &=2\EE\left[\frac{x}{x^{2}-t^{2}}\right] \EE \left[\frac{x\left(x^{2}+2t^{2}+\frac{t^{4}}{x^{2}}\right)}{(x^{2}-t^{2})^{3}} \right] \\
                                             &=  2\EE\left[\frac{x}{(x^{2}-t^{2})} \right]\left( \EE\left[\frac{x(x^{2}+3t^{2})}{(x^{2}-t^{2})^{3}}\right] -\EE\left[\frac{ x(t^{2} - t^4/x^2)}{(x^{2}-t^{2})^{3}}\right] \right) \\
  & = \varphi_{\mu}(x; 1) \left(\varphi_{\mu}''(x; 1)  - 2 \EE\left[\frac{ t^2/x}{(x^{2}-t^{2})^{2}}\right]\right)
\end{align*}
where the first inequality follows by Cauchy-Schwarz. This establishes that $\Theta$ is nondecreasing, and so $\Phi$ is nonincreasing, thus proving the Lemma.
  \end{proof}

\paragraph{Proof of Claim~\ref{claim:omg}.} Taking derivatives 
      $$T_2'(x;q)= \frac{\varphi''(x;q)\varphi(x;q)\varphi(x;1) - \varphi'(x;q)\left(\varphi'(x;q)\varphi(x;1) + \varphi(x;q)\varphi'(x;1)\right)}{\varphi(x;q)^2\varphi(x;1)^2}.$$
      Evaluating at $q = 0$ yields
      $$T_2'(x;0)= \frac{\varphi(x;1) + x\varphi'(x;1)}{x^2\varphi(x;1)^2}.$$
      To simplify the notation, we drop the dependency on $x$ via $\varphi_q := \varphi(x; q).$ Thus, 
      \begin{align*}
          T_2'(x;q) - T_2'(x;0) & = \frac{x^2\varphi''_q\varphi_q\varphi_1 - x^2\varphi'_q\left(\varphi'_q\varphi_1 + \varphi_q\varphi'_1\right)-\varphi_q^2 \left(\varphi_1 + x\varphi'_1\right) }{x^2\varphi_q^2\varphi_1^2}.
      \end{align*}
      It suffices to show that the numerator is nonnegative. Reorganizing yields that the numerator is equal to
      \begin{align*}
          &\left(x^2\varphi_q'' \varphi_q - x^2 (\varphi_q')^2 - \varphi_q^2\right)\varphi_1 - \left(x\varphi_q^2 + x^2 \varphi'_q \varphi_q\right) \varphi_1'\\ 
          & \geq \left(x^2\varphi_q'' \varphi_q/2 - \varphi_q^2\right)\varphi_1 - \left(x\varphi_q + x^2 \varphi'_q \right)\varphi_q \varphi_1' \\
          & = \left[\left(x^2\varphi_q''/2 - \varphi_q\right)\varphi_1 + 2\EE_{\tilde \mu}\left(\frac{x^2t^2}{(x^2 - t^2)^2}\right)   \varphi_1'\right]\varphi_q \\
          & = \left[\EE_{\tilde \mu}\left(\frac{xt^2 (5x^2 - t^2)}{(x^2 - t^2)^3}\right)\varphi_1 + 2\EE_{\tilde \mu}\left(\frac{xt^2}{(x^2 - t^2)^2}\right)   x\varphi_1'\right]\varphi_q \\
        & = \left[\EE_{\tilde \mu}\left(\frac{xt^2 (5x^2 - t^2)}{(x^2 - t^2)^3}\right)\varphi_1 - 2\EE_{\tilde \mu}\left(\frac{xt^2}{(x^2 - t^2)^2}\right)\left( \varphi_1 + 2 \EE_\mu \left(\frac{xt^2}{(x^2 -t^2)^2}\right)\right)\right]\varphi_q \\
        & = \left[\EE_{\tilde \mu}\left(\frac{xt^2 (3x^2 + t^2)}{(x^2 - t^2)^3}\right)\varphi_1 - 4\EE_{\tilde \mu}\left(\frac{xt^2}{(x^2 - t^2)^2}\right)\left( \EE_\mu \left(\frac{xt^2}{(x^2 -t^2)^2}\right)\right)\right]\varphi_q \\
        & =q \left[\EE_{\mu}\left(\frac{xt^2 (3x^2 + t^2)}{(x^2 - t^2)^3}\right)\varphi_1 - 4x^2\EE_{\mu}\left(\frac{t^2}{(x^2 - t^2)^2}\right)^2\right]\varphi_q \\
         & \geq q \left[\EE_{\mu}\left(\frac{xt^2 (3x^2 + t^2)}{(x^2 - t^2)^3}\right)\varphi_1 - 4x^2\EE_{\mu}\left(\frac{t^3}{(x^2 - t^2)^3}\right)\EE_{\mu}\left(\frac{t}{(x^2 - t^2)}\right)\right]\varphi_q \\
         & \geq q \left[\EE_{\mu}\left(\frac{xt^2 (3x^2 + t^2)}{(x^2 - t^2)^3}\right)\varphi_1 - 4x^2\EE_{\mu}\left(\frac{t^3}{(x^2 - t^2)^3}\right)\varphi_1\right]\varphi_q \\
         & = q \left[\EE_{\mu}\left(\frac{xt^2 (3x^2 + t^2 - 4xt)}{(x^2 - t^2)^3}\right)\varphi_1 \right]\varphi_q \\
         & = q \EE_{\mu}\left(\frac{xt^2 (3x-t)(x-t)}{(x^2 - t^2)^3}\right)\varphi_1 \varphi_q \\
         & \geq 0
      \end{align*}
This concludes the proof.

\subsubsection{Proof of~\cref{thm:rank-estimate-as}}\label{sec:proof-rank-estimate-as}
By singular-value interlacing \cite[Theorem 3.1.2]{roger1994topics}, we have that for all $i \geq n$
\begin{equation}
    \label{eq:inter}
    \sigma_{i+r}(\bE_n) \leq \sigma_i(\bX_n) \leq \sigma_{i-r}(\bE_n)
\end{equation} 
with the convention that $\sigma_k(\bX) = \infty$ if $k \leq 0$. Fix an index $\lceil n/2 \rceil \geq j > r;$ then, using the interlacing inequalities combined with a telescoping sum, we obtain that 
\begin{align*}
    \Delta_j = \sigma_{j-1}(\bX_n) - \sigma_{j}(\bX_n) &\leq \sigma_{j-1-r}(\bE_n) - \sigma_{j+r}(\bE_n) \\
    &= \sum_{i = -r}^{r} \left(\sigma_{j-1-i}(\bE_n) - \sigma_{j-i}(\bE_n)\right)\\
    & \leq 2r \max_{i \in [n-1]}\left\{\sigma_{i}(\bE_n) - \sigma_{i+1}(\bE_n)\right\}.
\end{align*}
Recall the labeling $s_i := \sigma_{i}(\bE_n) - \sigma_{i+1}(\bE_n).$ Assumption~\ref{ass:spacings} ensures that for any $\varepsilon >0 $ there exists an $N > 0$ such that for all $n \geq N$, we have that $\max_{i \in [n-1]}\left\{s_i\right\} \leq \varepsilon \cdot
n^{-1/2}$. Thus, by taking $\varepsilon = p/4r$, we obtain 
\begin{equation}
    \Delta_j \leq \frac{p}{2} \cdot
    n^{-1/2} < 
    p \cdot  n^{-1/2} \quad \text{for all }j>r.
\end{equation}
Thus, we conclude that for all large enough $n$,  
$\textsc{RankEstimate}(\bX_{n}) \leq r$. Next, let us show the lower bound of $r^\star.$ By \cref{thm:raj} we have that $\lambda_{r^\star}(\bX) \rightarrow \rho_{r^\star} > b.$ Furthermore by Theorem~\ref{thm:raj-as} we have that $\sigma_{{r^\star} + 1}(\bX_n)$ almost surely converges to $b$. Thus, for large enough $n$, we obtain that $$\sigma_{r^\star}(\bX_n) - \sigma_{r^\star + 1}(\bX_n) \geq \frac{\rho_{r^\star} - b}{2} > 
p \cdot n^{-1/2}$$
where the last inequality holds for any $n$ large since $n^{-1/2}$ tends to zero. Hence $\textsc{RankEstimate}(\bX_{n}) \geq r^\star$
This concludes the proof of \cref{thm:rank-estimate-as}.
\hfill $\square$

\subsection{Proof of Proposition~\ref{prop:data}} \label{sec:proof-data}
 By definition, the product $\bW_n = \bE_n \bE_n^\top$ is a Wishart ensemble; moreover, the left singular vectors of $\bE_n$ match the eigenvectors of $\bW_n.$ Then, invoking~\cref{prop:covariance} we obtain that $\bW_n$ satisfies Assumptions~\ref{ass:model2}, and~\ref{ass:spacings}. Combining this with the fact that $\sigma_i^2(\bE_n) = \lambda_i(\bW_n)$ shows that $\{\bE_n\}$ satisfies Assumption~\ref{ass:amodel1} with a limiting singular value distribution $\mu_{\bE}$ that is characterized by the density 
      $$
      f_{\bE}(t) = 2t \cdot \nu(t^2) = \frac{1}{\pi \vartheta t} \sqrt{\left((c_+ - t^2)(t^2 - c_-)\right)_+} \quad \text{with}\quad c_{\pm} = \left(1 \pm \sqrt{\vartheta}\right)^2,
      $$
      which is supported on the interval $\left(\sqrt{c_-}, \sqrt{c_+}\right).$ Thus, the first part of Assumption~\ref{ass:spacings-as} also holds true. To prove that the second part of Assumption~\ref{ass:spacings-as} holds, recall that $\vartheta < 1$. Therefore, for all large $n$ and $i \leq n$ we have $\sigma_i(\bX_n) \in \left(\frac{\sqrt{c_-}}{2}, 2\sqrt{c_+}\right)$. Within that interval, the function $x \mapsto \sqrt{x}$ is $L$-Lipschitz for some $L > 0$. Therefore,
      $$
      |\sigma_i(\bE_n) - \sigma_{i+1}(\bE_n)| \leq L \cdot |\sigma_i(\bE_n)^2 - \sigma_{i+1}(\bE_n)^2| =  L \cdot |\lambda_i(\bW_n) - \lambda_{i+1}(\bW_n)| = o(n^{-1/2}),
      $$
      where the last relation follows since~\cref{ass:spacings} holds for $\bW_n;$ see Proposition~\ref{prop:covariance}. This completes the proof.
\hfill $\square$

\section{Future Directions} \label{sec:future}

We presented methodology for subspace selection in PCA in which the objective is to obtain subspace estimates that come with FDR control guarantees.  Although subspace estimation has a long history, the perspective in our paper is grounded in multiple testing and this is not common in the literature.  Our methods are free of tuning parameters and although our analysis is asymptotic in nature, the empirical results showcase strong performance even for modest-sized problem instances.

There are a number of research directions that are suggested by our work.  It is of interest to evaluate the extent to which some of our method yields non-asymptotic FDR control; progress in this direction likely require new advances in random matrix theory.  More generally, subspace estimation arises in many applications in which one may only have partial observations of a data matrix (e.g., matrix completion and collaborative filtering), and generalizing our methods to such settings would broaden their applicability.  Finally, in recent work~\cite{taeb2024model} natural notions of FDR, generalizing those for multiple testing and subspace selection, were obtained for a variety of model selection problems such as clustering, ranking, and causal inference.  Developing methodology for controlling FDR in these contexts would substantially expand the scope of the FDR control paradigm.

\section*{Acknowledgements}
We are profoundly grateful to Jorge Garza-Vargas for many insightful conversations during the development of this work.  VC was supported in part by AFOSR grants FA9550-23-1-0204, FA9550-23-1-0070 and NSF grant DMS 2113724.
\bibliographystyle{abbrvnat}
\bibliography{biblio}
\appendix
\section{Additional Numerical Experiments}\label{appendix:experiments}
In this Appendix, we include the additional numerical experiments missing from Section~\ref{sec:experiments}. In particular, we include plots for all the ensembles in Table~\ref{table:noise}. The setup here is exactly the same as the one described in Section~\ref{sec:generated}. The proportions of all the experiments with rectangular matrices are set $n/m = 1/2.$ 

 \begin{figure}[t!]
    \def\arraystretch{0}%
    \setlength{\imagewidth}{\dimexpr \textwidth - 4\tabcolsep}%
    \divide \imagewidth by 3
    \hspace*{\dimexpr -\baselineskip - 2\tabcolsep}%
    \begin{tabular}{@{}cIII@{}}
      \stepcounter{imagerow}\raisebox{0.35\imagewidth}{\rotatebox[origin=c]{90}%
        {\strut \texttt{well-separated}}} &
      \includegraphics[width=\imagewidth]{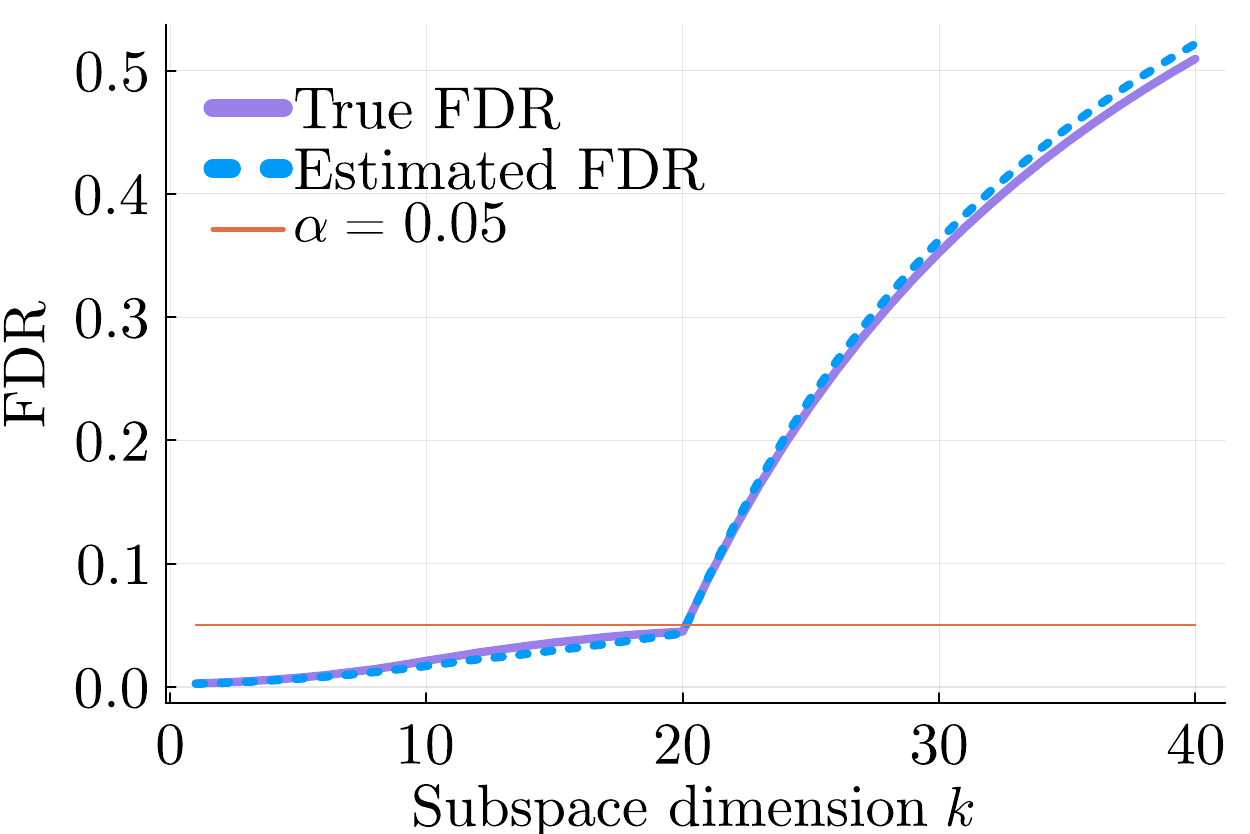} &
      \includegraphics[width=\imagewidth]{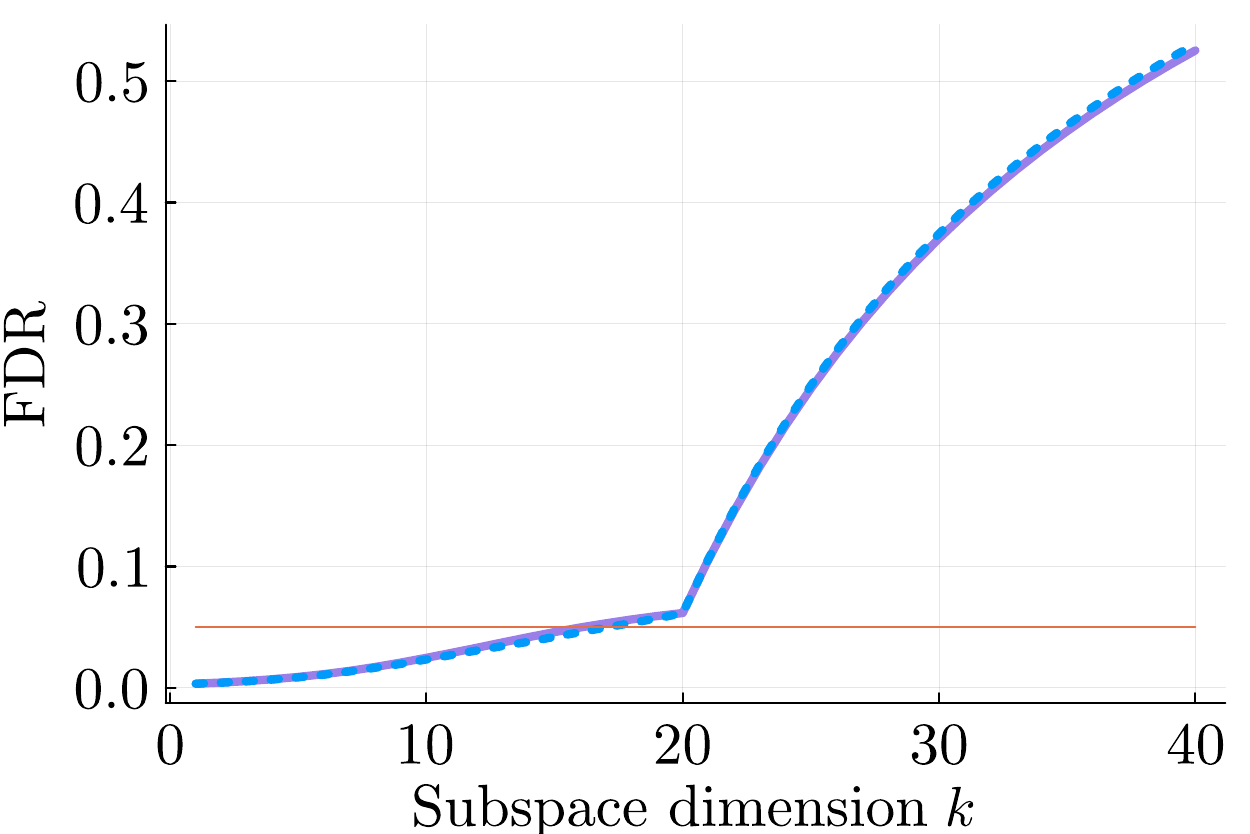}\label{example} &
      \includegraphics[width=\imagewidth]{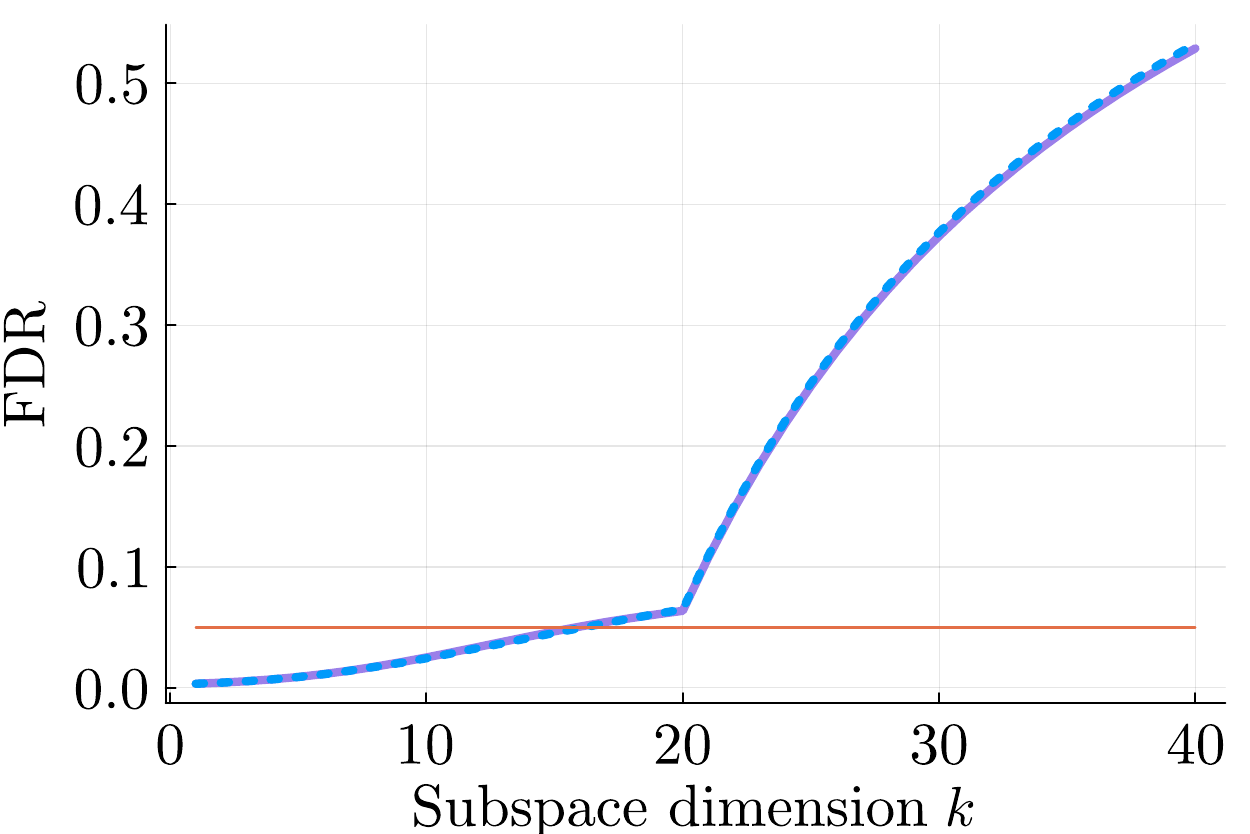} \\[2\tabcolsep]%
      \stepcounter{imagerow}\raisebox{0.35\imagewidth}{\rotatebox[origin=c]{90}%
        {\strut \texttt{barely-separated}}} &
      \includegraphics[width=\imagewidth]{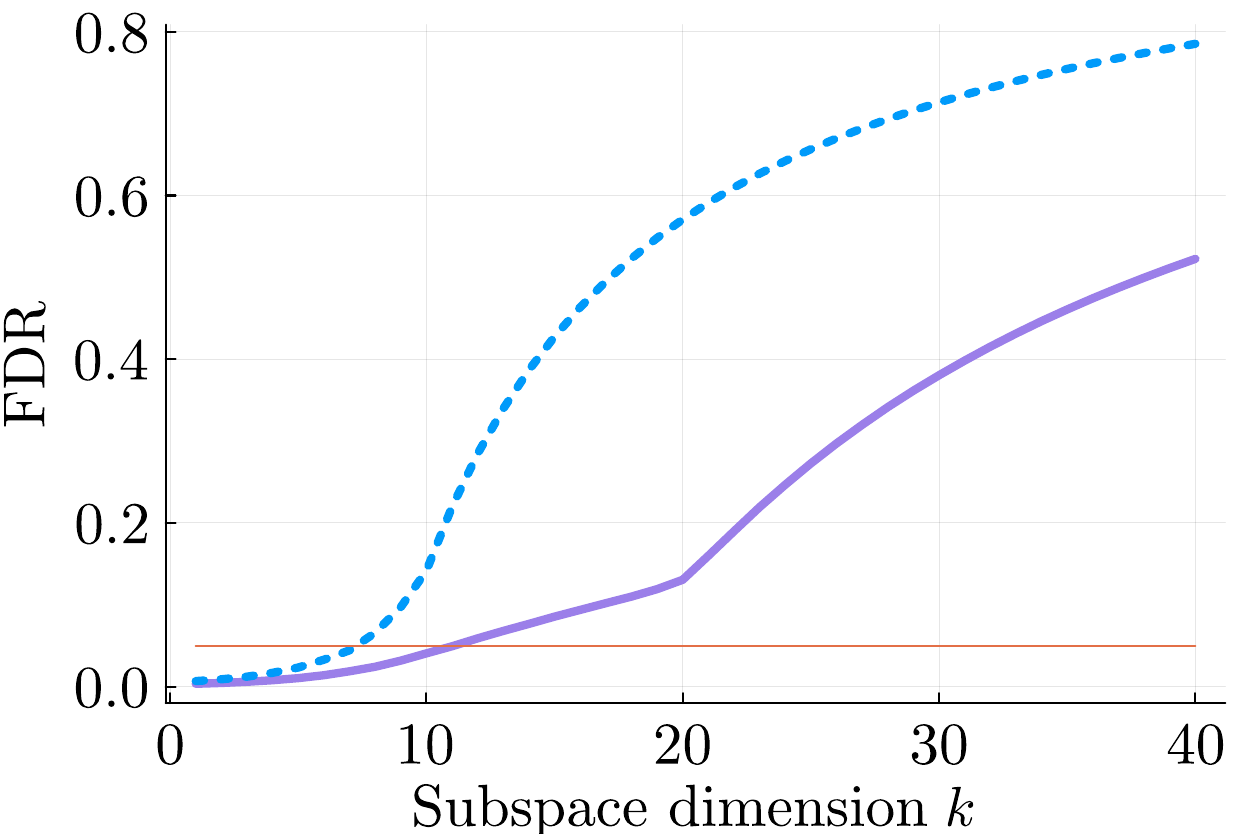} &
      \includegraphics[width=\imagewidth]{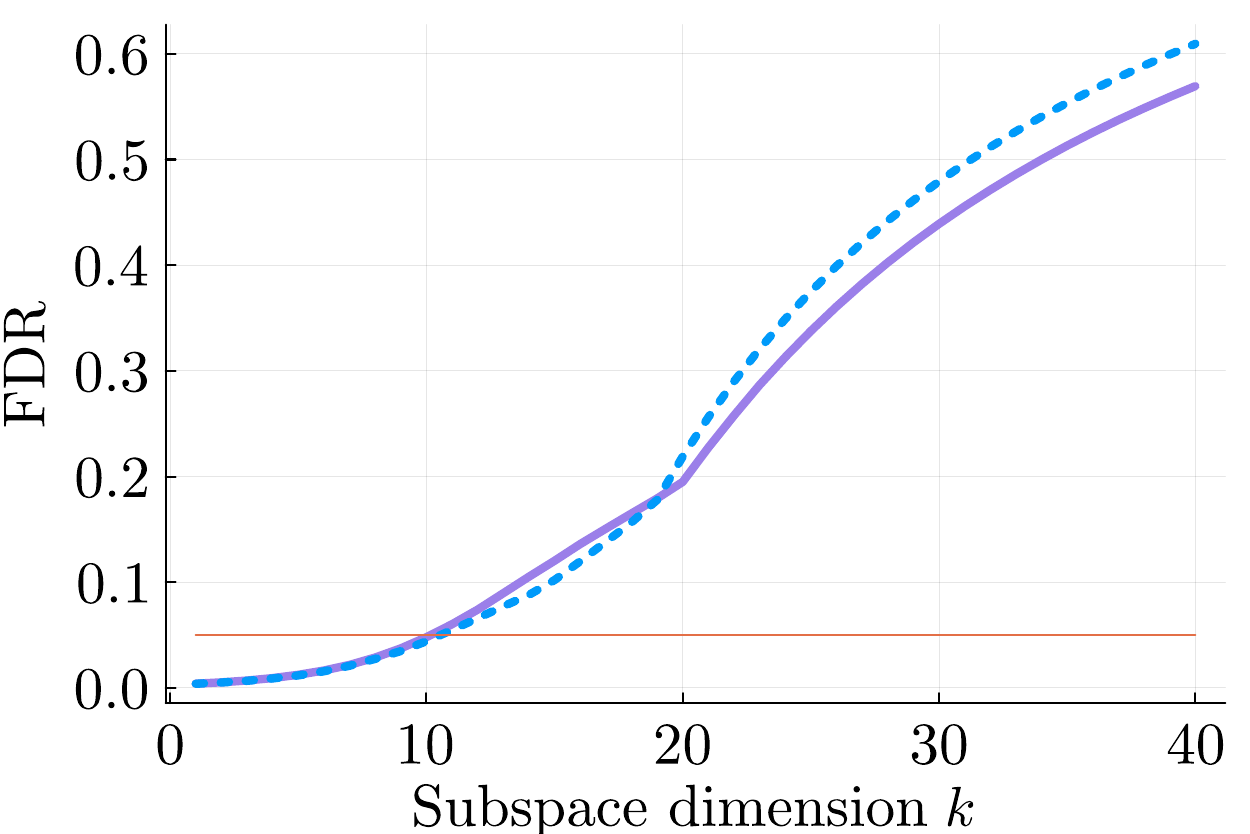} &
      \includegraphics[width=\imagewidth]{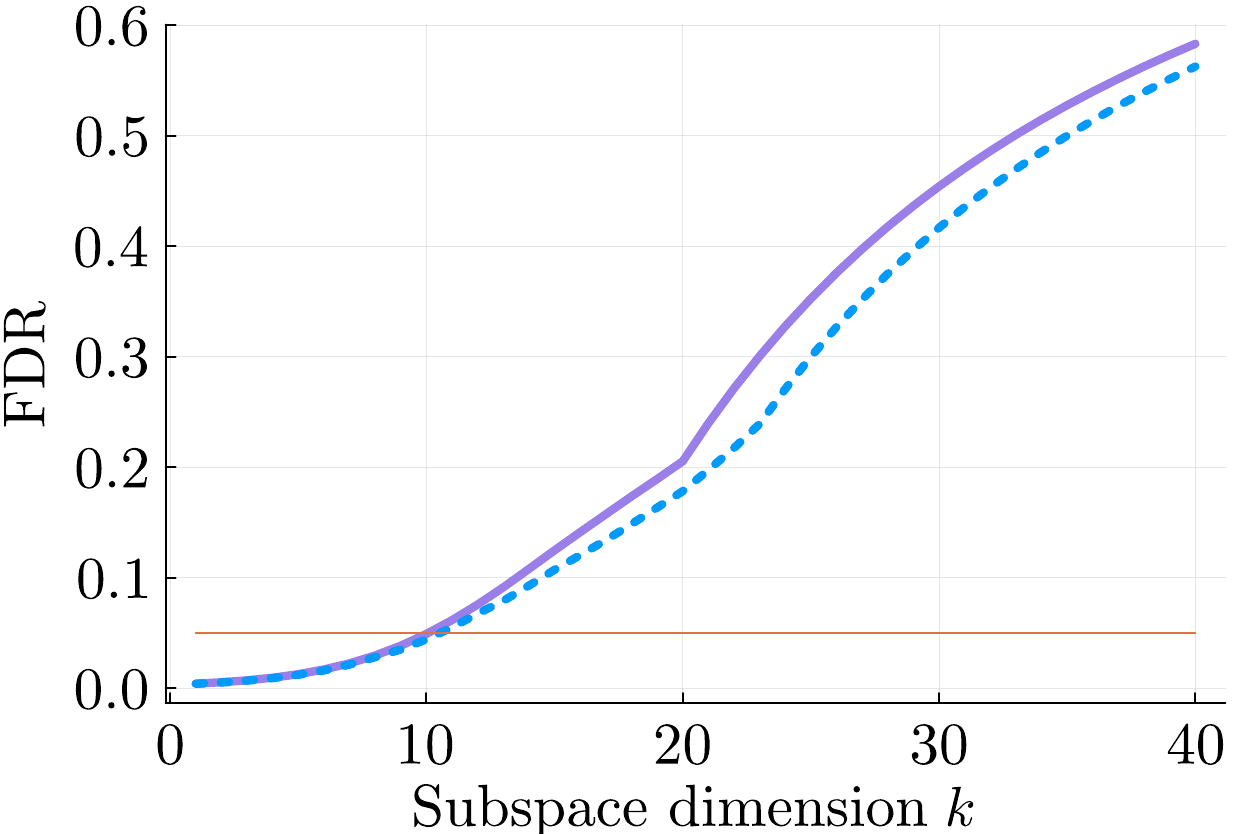} \\[2\tabcolsep]
      \stepcounter{imagerow}\raisebox{0.35\imagewidth}{\rotatebox[origin=c]{90}%
     {\strut \texttt{entangled}}} &
      \includegraphics[width=\imagewidth]{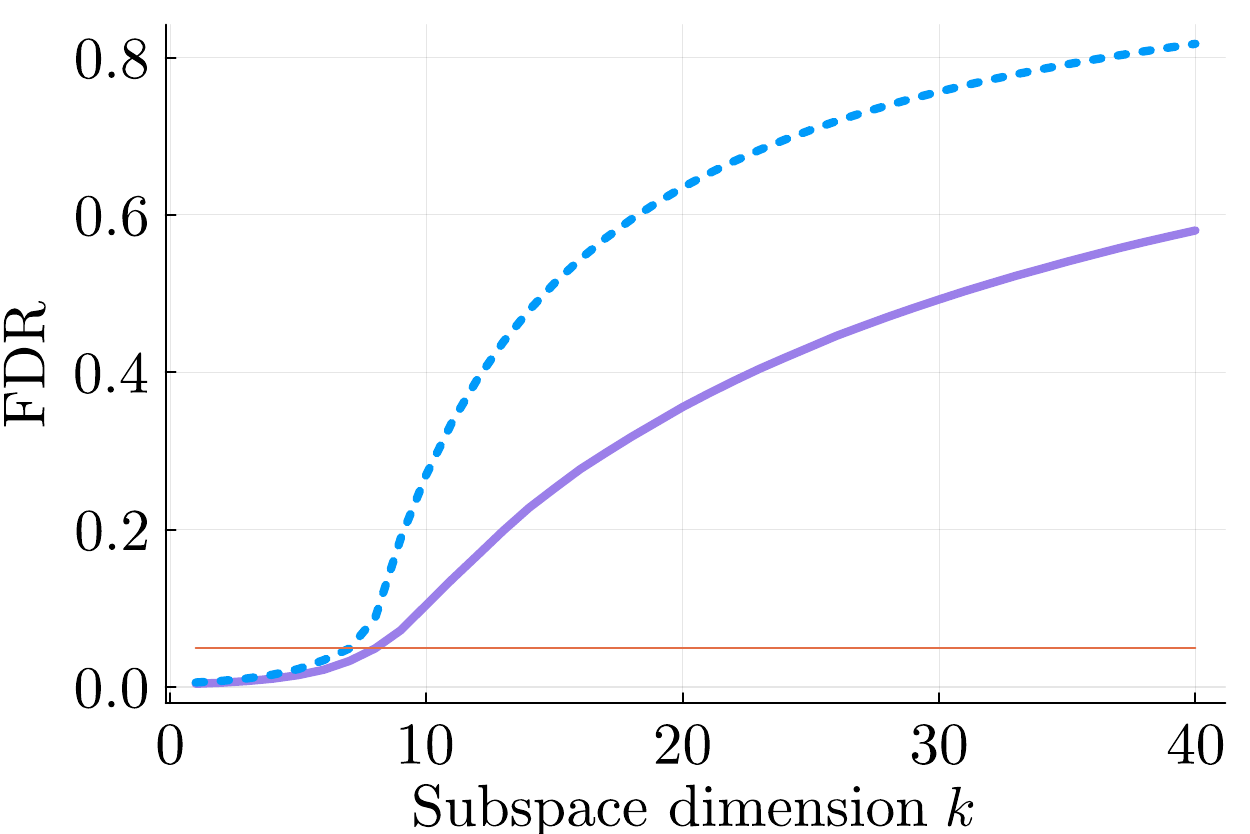} &
      \includegraphics[width=\imagewidth]{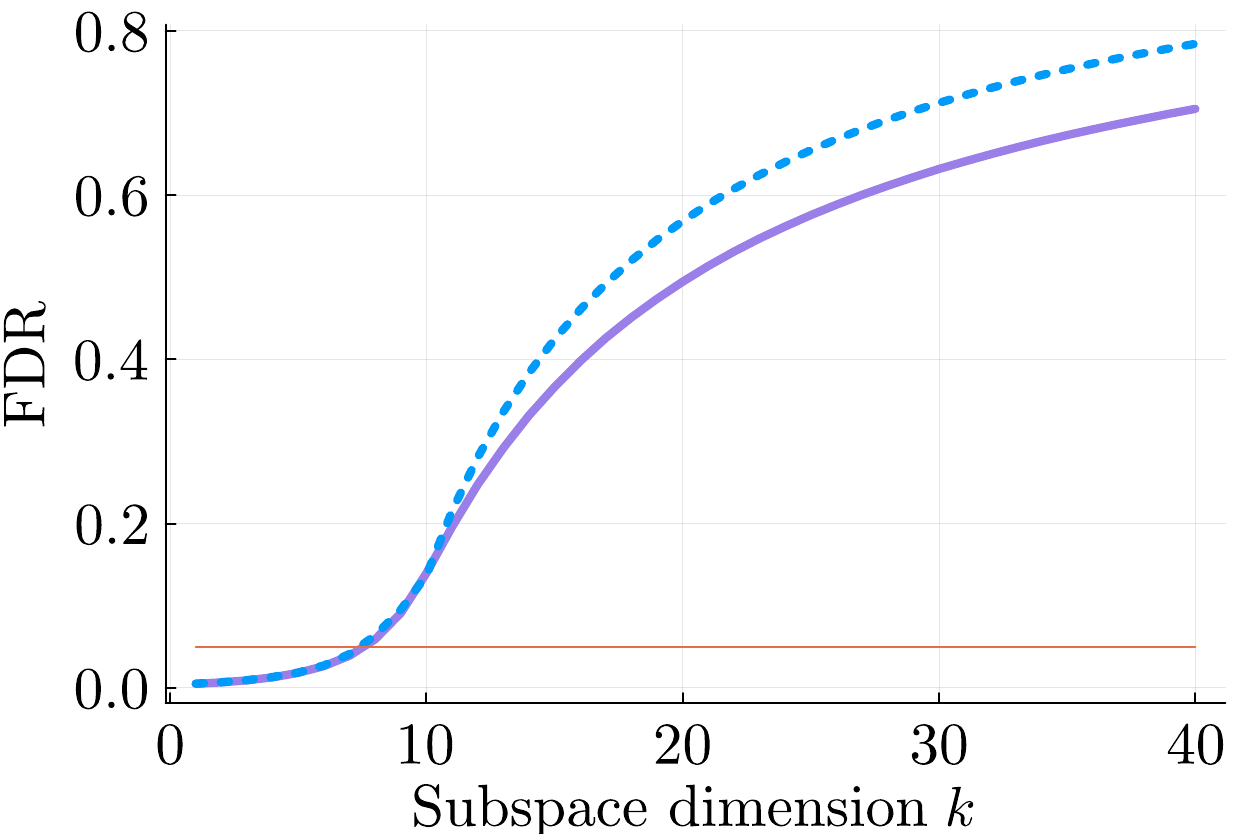} &
      \includegraphics[width=\imagewidth]{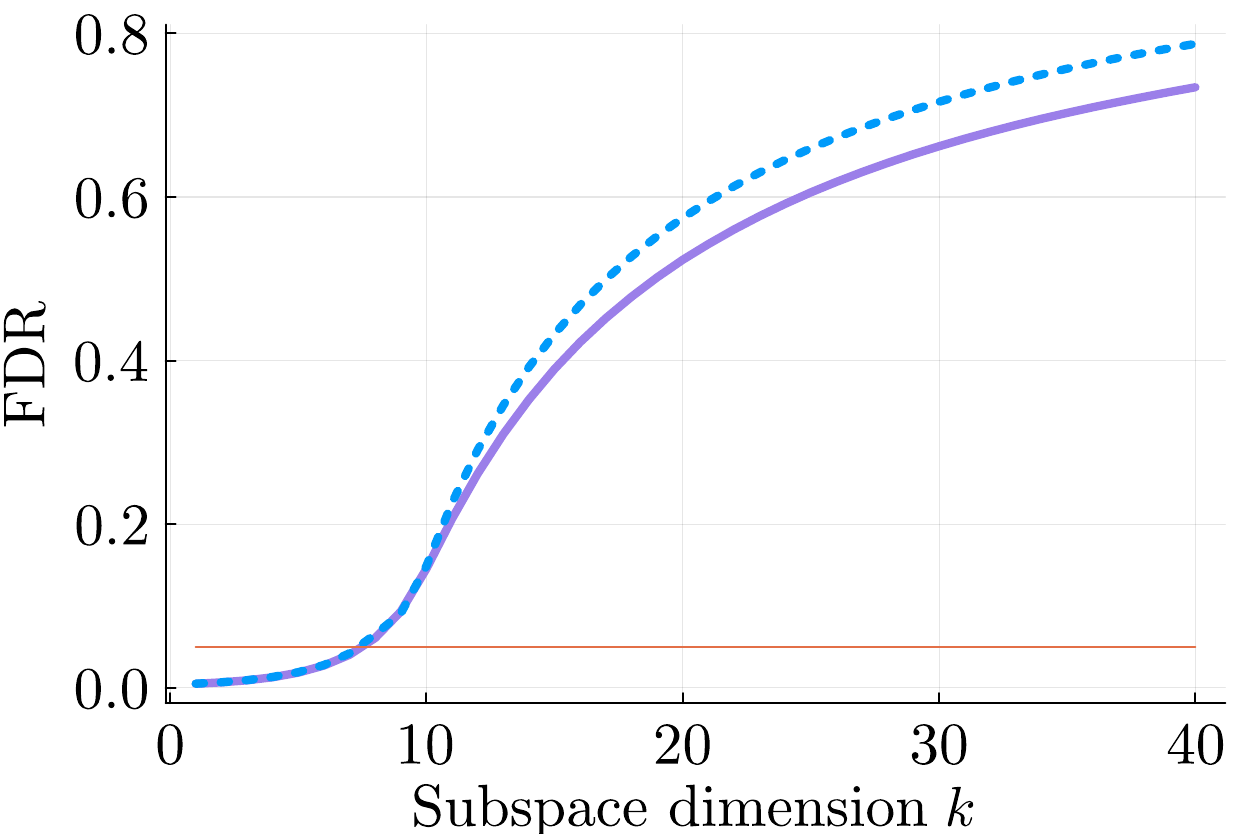} \\[2\tabcolsep]
        \setcounter{imagecolumn}{0} &%
        Dimension $n = 100$ &
        Dimension $n = 500$ &
        Dimenson $n = 1000$
    \end{tabular}
    \caption{Results for estimating the FDR of the \texttt{Wishart} ensemble using Algorithms~\ref{alg:FDR} and~\ref{alg:rank-estimate}.}
	\label{fig:res-wishart}
\end{figure}
    \begin{figure}[t]
    \def\arraystretch{0}%
    \setlength{\imagewidth}{\dimexpr \textwidth - 4\tabcolsep}%
    \divide \imagewidth by 3
    \hspace*{\dimexpr -\baselineskip - 2\tabcolsep}%
    \begin{tabular}{@{}cIII@{}}
      \stepcounter{imagerow}\raisebox{0.35\imagewidth}{\rotatebox[origin=c]{90}%
        {\strut \texttt{well-separated}}} &
      \includegraphics[width=\imagewidth]{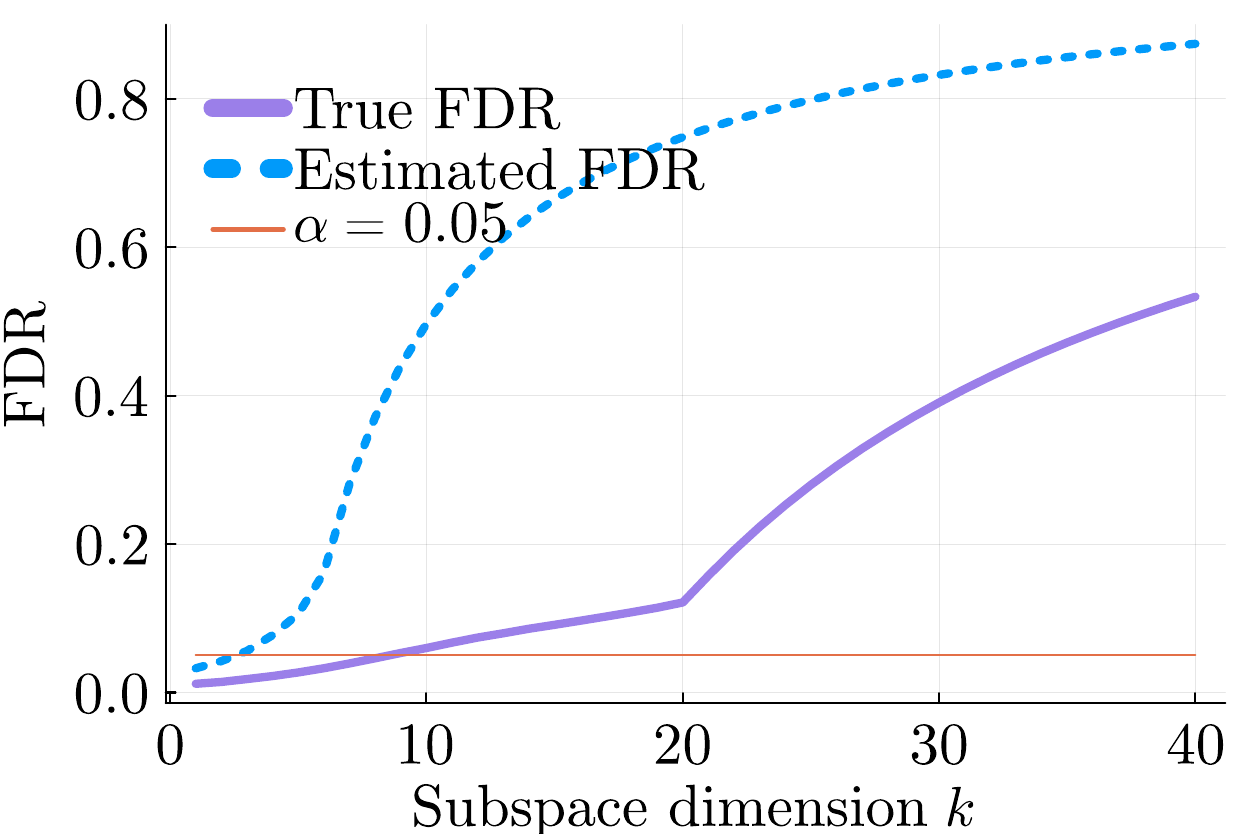} &
      \includegraphics[width=\imagewidth]{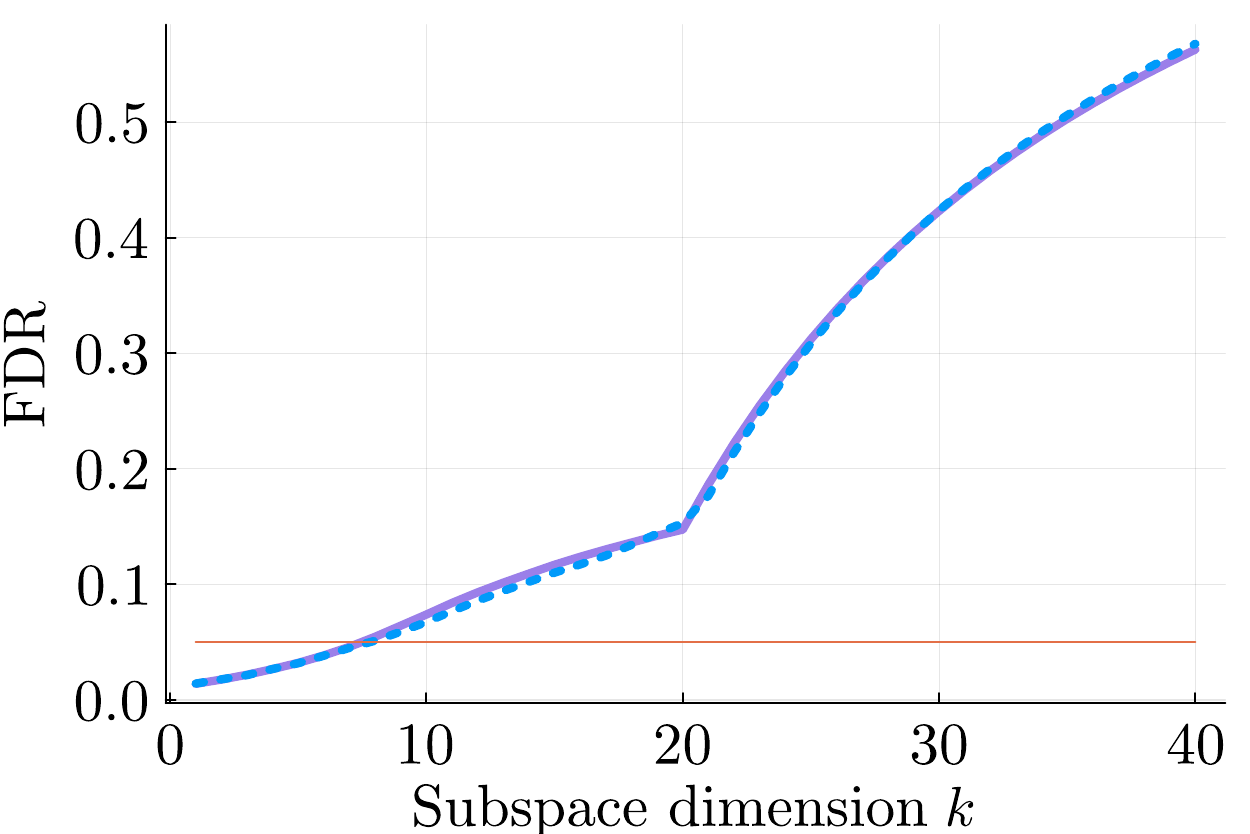}\label{example} &
      \includegraphics[width=\imagewidth]{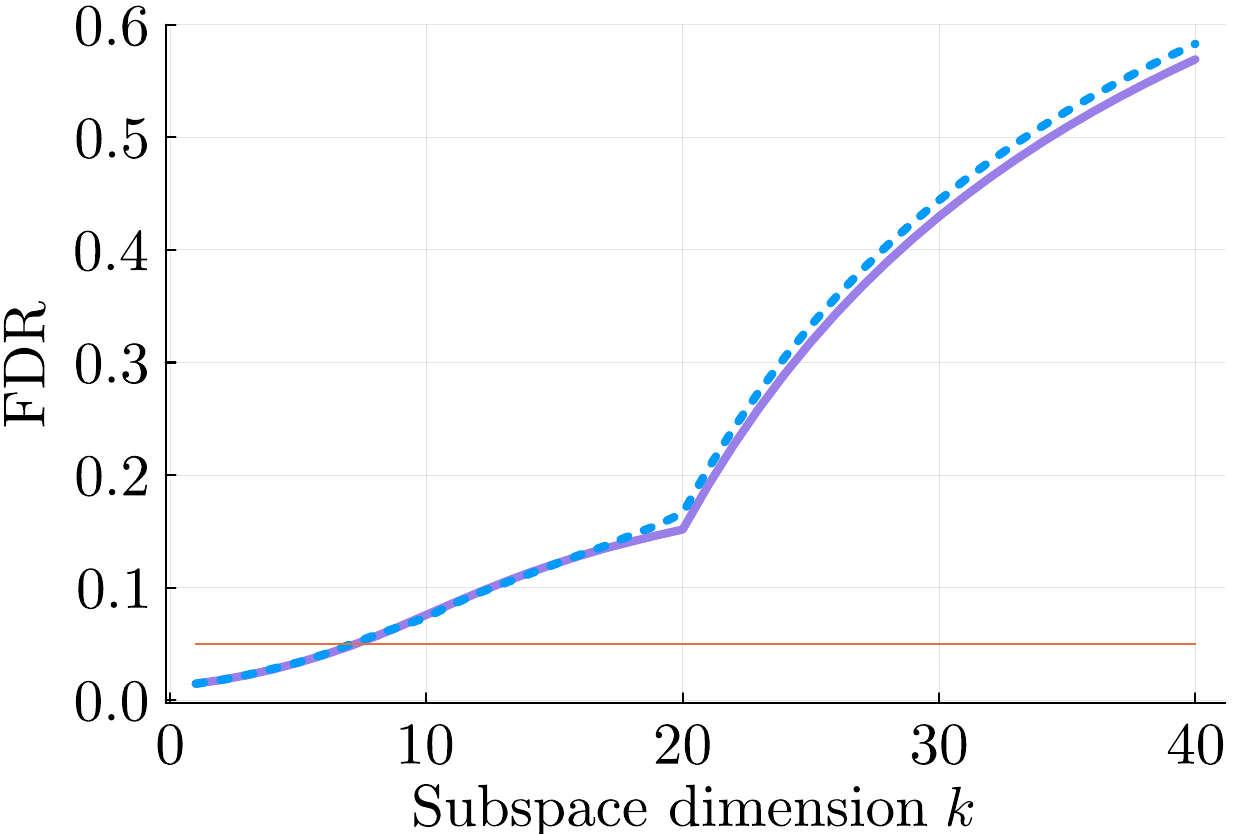} \\[2\tabcolsep]%
      \stepcounter{imagerow}\raisebox{0.35\imagewidth}{\rotatebox[origin=c]{90}%
        {\strut \texttt{barely-separated}}} &
      \includegraphics[width=\imagewidth]{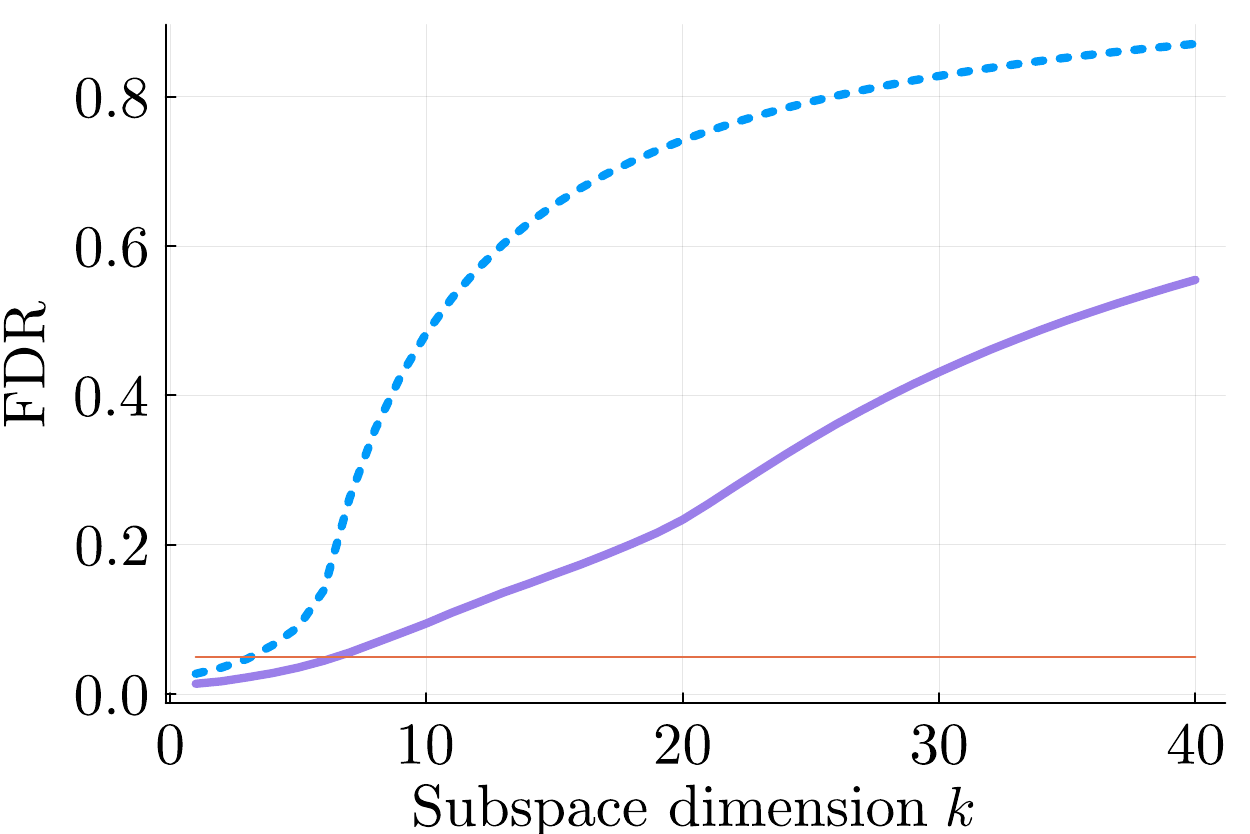} &
      \includegraphics[width=\imagewidth]{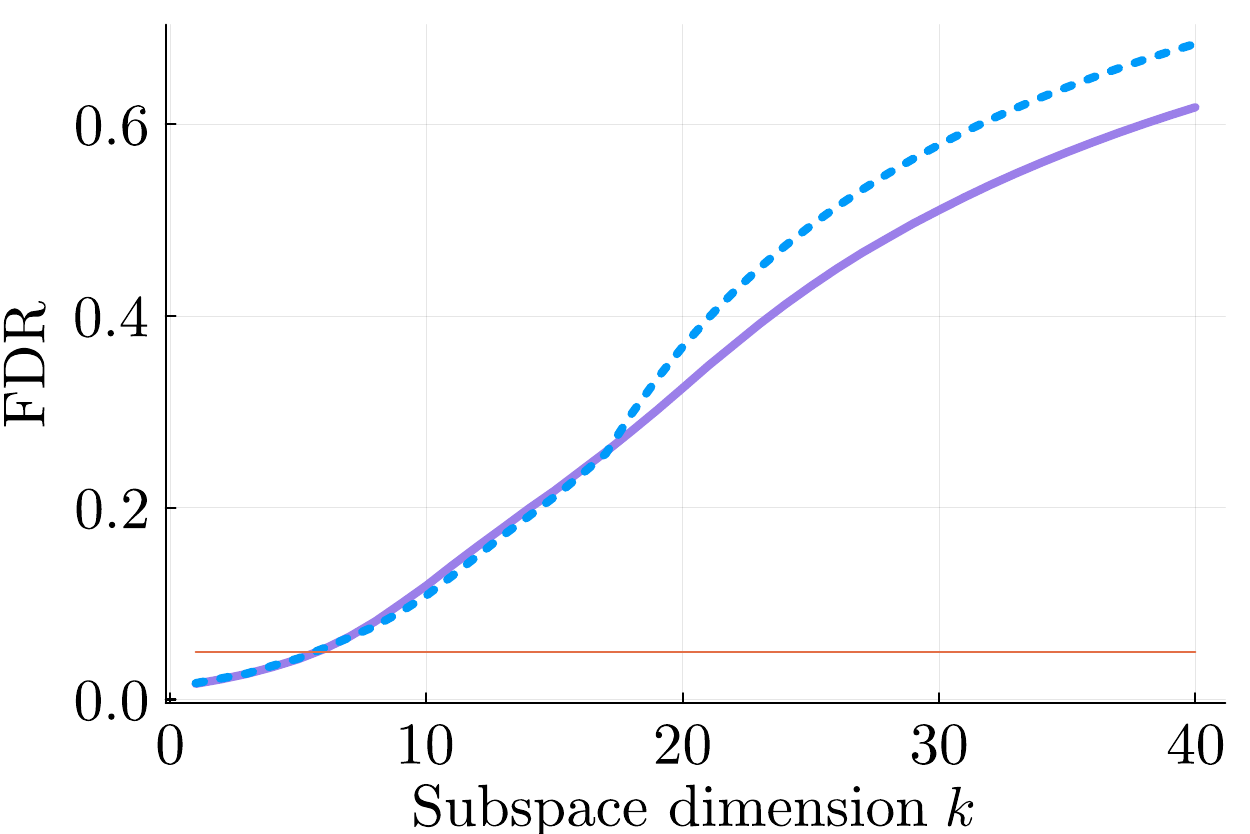} &
      \includegraphics[width=\imagewidth]{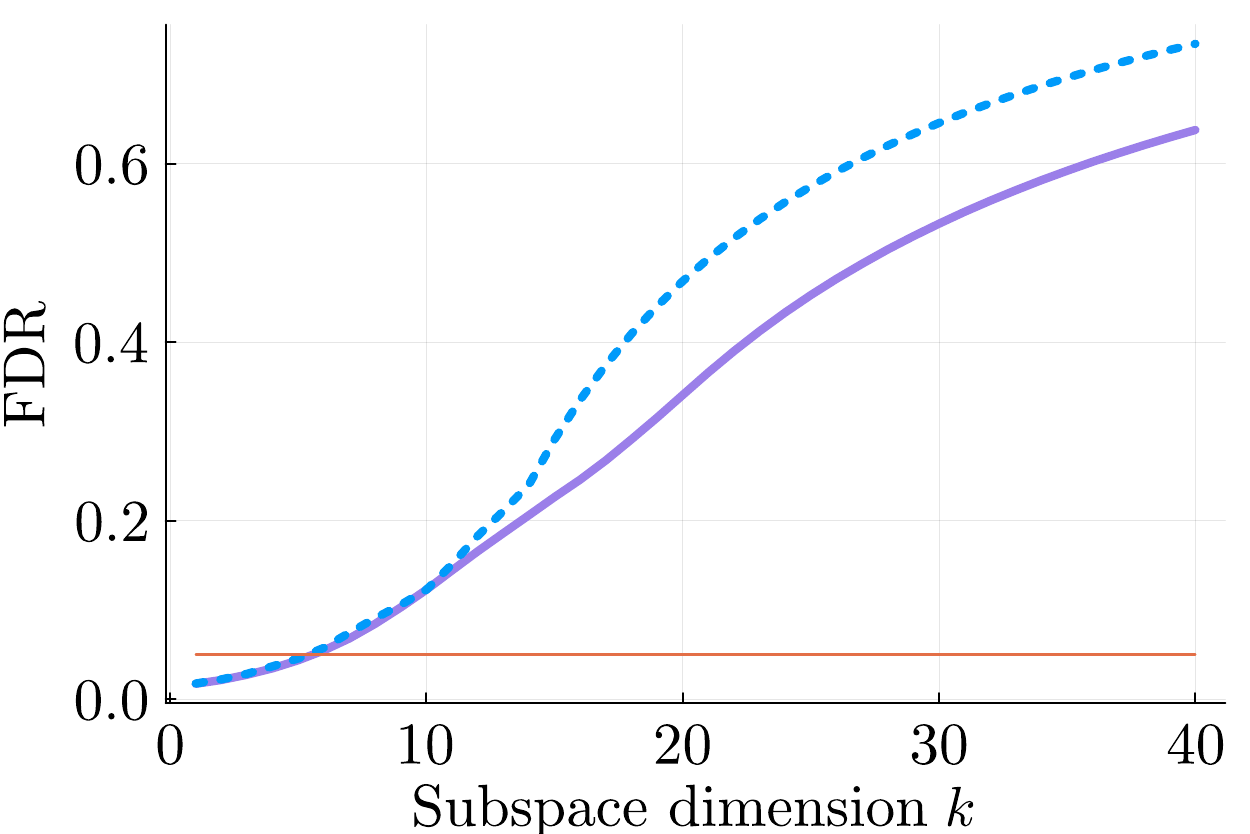} \\[2\tabcolsep]
      \stepcounter{imagerow}\raisebox{0.35\imagewidth}{\rotatebox[origin=c]{90}%
     {\strut \texttt{entangled}}} &
      \includegraphics[width=\imagewidth]{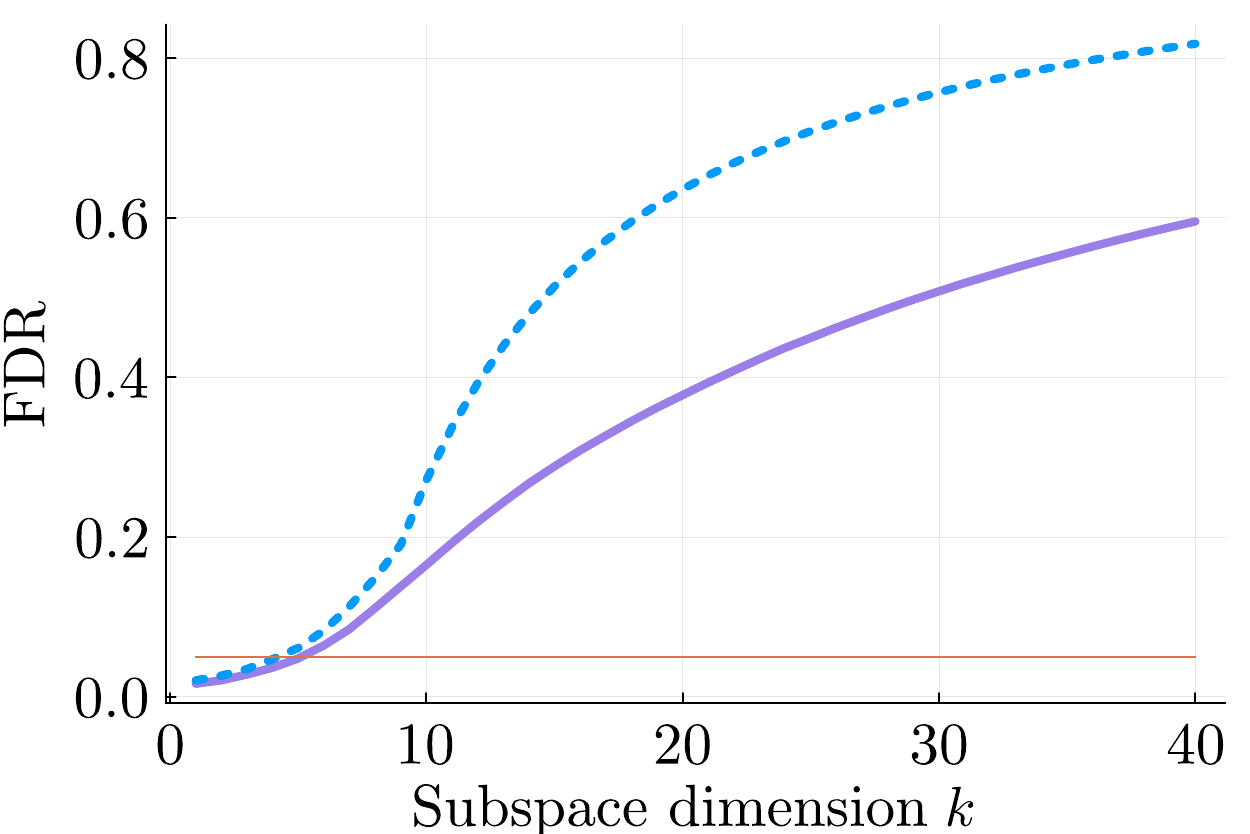} &
      \includegraphics[width=\imagewidth]{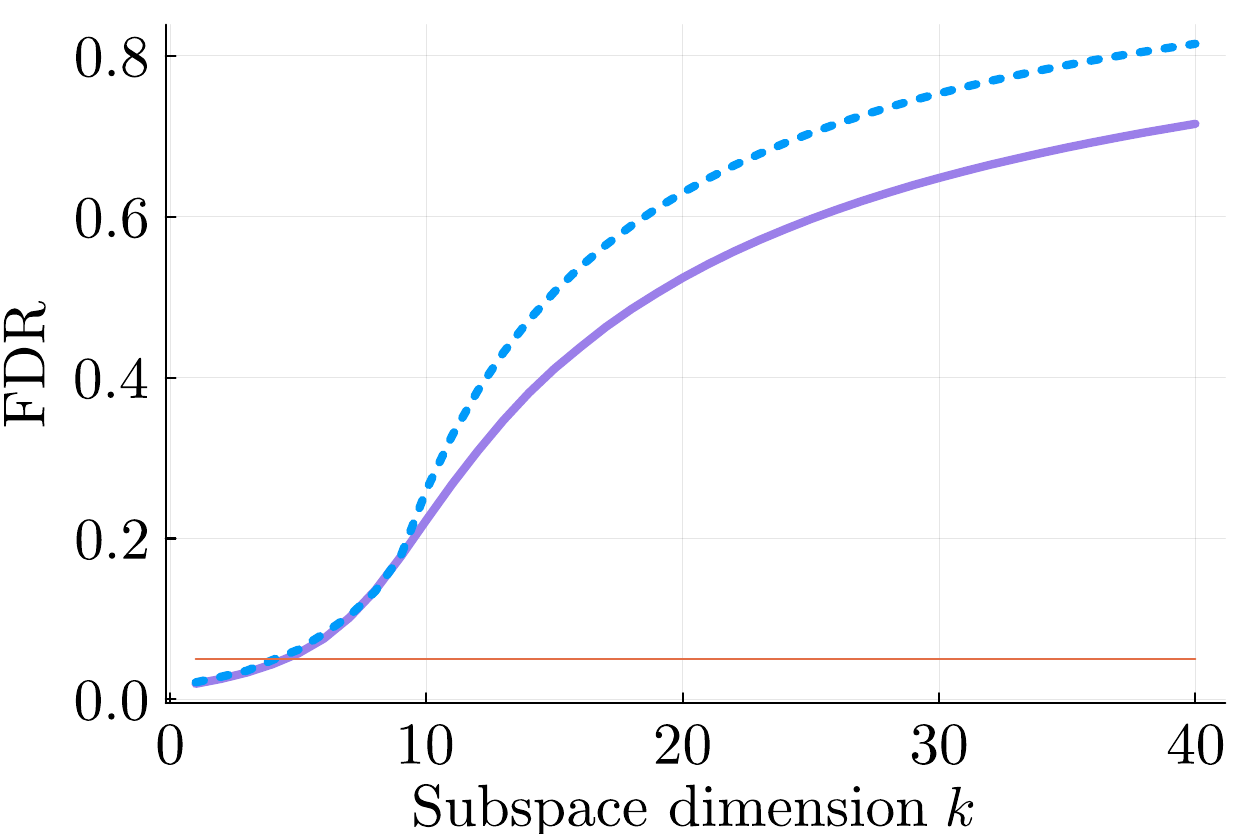} &
      \includegraphics[width=\imagewidth]{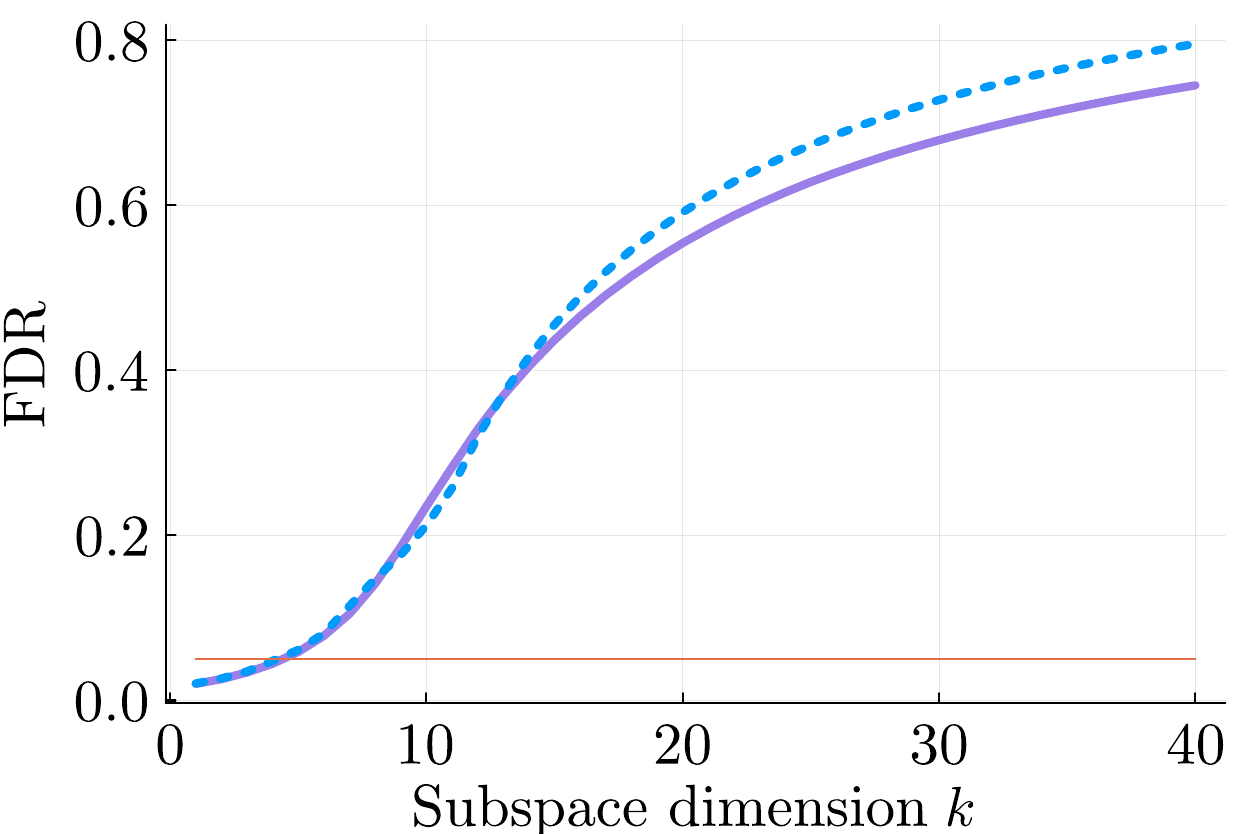} \\[2\tabcolsep]
        \setcounter{imagecolumn}{0} &%
        Dimension $n = 100$ &
        Dimension $n = 500$ &
        Dimenson $n = 1000$
    \end{tabular}
    \caption{Results for estimating the FDR of the \texttt{Uniform} ensemble using Algorithms~\ref{alg:FDR} and~\ref{alg:rank-estimate}.}
  \end{figure}
 \begin{figure}[t!]
    \def\arraystretch{0}%
    \setlength{\imagewidth}{\dimexpr \textwidth - 4\tabcolsep}%
    \divide \imagewidth by 3
    \hspace*{\dimexpr -\baselineskip - 2\tabcolsep}%
    \begin{tabular}{@{}cIII@{}}
      \stepcounter{imagerow}\raisebox{0.35\imagewidth}{\rotatebox[origin=c]{90}%
        {\strut \texttt{well-separated}}} &
      \includegraphics[width=\imagewidth]{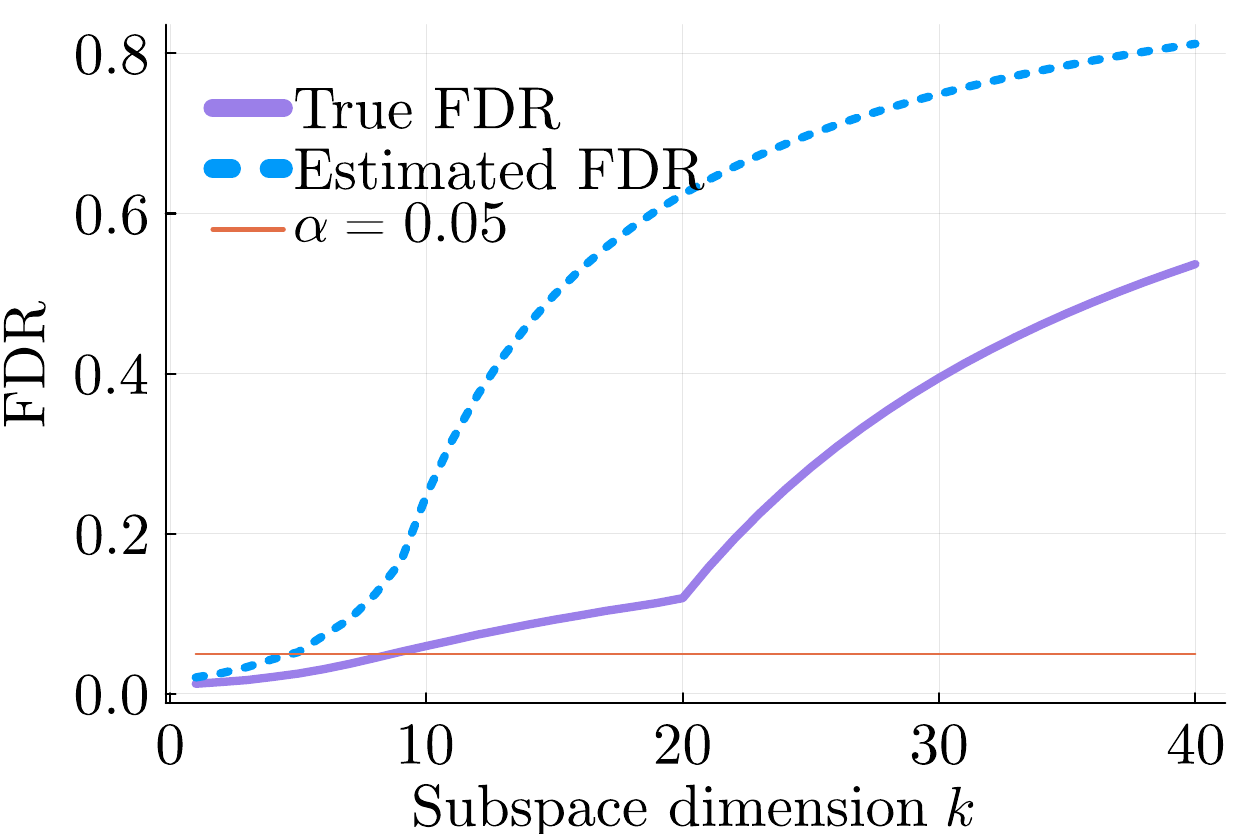} &
      \includegraphics[width=\imagewidth]{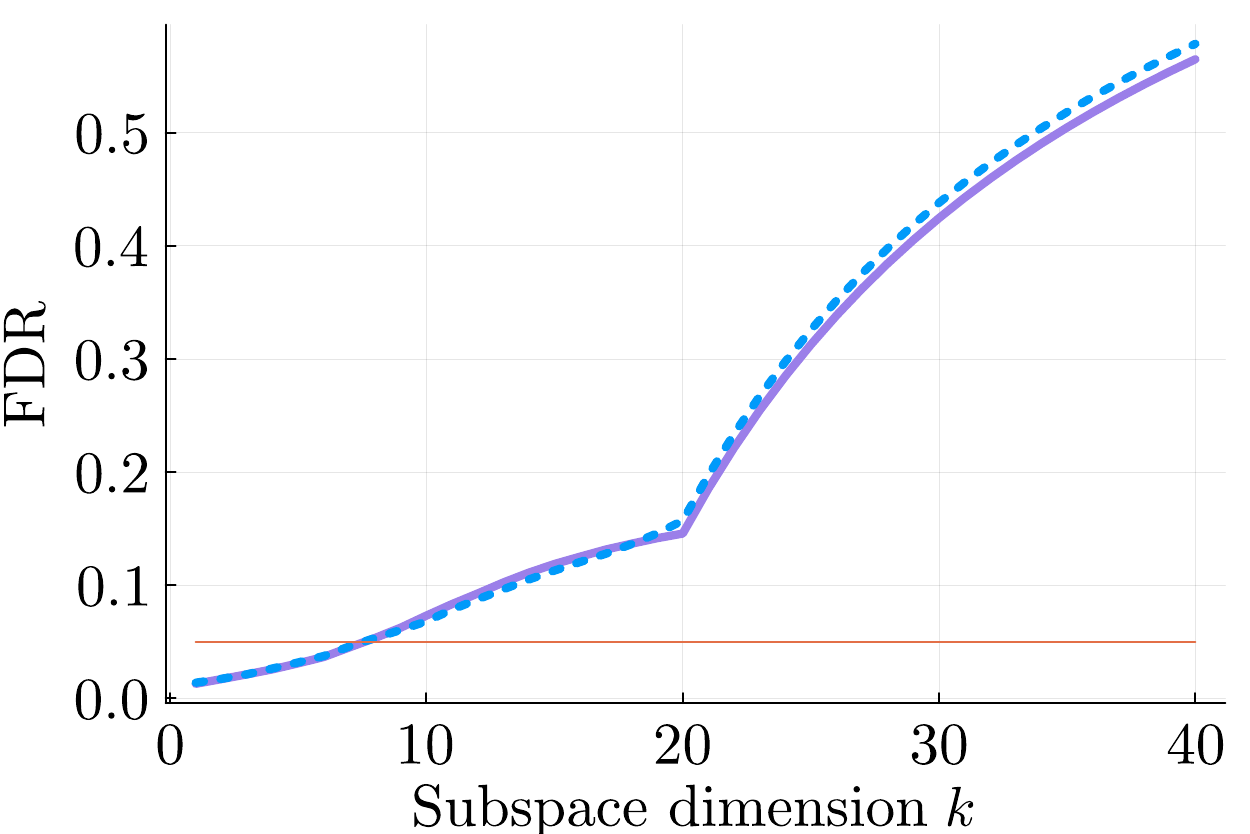}\label{example} &
      \includegraphics[width=\imagewidth]{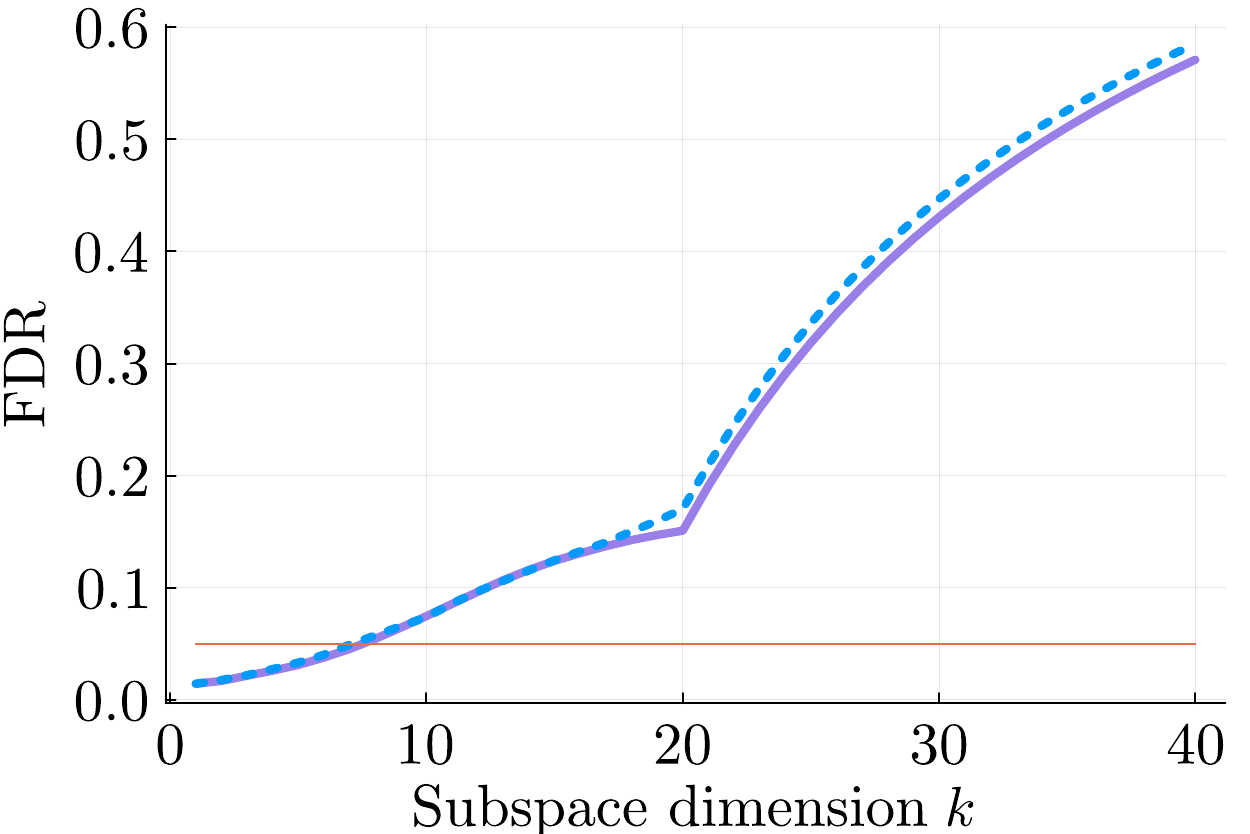} \\[2\tabcolsep]%
      \stepcounter{imagerow}\raisebox{0.35\imagewidth}{\rotatebox[origin=c]{90}%
        {\strut \texttt{barely-separated}}} &
      \includegraphics[width=\imagewidth]{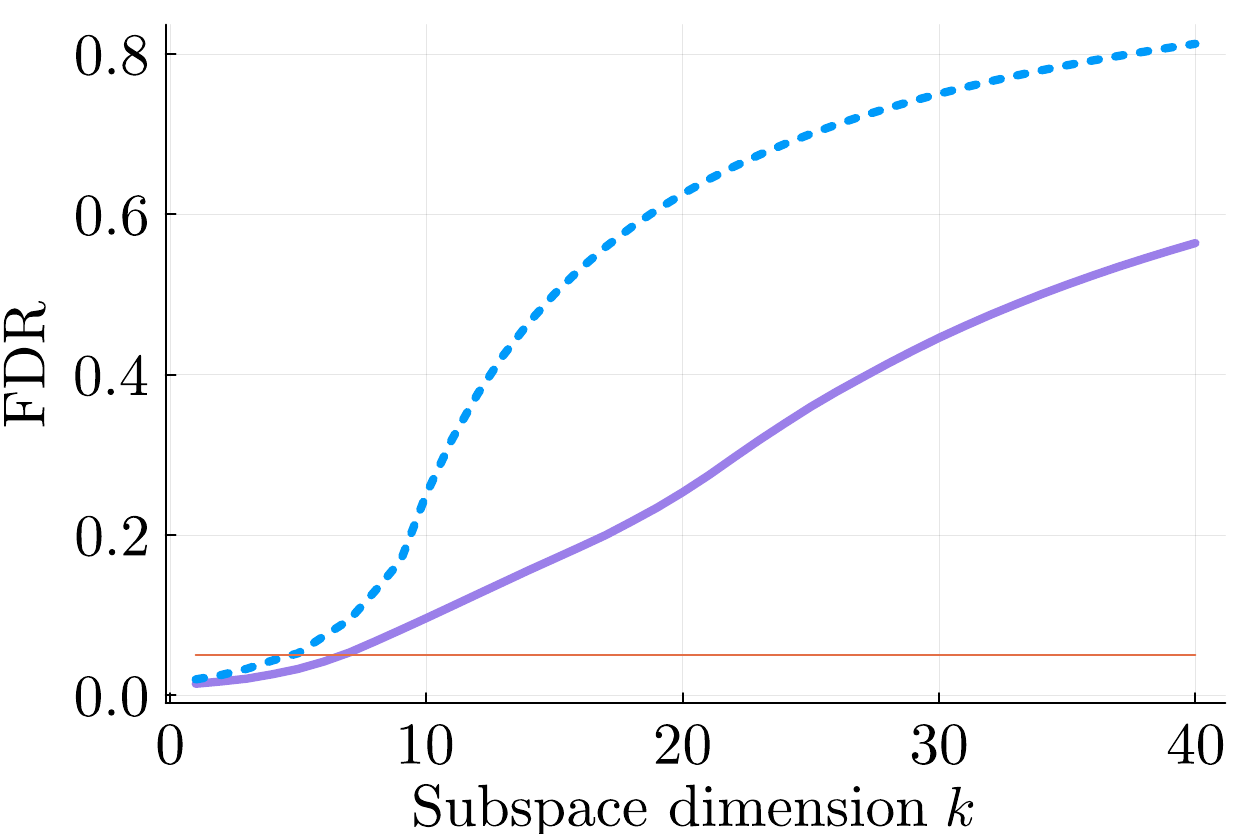} &
      \includegraphics[width=\imagewidth]{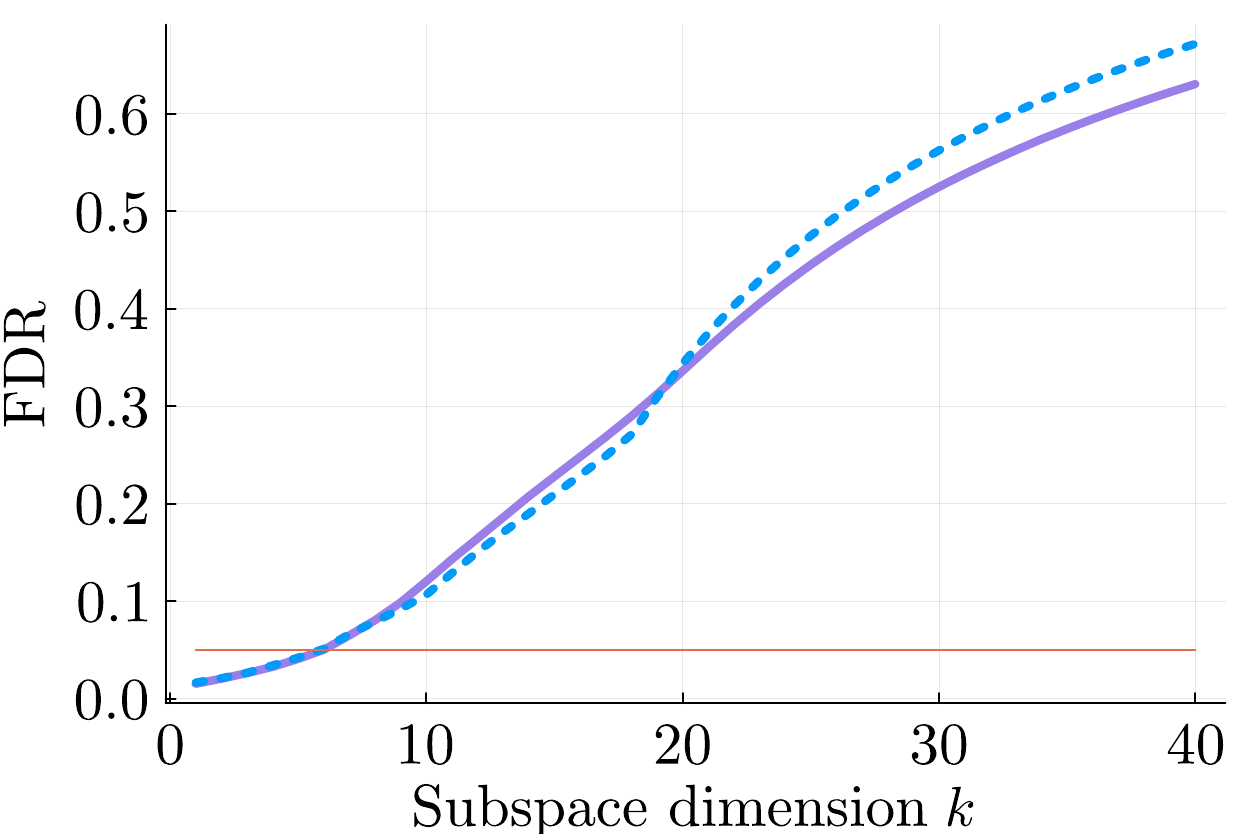} &
      \includegraphics[width=\imagewidth]{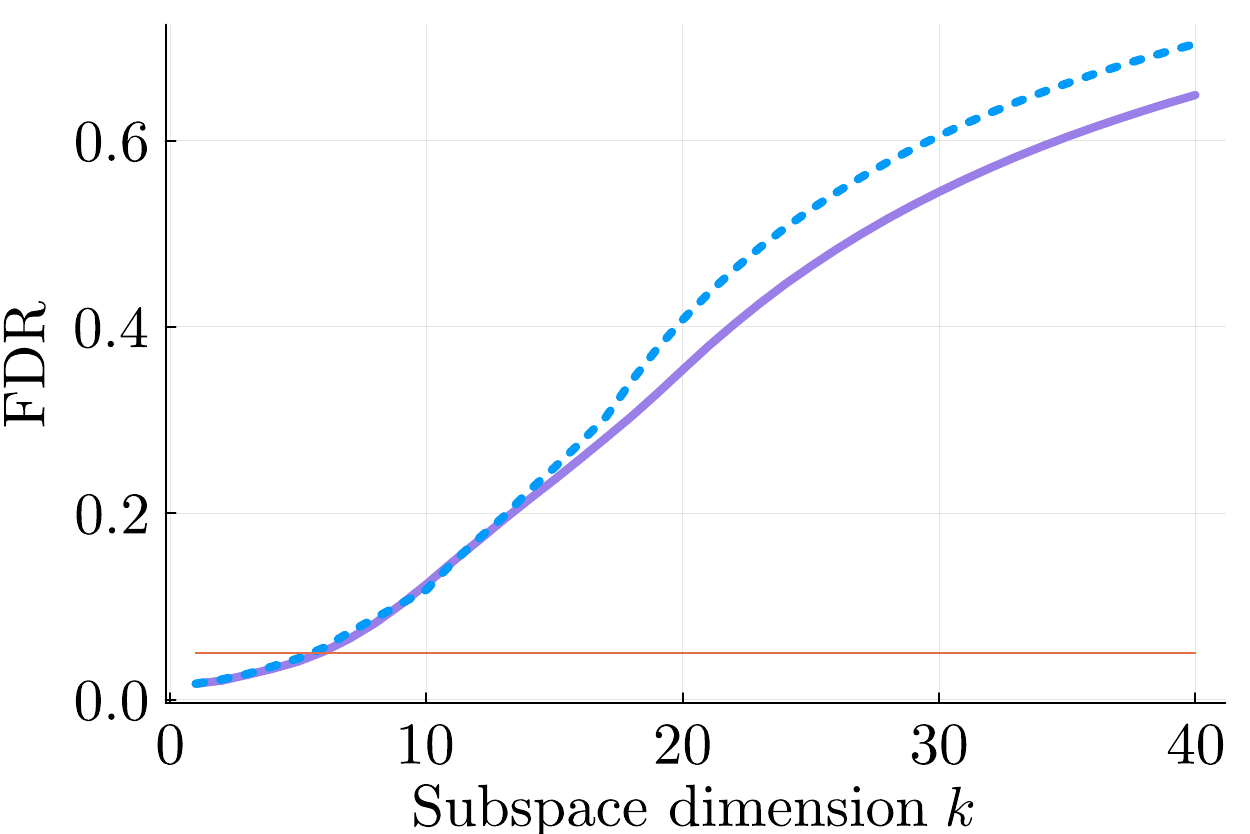} \\[2\tabcolsep]
      \stepcounter{imagerow}\raisebox{0.35\imagewidth}{\rotatebox[origin=c]{90}%
     {\strut \texttt{entangled}}} &
      \includegraphics[width=\imagewidth]{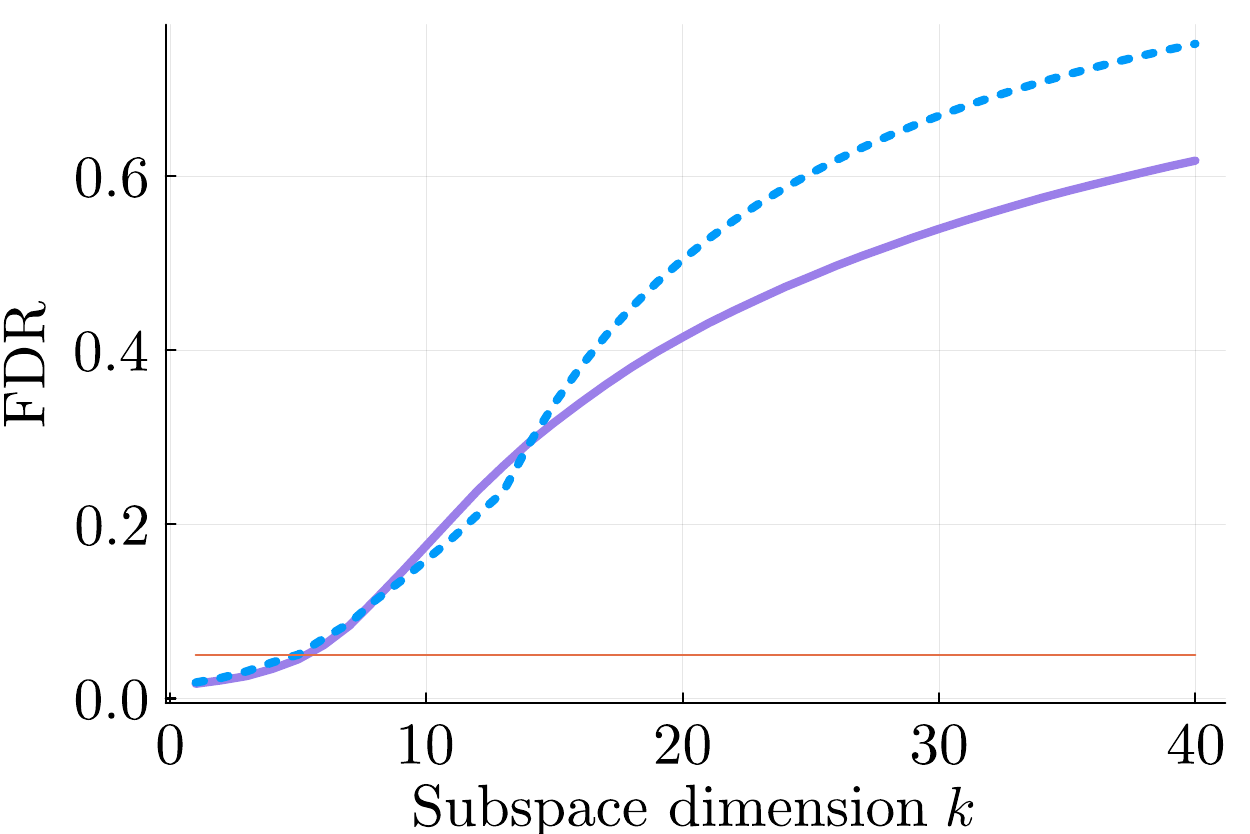} &
      \includegraphics[width=\imagewidth]{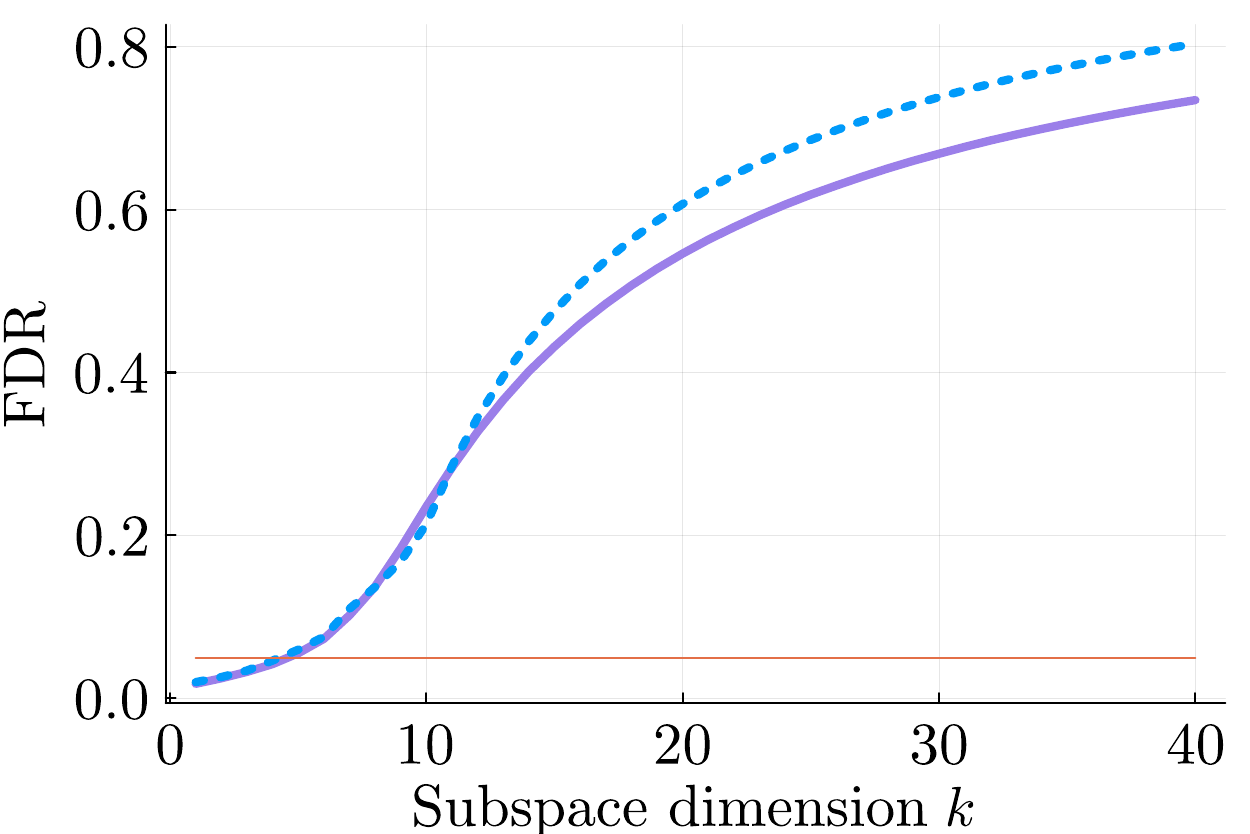} &
      \includegraphics[width=\imagewidth]{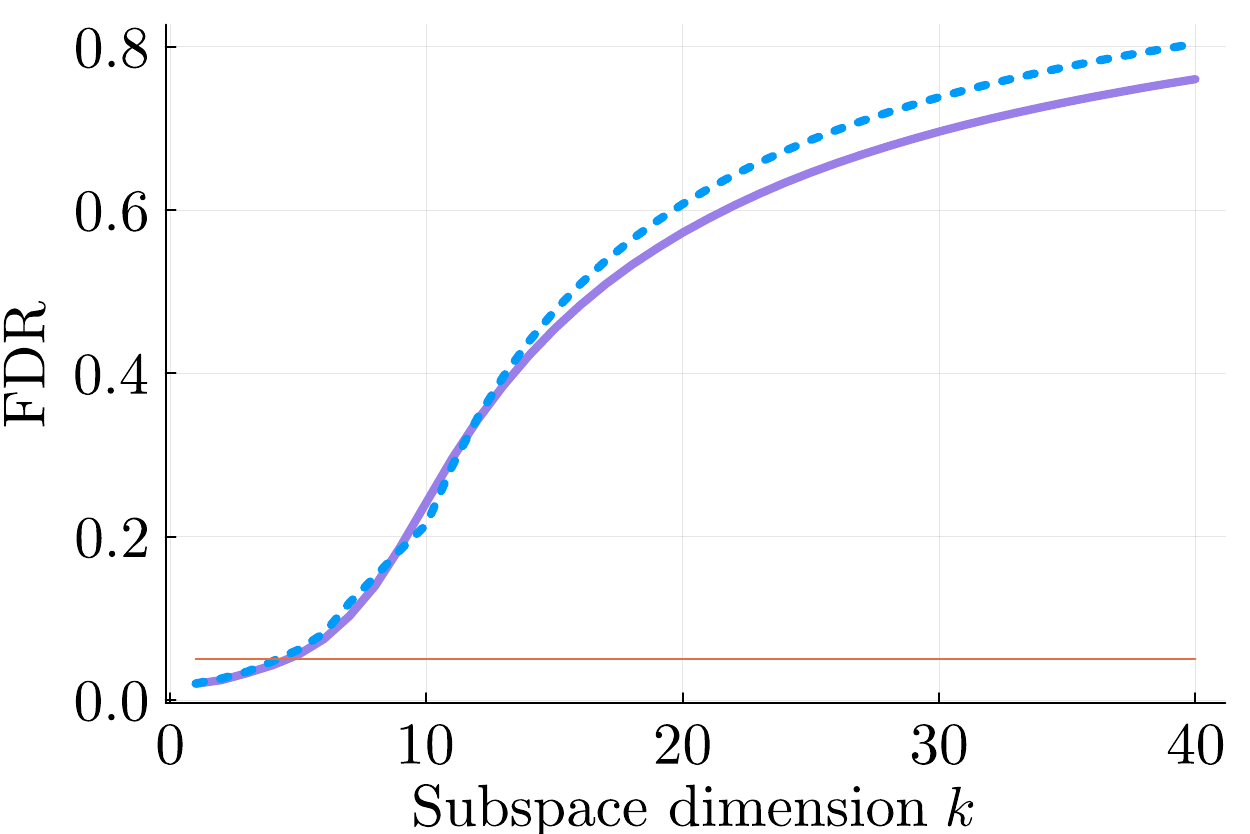} \\[2\tabcolsep]
        \setcounter{imagecolumn}{0} &%
        Dimension $n = 100$ &
        Dimension $n = 500$ &
        Dimenson $n = 1000$
    \end{tabular}
    \caption{Results for estimating the FDR of the \texttt{UniformFactor} ensemble using Algorithms~\ref{alg:FDR-as} and~\ref{alg:rank-estimate-as}.}
\end{figure}

\begin{figure}[t]
    \def\arraystretch{0}%
    \setlength{\imagewidth}{\dimexpr \textwidth - 4\tabcolsep}%
    \divide \imagewidth by 3
    \hspace*{\dimexpr -\baselineskip - 2\tabcolsep}%
    \begin{tabular}{@{}cIII@{}}
      \stepcounter{imagerow}\raisebox{0.35\imagewidth}{\rotatebox[origin=c]{90}%
        {\strut \texttt{well-separated}}} &
      \includegraphics[width=\imagewidth]{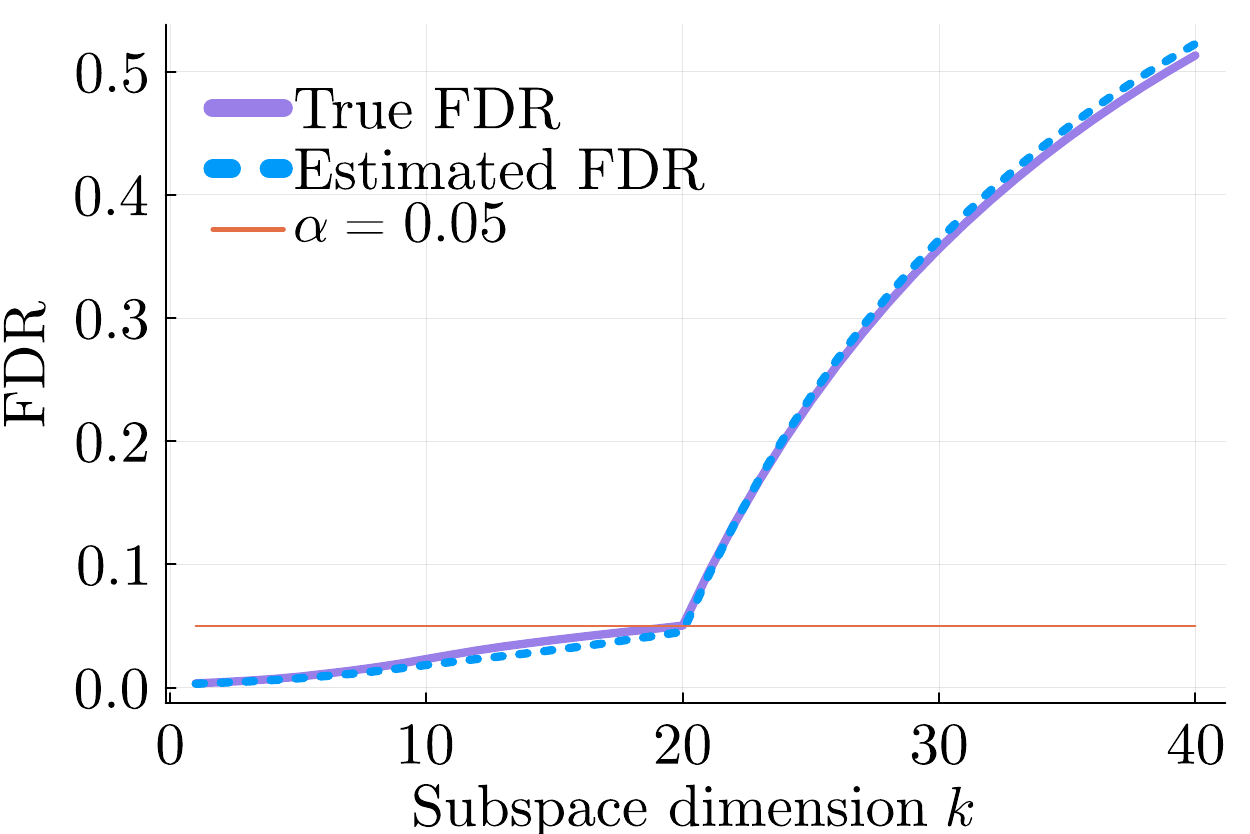} &
      \includegraphics[width=\imagewidth]{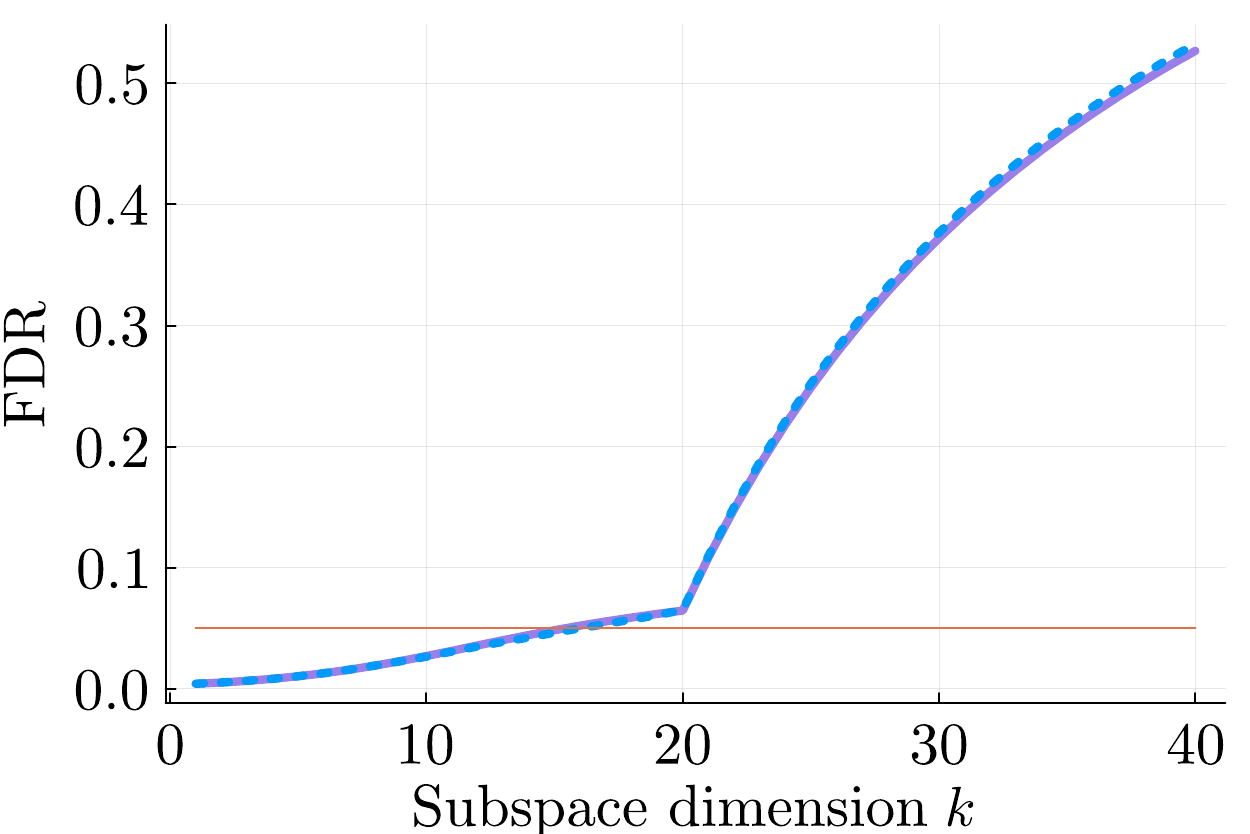}\label{example} &
      \includegraphics[width=\imagewidth]{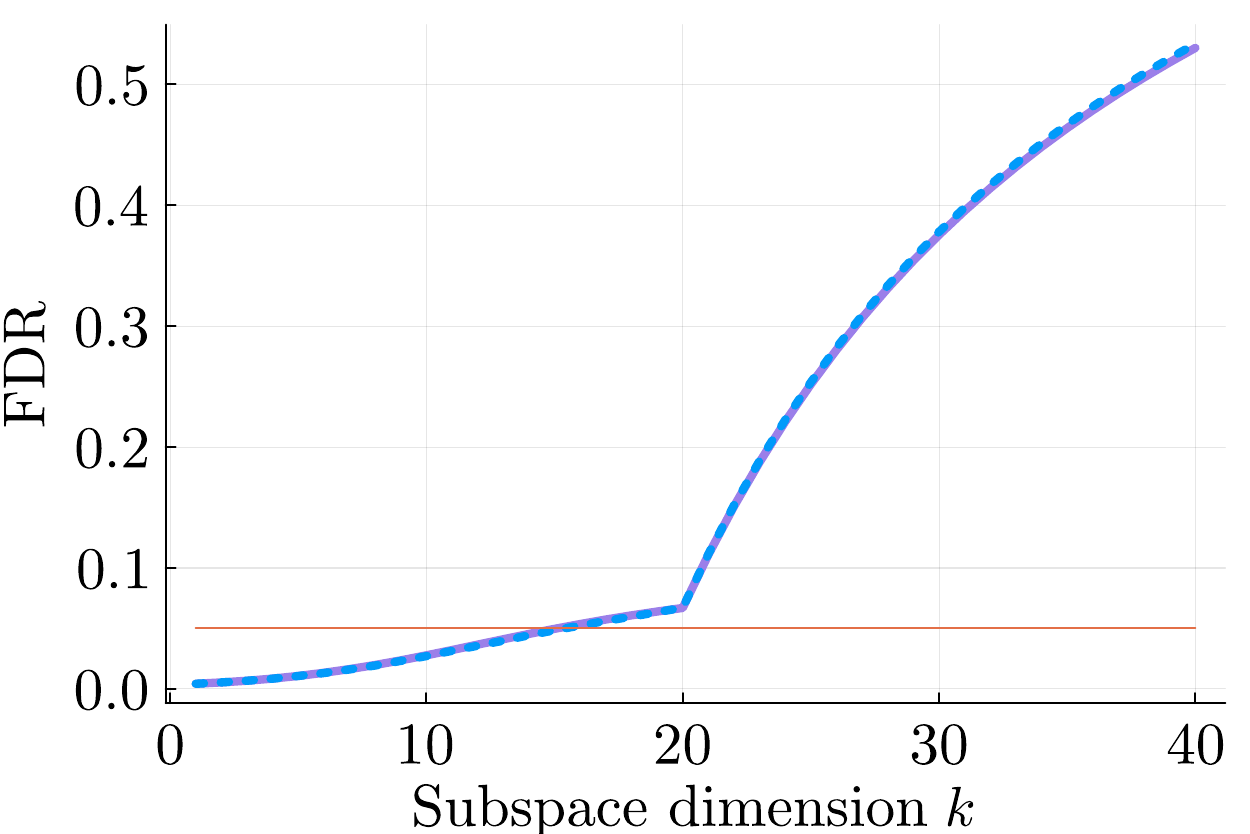} \\[2\tabcolsep]%
      \stepcounter{imagerow}\raisebox{0.35\imagewidth}{\rotatebox[origin=c]{90}%
        {\strut \texttt{barely-separated}}} &
      \includegraphics[width=\imagewidth]{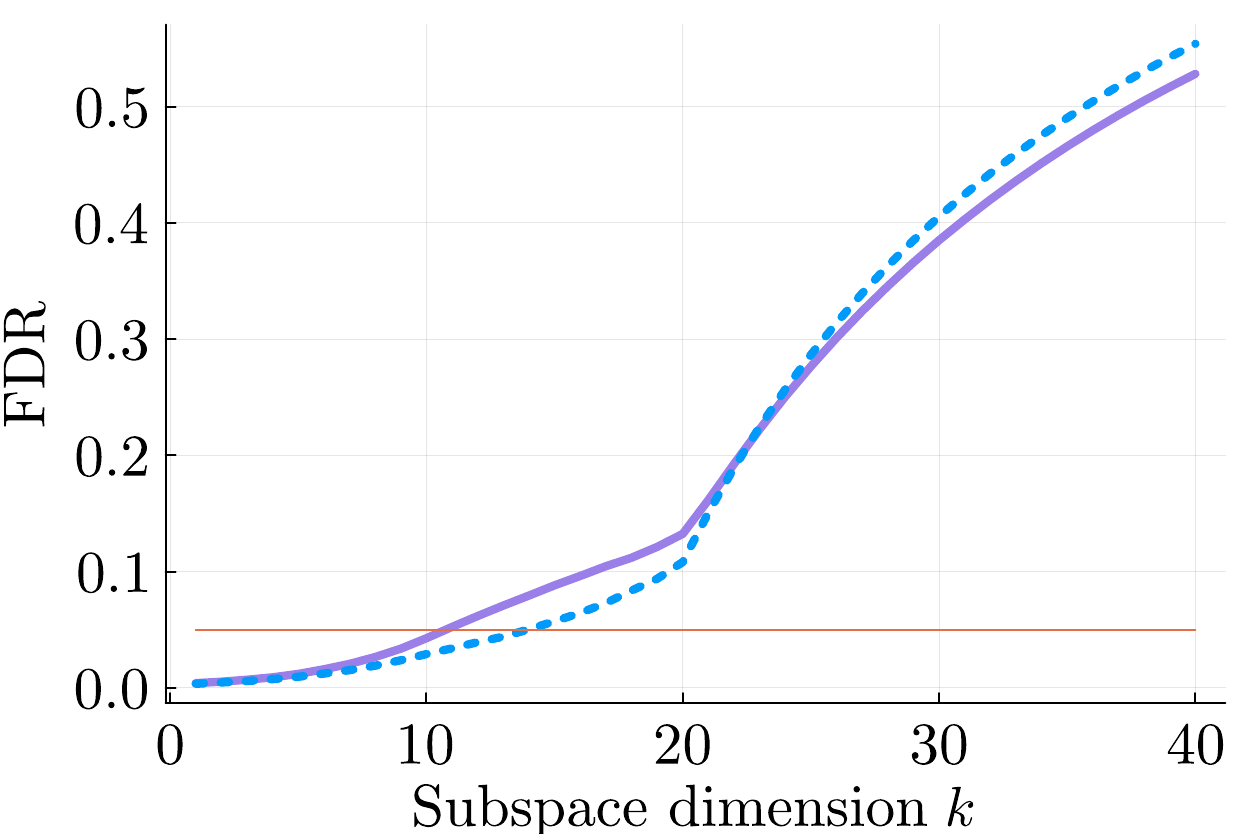} &
      \includegraphics[width=\imagewidth]{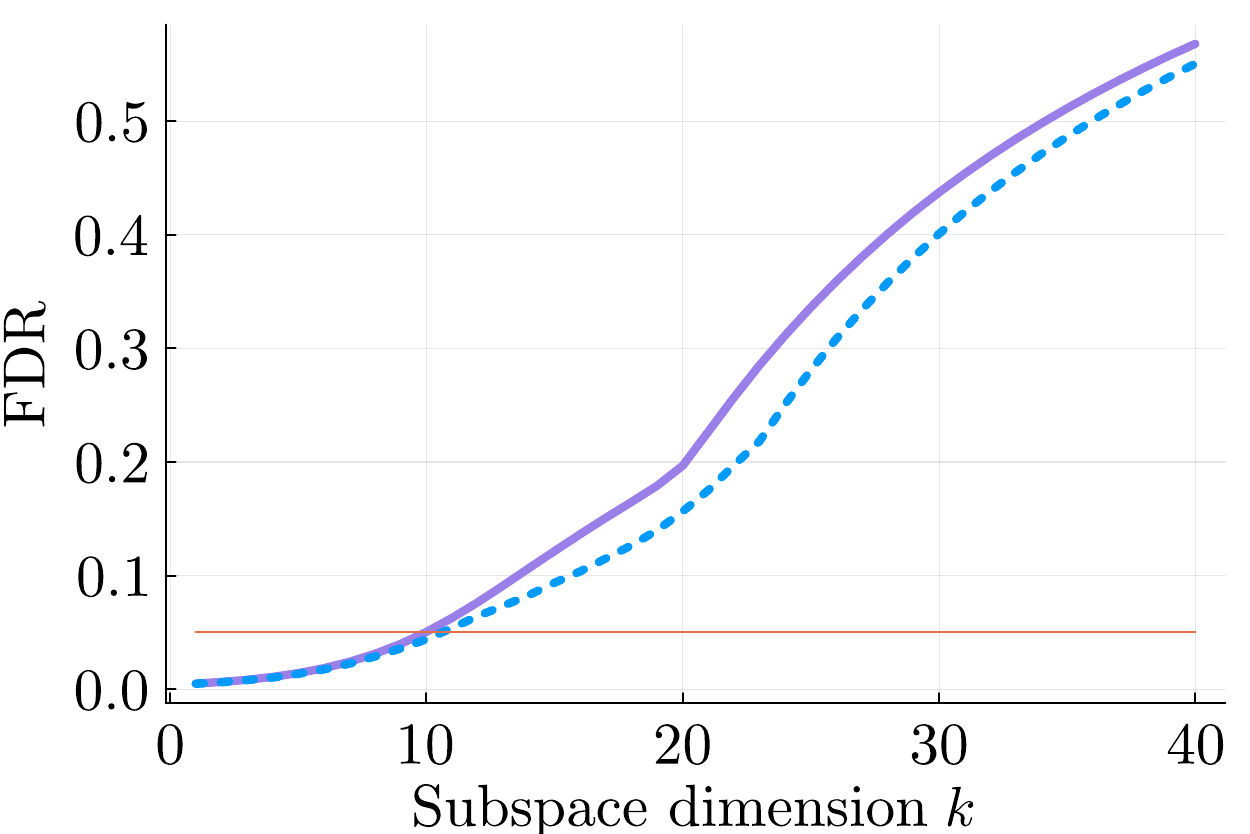} &
      \includegraphics[width=\imagewidth]{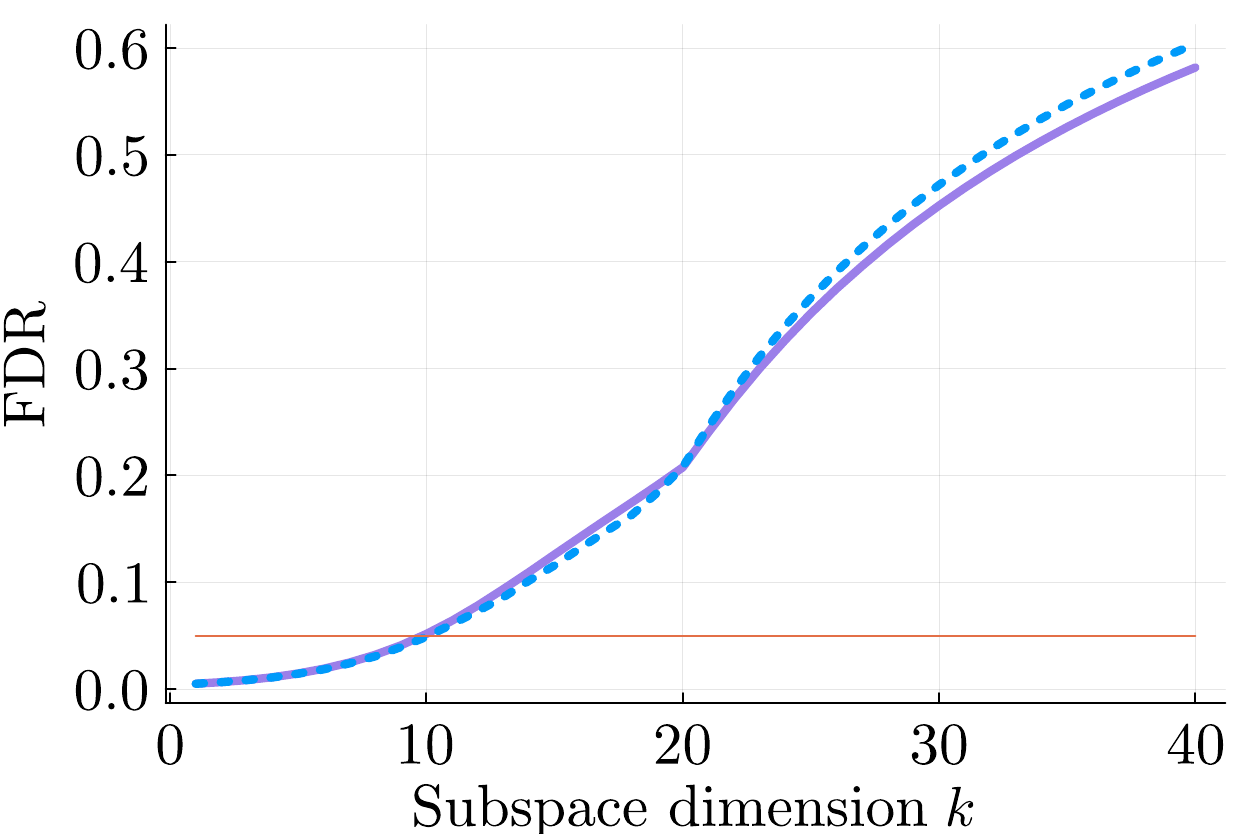} \\[2\tabcolsep]
      \stepcounter{imagerow}\raisebox{0.35\imagewidth}{\rotatebox[origin=c]{90}%
     {\strut \texttt{entangled}}} &
      \includegraphics[width=\imagewidth]{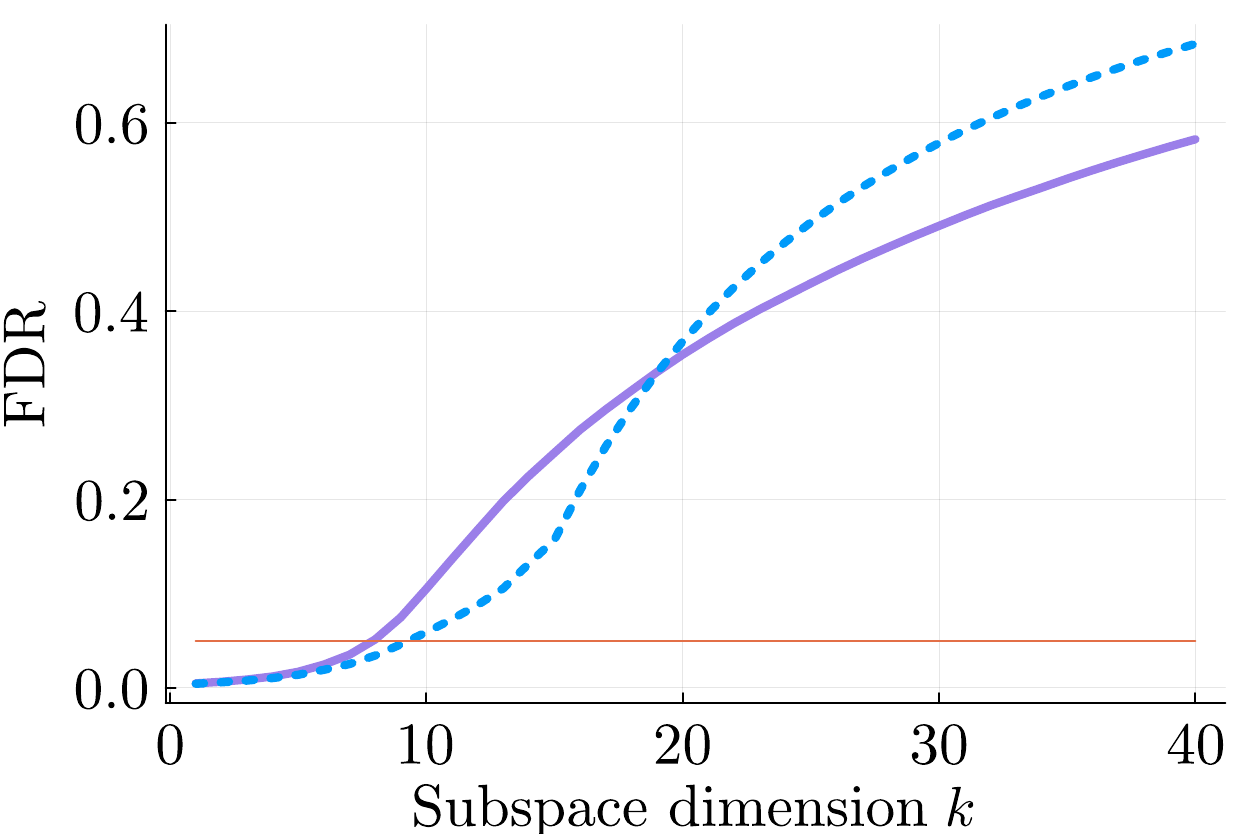} &
      \includegraphics[width=\imagewidth]{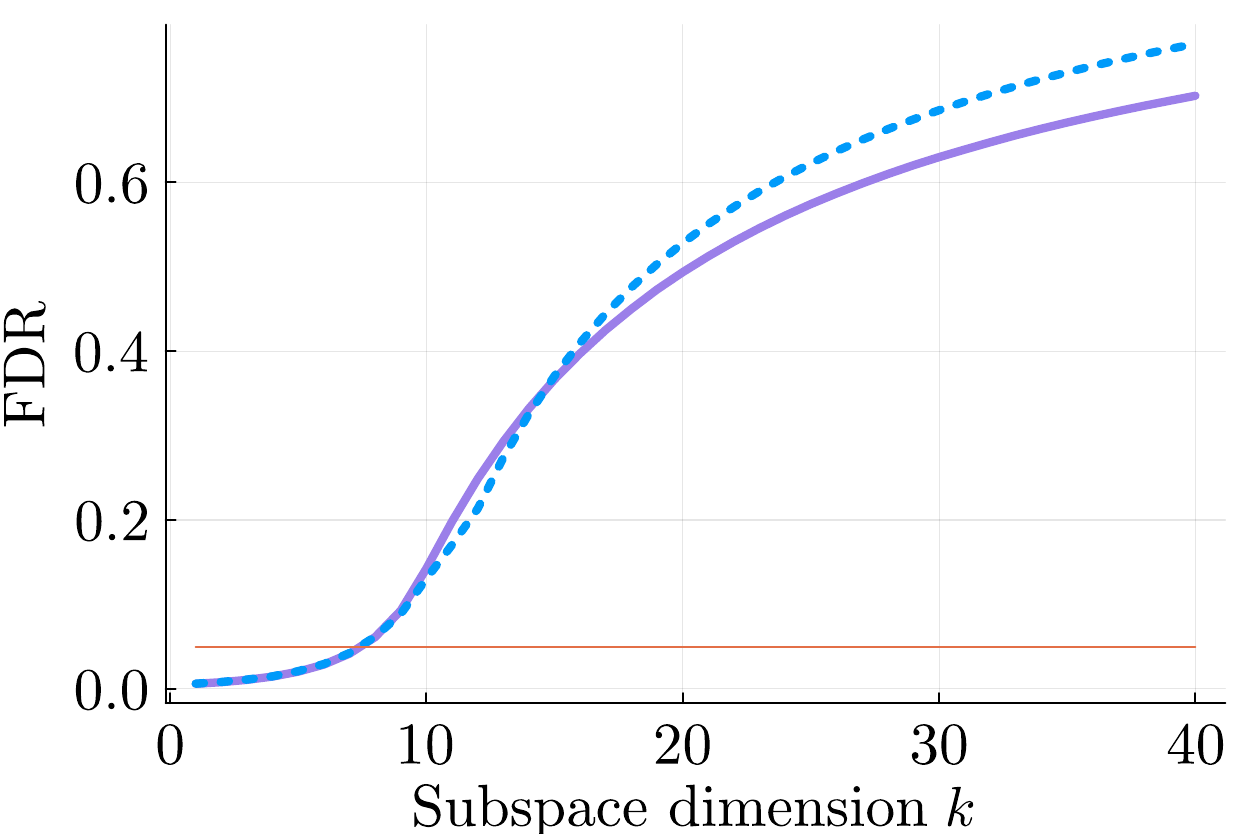} &
      \includegraphics[width=\imagewidth]{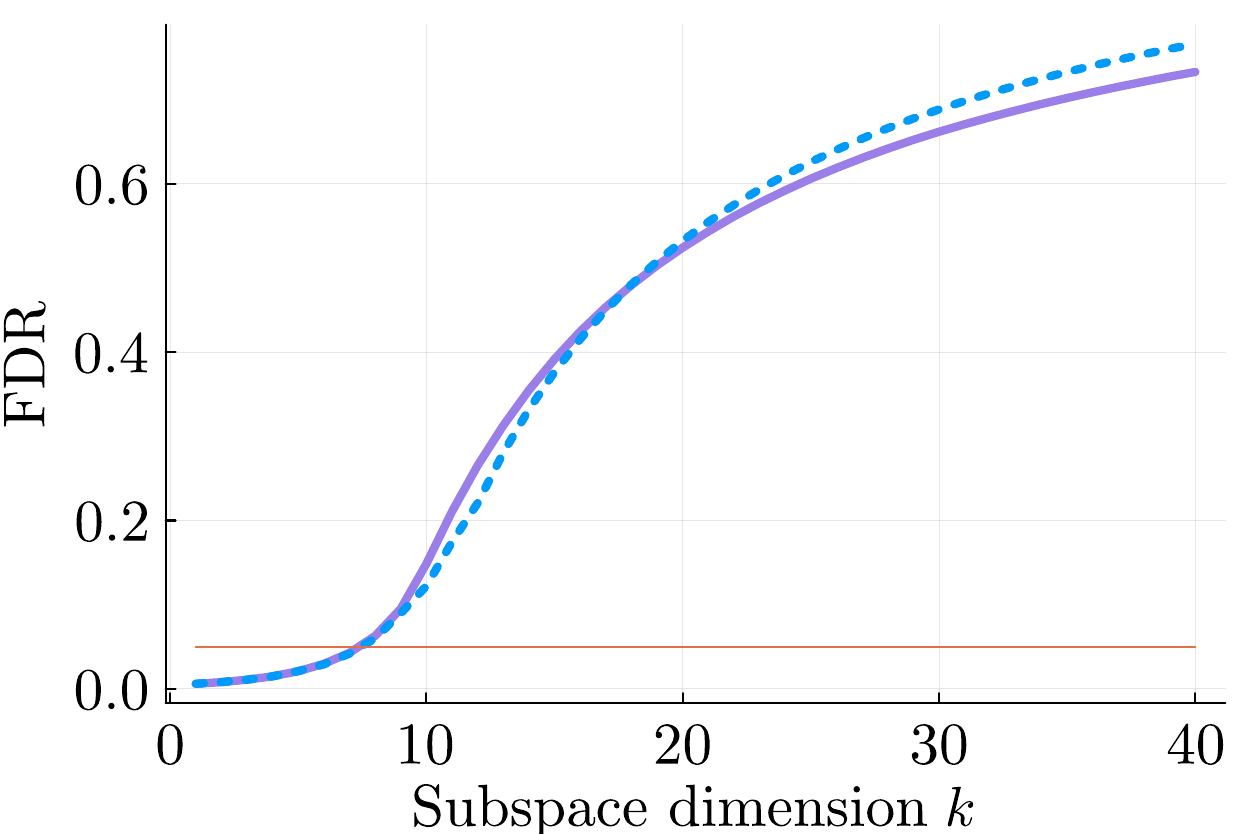} \\[2\tabcolsep]
        \setcounter{imagecolumn}{0} &%
        Dimension $n = 100$ &
        Dimension $n = 500$ &
        Dimenson $n = 1000$
    \end{tabular}
    \caption{Results for estimating the FDR of the \texttt{Fisher} ensemble using Algorithms~\ref{alg:FDR} and~\ref{alg:rank-estimate}.}
  \end{figure}
 \begin{figure}[t!]
    \def\arraystretch{0}%
    \setlength{\imagewidth}{\dimexpr \textwidth - 4\tabcolsep}%
    \divide \imagewidth by 3
    \hspace*{\dimexpr -\baselineskip - 2\tabcolsep}%
    \begin{tabular}{@{}cIII@{}}
      \stepcounter{imagerow}\raisebox{0.35\imagewidth}{\rotatebox[origin=c]{90}%
        {\strut \texttt{well-separated}}} &
      \includegraphics[width=\imagewidth]{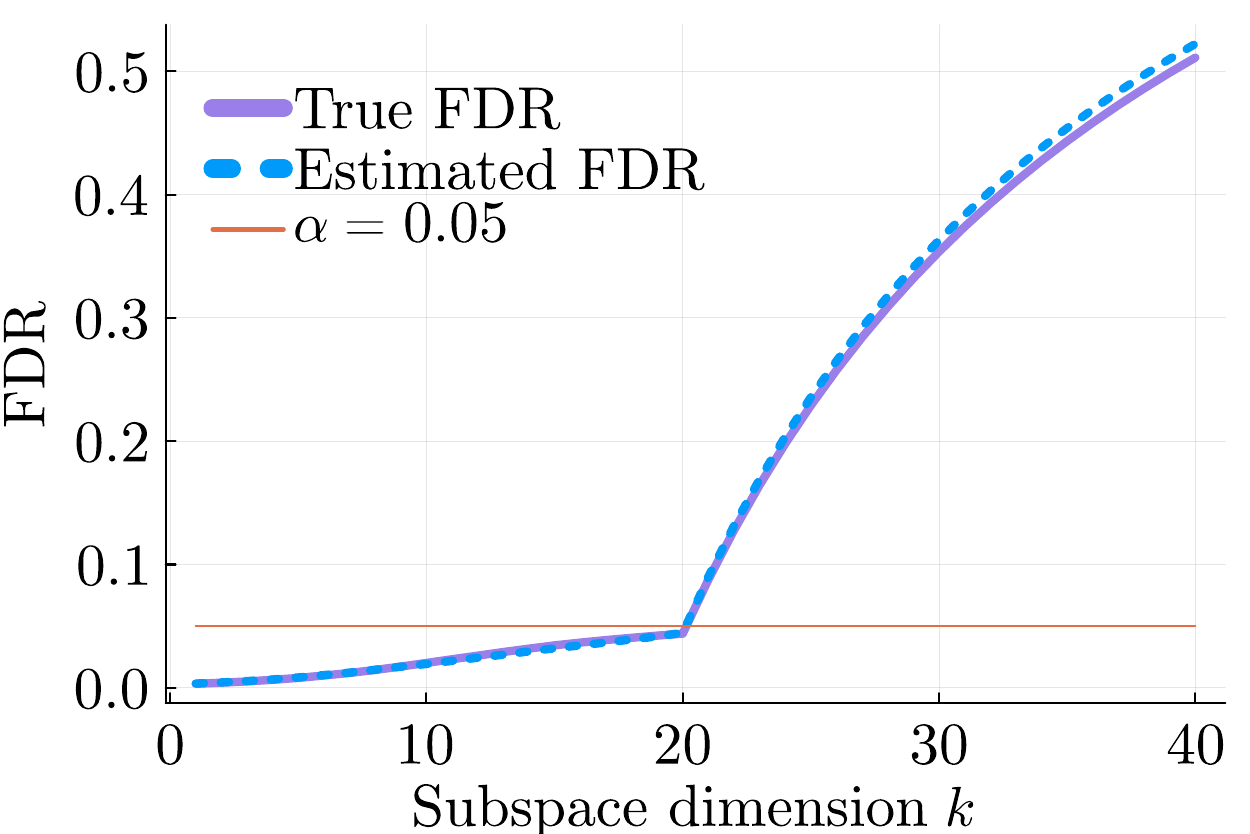} &
      \includegraphics[width=\imagewidth]{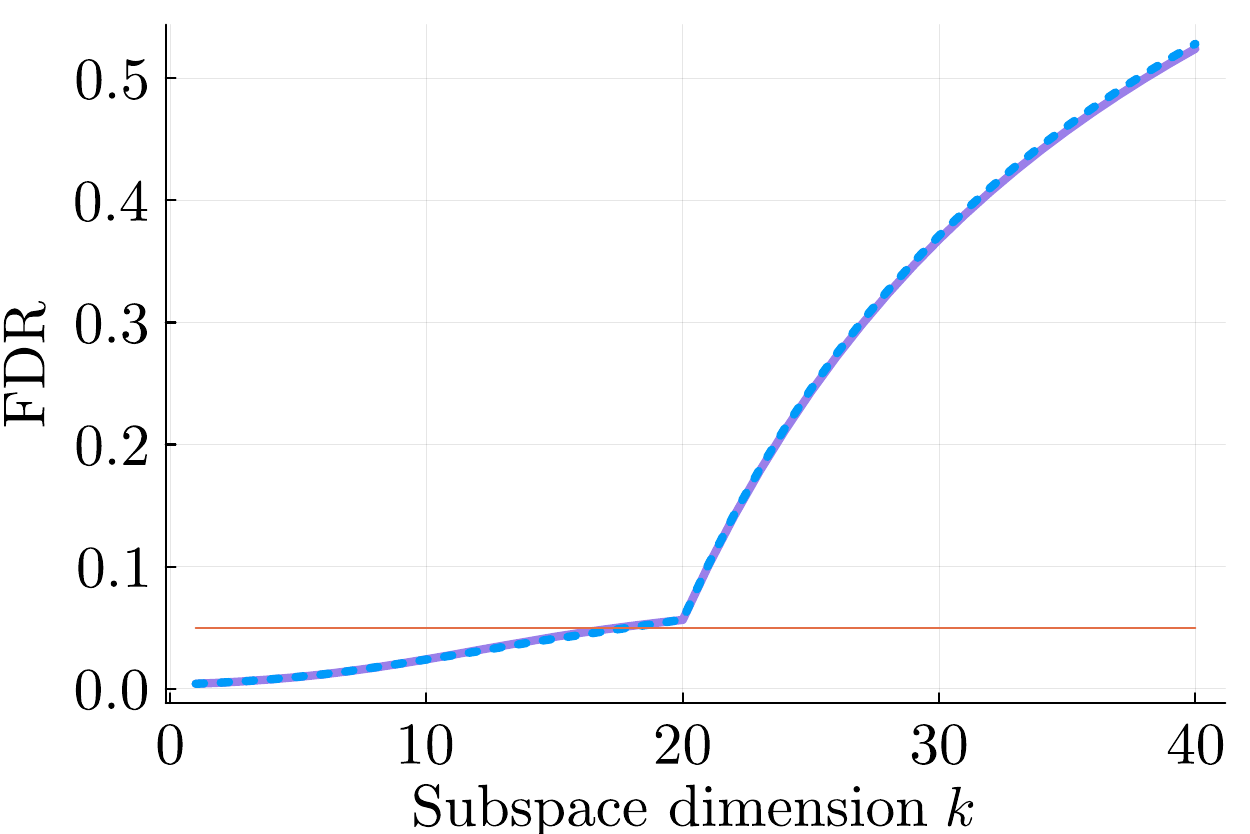}\label{example} &
      \includegraphics[width=\imagewidth]{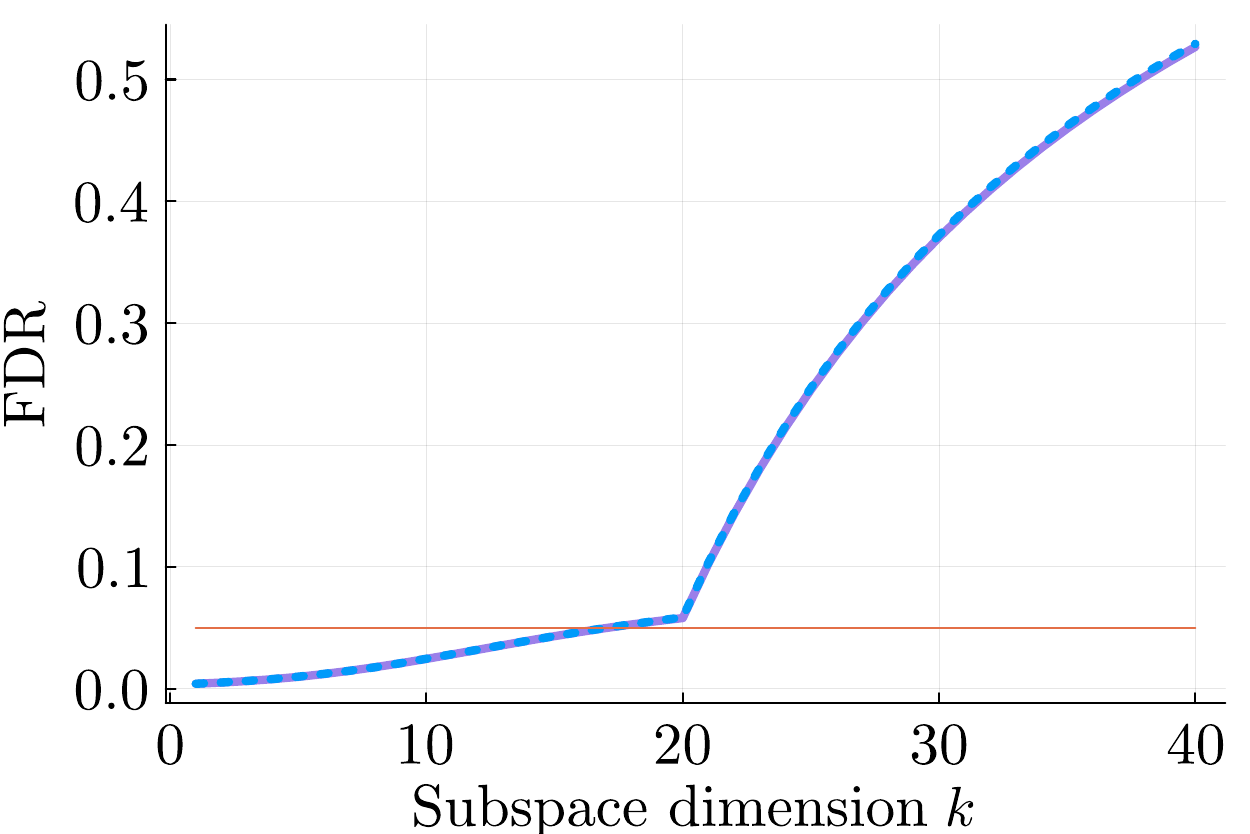} \\[2\tabcolsep]%
      \stepcounter{imagerow}\raisebox{0.35\imagewidth}{\rotatebox[origin=c]{90}%
        {\strut \texttt{barely-separated}}} &
      \includegraphics[width=\imagewidth]{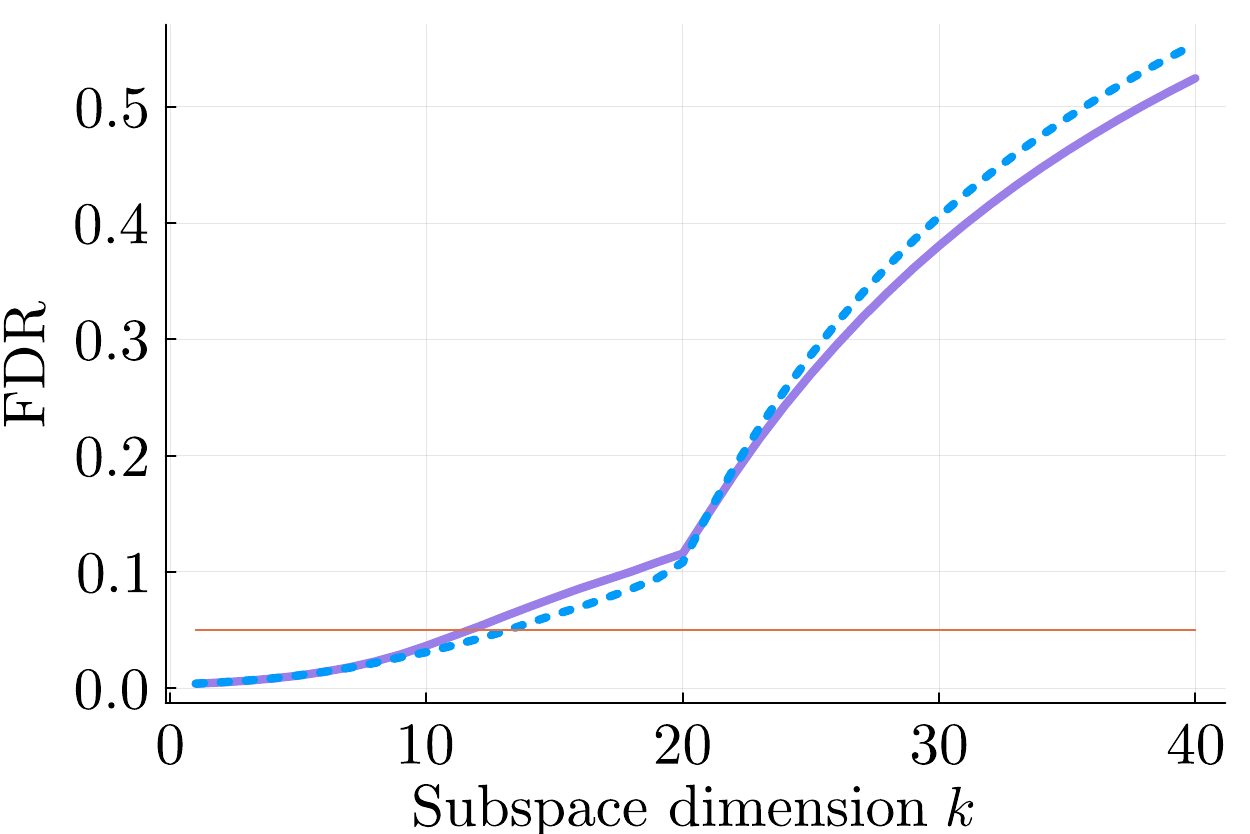} &
      \includegraphics[width=\imagewidth]{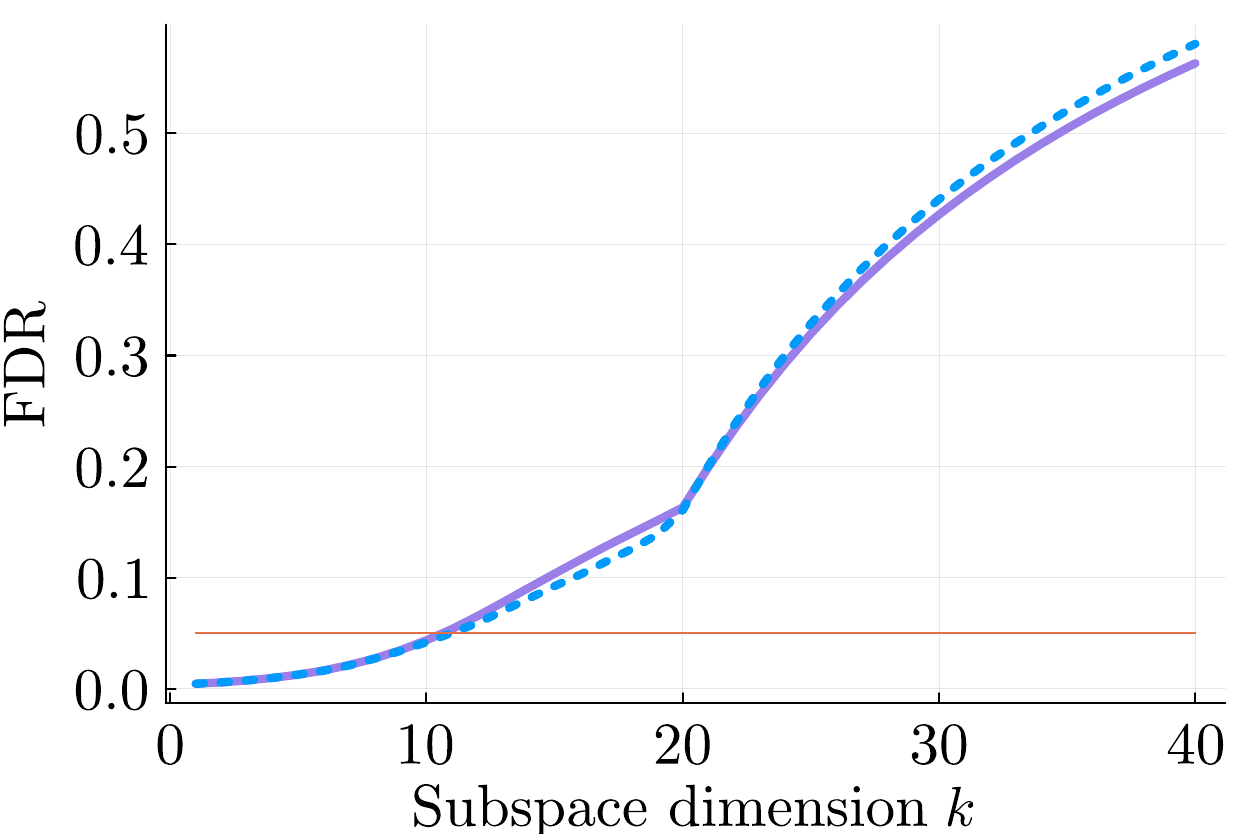} &
      \includegraphics[width=\imagewidth]{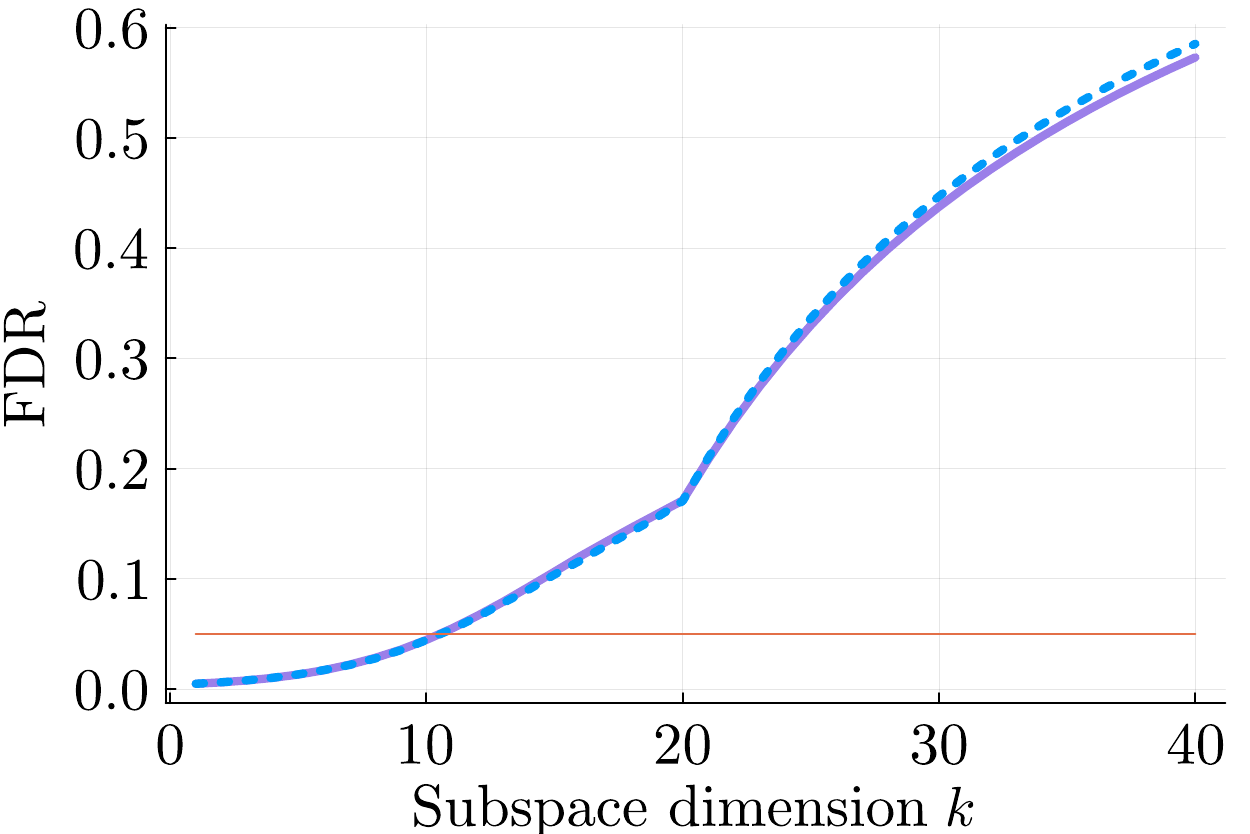} \\[2\tabcolsep]
      \stepcounter{imagerow}\raisebox{0.35\imagewidth}{\rotatebox[origin=c]{90}%
     {\strut \texttt{entangled}}} &
      \includegraphics[width=\imagewidth]{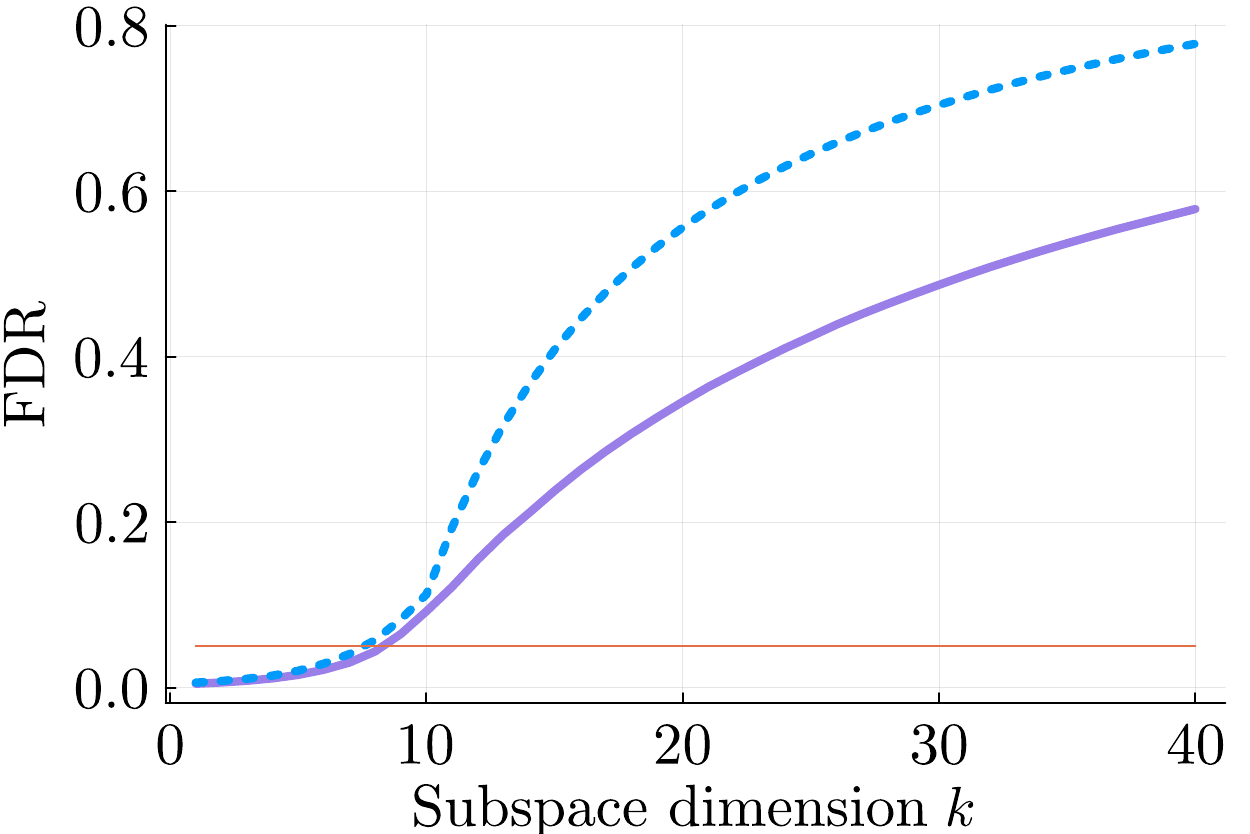} &
      \includegraphics[width=\imagewidth]{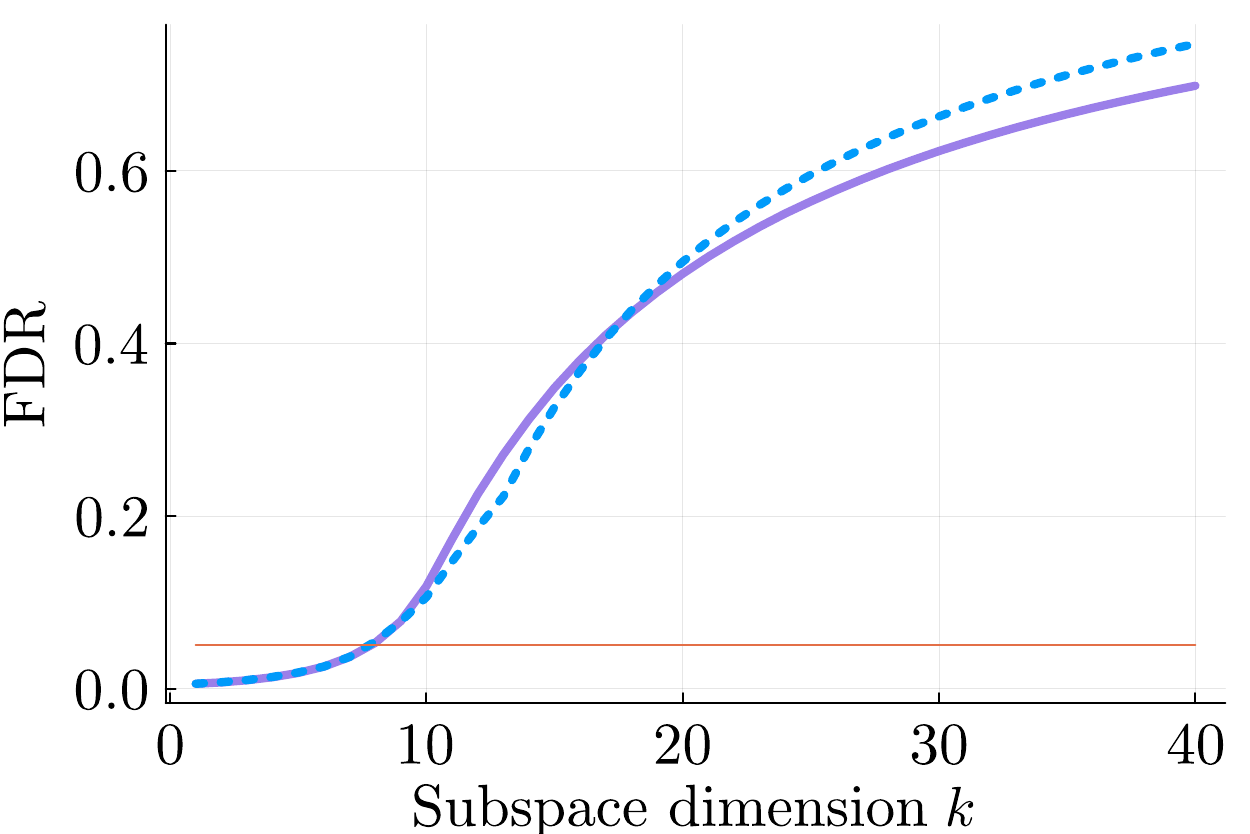} &
      \includegraphics[width=\imagewidth]{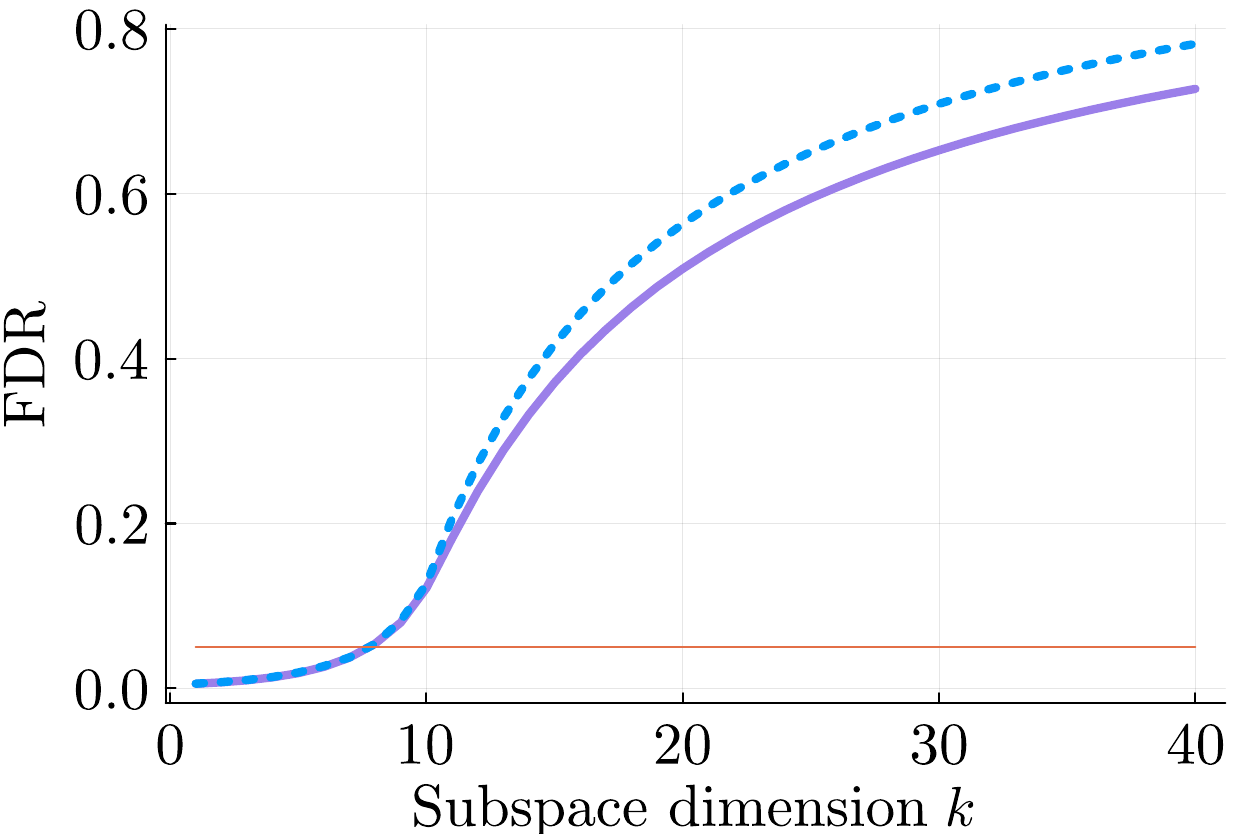} \\[2\tabcolsep]
        \setcounter{imagecolumn}{0} &%
        Dimension $n = 100$ &
        Dimension $n = 500$ &
        Dimenson $n = 1000$
    \end{tabular}
    \caption{Results for estimating the FDR of the \texttt{FisherFactor} ensemble using Algorithms~\ref{alg:FDR-as} and~\ref{alg:rank-estimate-as}.}
\end{figure}

\section{Supplementary Results} \label{sec:background}
  This section reviews the necessary background to establish theoretical guarantees for our methods. We highlight that the results in this section are known.

\subsection{Symmetric Case}
Let us start with an asymptotic characterization of the eigenvalues and eigenvectors of $\bX$.
    \begin{theorem}[Theorems 2.1 and 2.2 in \cite{benaych2011eigenvalues}]\label{thm:raj}
      Suppose that Assumption~\ref{ass:model1} and \ref{ass:model2} hold. For each $n$, let $\lambda_{i}^{(n)} = \lambda_{i}(\bX_{n})$ and let $\tilde u$ be a unit-norm eigenvector associated with $\lambda_{i}^{(n)}$. Then, we have that as $n$ grows to infinity,
      \begin{itemize}
              \item[] (\textbf{Eigenvalue limit}) Fix $j \in \{1, \dots, r\}$, then,
      \begin{equation}
        \label{eq:ass_eigenvalue}
 \lambda_{i}^{(n)} \as \begin{cases} \rho_i := G_{\mu_{\bE}}^{-1}(1/\theta_{i}) & \text{if }\theta_{j} > 1/G_{\mu_{\bE}}(b^{+}),  \\
 b  & \text{otherwise.}
 \end{cases}
      \end{equation}
              \item[] (\textbf{Eigenvector correlation limit})  Fix $j \in \{1, \dots, r\}$ such that $\frac{1}{\theta_{j}} \in \big(G_{\mu_{\bE}}(a^{-}), G_{\mu_{\bE}}(b^{+})\big)$, let $\cW_{j} = \ker(\theta_{j}\bI_{n} - \bA_{n})$ be the subspace of eigenvectors associated with $\theta_{i}$, and let $\widehat u$ be a unit norm eigenvector associated with $\lambda_j$. Then, 
      \begin{equation}
        \label{eq:ass_eigenvector}
        \|\cP_{\cW_{j}}(\widehat u)\|^{2} \xrightarrow{a.s.} -\frac{1}{\theta_j^2G'_{\mu_{\bE}}(\rho_j)}  = - \frac{G_{\mu_{\bE}}(\rho_j)^{2}}{G'_{\mu_{\bE}}(\rho_j)}.
      \end{equation}
      Furthermore, let $\cT_j = \bigoplus_{i \neq j}^r \ker(\theta_i \bI_n - \bA_n)$ be the eigenspace corresponding to the first $r$ eigenvectors that are not associated with $\theta_i$. Then,
      \begin{equation}\label{eq:decor} \|\cP_{\cT_{j}}(\widehat u)\|^{2} \as 0. \end{equation}
    \end{itemize}
    \end{theorem}

The following proposition unveils the behavior of the Cauchy transform and its derivative at the edge of the spectrum. It relies crucially on the square-root decay assumption. The previous Theorem holds without the need for the square-root decay condition. 
    \begin{proposition}[Proposition 2.4 in \cite{benaych2011eigenvalues}]
        \label{prop:nice-edge}
        Suppose that~\cref{ass:model2} holds. Then, we have that $$G_{\mu_{\bE}}(b^+) < \infty \quad \text{and} \quad G'_{\mu_{\bE}}(b^+) = \infty.$$
    \end{proposition}

    Finally, the following proposition shows that eigenvectors of $\bX$ and $\bA$ that are not associated with the same eigenvalue are decorrelated. 
    \begin{proposition}\label{prop:zero-correlation}
      Suppose that Assumption~\ref{ass:model1} and~\ref{ass:model2} hold. Fix $j \in \{1, \dots, r\}$ and $i \neq j.$
      Let $\bu_{j}$ be a unit-norm eigenvector of $\bA_{n}$ associated with $\theta_j$ and let $\widehat \bu_{i}$ be a unit-norm eigenvector of $\bX_{n}$ associated with $\lambda_i(\bX)$. Then,
      \begin{equation}
        \label{eq:decorrelation}
        \dotp{\widehat u_{\ell}, u_{j}} \xrightarrow{a.s.} 0 \qquad \text{almost surely.}
      \end{equation}
    \end{proposition}

 The conclusion of this guarantee is similar to \eqref{eq:decor}, but it does not assume that $i \leq r$. 

 \begin{proof}
     The proof uses an analogous result for singular vectors; see \cref{prop:zero-correlation-as} below. We will leverage the following folklore result. 
     \begin{fact} \label{fact:more-randomness}
         Let $\bU, \bV, \bW \in \RR^{d\times d}$ be two independent orthogonal matrices with uniform distribution over the orthogonal group. Then, the product $\bU\bW$ and $\bV \bW$ are independent.
     \end{fact}
     \begin{fact}\label{fact:more-randomness2}
         Let $\bU, \bW \in \RR^{d\times d}$ be two independent orthogonal matrices with uniform distribution over the orthogonal group. Then, the product $\bU\bW$ is also uniform over the orthogonal orthogonal group, and $\bU\bW$ is independent of both $\bU$ and $\bW.$ 
     \end{fact}

     Thus, without loss of generality, we can assume that the distributions of the eigenvectors of both $\bA$ and $\bE$ are orthogonally invariant. Moreover, we can post-multiply the matrix $\bX$ by a uniform at random orthogonal matrix $\bW$, then 
     $\tilde \bX = \tilde \bA + \tilde \bE$ where $\tilde \bA = \bA \bW$ and $\tilde \bE = \bE \bW$. The left singular vectors of $\tilde \bX$, $\tilde \bE$ and $\tilde \bA$ match the eigenvectors of $\bX$, $\bE$ and $\bA$, respectively. Moreover, thanks to Facts~\ref{fact:more-randomness} and \ref{fact:more-randomness2}, the singular vectors of $\tilde \bA$ and $\tilde \bE$ are independent, and so are the left and right singular vectors of $\bE$. Notice that this implies that $\tilde \bA$ and $\tilde \bE$ satisfy Assumptions~\ref{ass:amodel1} and~\ref{ass:amodel2}. Thus, the result follows after invoking~\cref{prop:zero-correlation-as}. 
     
     \end{proof}
\subsection{Asymmetric Case}\label{sec:background-as}     

    \begin{theorem}[Theorems 2.8 and 2.9 in \cite{benaych2012singular}]\label{thm:raj-as}
      Suppose that Assumption~\ref{ass:amodel1} and \ref{ass:amodel2} hold. For each $n$, let $\sigma_{i}^{(n)} = \sigma_{i}(\bX_{n})$ and let $\widehat \bu$ and $\widehat \bv$ be unit-norm left and right singular vectors associated with $\sigma_{i}^{(n)}$. Then, we have that as $n$ grows to infinity,
      \begin{itemize}
              \item[] (\textbf{Eigenvalue limit}) Fix $j \in \{1, \dots, r\}$, then,
      \begin{equation}
        \label{eq:ass_singularvalue}
 \sigma_{i}^{(n)} \as \begin{cases} \rho_i := D_{\mu_{\bE}}^{-1}(1/\theta_{i}^2) & \text{if }\theta_{j}^2 > 1/D_{\mu_{\bE}}(b^{+}),  \\
 b  & \text{otherwise.}
 \end{cases}
      \end{equation}
              \item[] (\textbf{Eigenvector correlation limit})  Fix $j \in \{1, \dots, r\}$ such that $\theta_{j}^2 \geq 1/D_{\mu_{\bE}}(b^{+})$, let $\cW_{j} = \ker(\theta_{j}\bI_{n} - \bA_{n}\bA_{n}^\top)$ be the subspace of singula vectors associated with $\theta_{i}$, and let $\widehat u$ be a unit norm singular-vector of $\bX_n$ associated with $\sigma_j$. Then, 
      \begin{equation}
        \label{eq:ass_singularvector}
        \|\cP_{\cW_{j}}(\widehat u)\|^{2} \xrightarrow{a.s.} -\frac{2\varphi_\mu(\rho_j; 1)}{\theta_j^2D'_{\mu_{\bE}}(\rho_j)}  = - 2\frac{D_{\mu_{\bE}}(\rho_j)\varphi_\mu(x;1)}{D'_{\mu_{\bE}}(\rho_j)}.
      \end{equation}
      Furthermore, let $\cT_j = \bigoplus_{i \neq j}^r \ker(\theta_i \bI_n - \bA_n\bA_n^\top)$ be the eigenspace corresponding to the first $r$ eigenvectors that are not associated with $\theta_i$. Then,
      \begin{equation}\label{eq:decor} \|\cP_{\cT_{j}}(\widehat u)\|^{2} \as 0. \end{equation}
    \end{itemize}
    \end{theorem}

Similarly to the one in the symmetric case, the following proposition unveils the behavior of the D-transform and its derivative at the edge of the spectrum. It relies crucially on the square-root decay assumption. 
    \begin{proposition}[Proposition 2.11 in \cite{benaych2012singular}]
        \label{prop:nice-edge-as}
        Suppose that~\cref{ass:amodel2} holds. Then, we have that $$D_{\mu_{\bE}}(b^+)^{-1/2} > 0 \quad \text{and} \quad \varphi'_{\mu_{\bE}}(b^+;1) =\varphi'_{\mu_{\bE}}(b^+;q) = -\infty.$$
    \end{proposition}
Finally, we close with the asymptotic decorrelation of singular values associated with different singular values.
     \begin{proposition}[Proposition 2 in \cite{donoho2023screenot}]\label{prop:zero-correlation-as}
      Suppose that Assumption~\ref{ass:amodel1} and~\ref{ass:amodel2} hold. Fix $j \in \{1, \dots, r\}$ and $i \neq j.$
      Let $\bu_{j}$ and $\bv_j$ be unit-norm left and right singular value of $\bA_{n}$ associated with $\theta_j$ and let $\widehat \bu_{i}$ and $\widehat \bv_i$ be the left and right unit-norm eigenvectors of $\bX_{n}$ associated with $\sigma_i(\bX)$. Then,
      \begin{equation}
        \label{eq:decorrelation}
        \dotp{\widehat u_{i}, u_{j}} \xrightarrow{a.s.} 0 \qquad \text{and} \quad \dotp{\widehat v_{i}, v_{j}} \xrightarrow{a.s.} 0 \qquad \text{almost surely.}
      \end{equation}
    \end{proposition}   
    We highlighted that although the statement of the proposition in \cite{donoho2023screenot} states the conclusion for the product of the two inner products, a quick examination of the proof reveals that the conclusion holds for both terms separately.

\end{document}